\makeatletter
\providecommand*{\input@path}{}
\g@addto@macro\input@path{{include/}{../include/}{.}}
\makeatother

  \documentclass[12pt]{article} %
  \usepackage{akazachk}
  \title{$\mathcal{V}$-Polyhedral Disjunctive Cuts}
  \author{Egon Balas \and Aleksandr M. Kazachkov}
  \date{February 18, 2024}

\providecommand{\bibDirName}{.}  %
\providecommand{\mainDir}{.}  %

\newcommand{\biblio}{\bibliographystyle{plainnat}\bibliography{\mainDir/\bibDirName/akazachk}}  %

\newcommand{\rlaptext}[1]{\rlap{\text{#1}}}
\DeclareMathAlphabet{\mathpzc}{OT1}{pzc}{m}{it}

\newcommand{\xangle}{-20}
\newcommand{\yangle}{45}
\newcommand{\zangle}{90}

\newcommand{\xlength}{1}
\newcommand{\ylength}{1}
\newcommand{\zlength}{1}

\pgfmathsetmacro{\xx}{\xlength*cos(\xangle)}
\pgfmathsetmacro{\xy}{\xlength*sin(\xangle)}
\pgfmathsetmacro{\yx}{\ylength*cos(\yangle)}
\pgfmathsetmacro{\yy}{\ylength*sin(\yangle)}
\pgfmathsetmacro{\zx}{\zlength*cos(\zangle)}
\pgfmathsetmacro{\zy}{\zlength*sin(\zangle)}

\newcommand{\xX}{\xx}
\newcommand{\xY}{\xy}
\newcommand{\yX}{\yx}
\newcommand{\yY}{\yy}
\newcommand{\zX}{\zx}
\newcommand{\zY}{\zy}

\RequirePackage{algorithmicx}
\algnewcommand\algorithmicinput{\textbf{Input:}}
\algnewcommand\Input{\item[\algorithmicinput]}
\algnewcommand\algorithmicinit{\textbf{Initialize:}}
\algnewcommand\Initialize{\item[\algorithmicinit]}
\algnewcommand\algorithmicoutput{\textbf{Output:}}
\algnewcommand\Output{\item[\algorithmicoutput]}
\algnewcommand\algorithmicpreprocessing{\textbf{Preprocessing:}}
\algnewcommand\Preprocessing{\item[\algorithmicpreprocessing]}

\makeatletter
\let\OldStatex\Statex
\renewcommand{\Statex}[1][3]{%
  \setlength\@tempdima{\algorithmicindent}%
  {{\parfillskip0pt\par}}\OldStatex\hskip\dimexpr#1\@tempdima\relax}
\makeatother

\newcommand{\suchthat}{: }

\newcommand*{\defeq}{\mathrel{\vcenter{\baselineskip0.5ex \lineskiplimit0pt
                     \hbox{\footnotesize.}\hbox{\footnotesize.}}}%
                     =}

\newcommand*{\defas}{\defeq}

\DeclareMathOperator*{\argmin}{arg\,min}

\DeclareMathOperator*{\conv}{conv}
\DeclareMathOperator*{\cone}{cone}

\DeclareMathOperator*{\cl}{cl}

\newcommand{\abs}[1]{\lvert#1\rvert}

\newcommand{\card}[1]{\lvert #1 \rvert}

\newcommand{\norm}[1]{\lVert #1 \rVert}

\newcommand{\ceil}[1]{\left\lceil #1 \right\rceil}
\newcommand{\floor}[1]{\left\lfloor #1 \right\rfloor}

\newcommand{\field}[1]{\mathbbm{#1}}
\newcommand{\reals}{\field{R}}
\newcommand{\integers}{\field{Z}}

\newcommand{\R}{\reals}

\newcommand{\Z}{\integers}

\newcommand{\eps}{\epsilon}

\newcommand{\T}{\mathsf{\scriptscriptstyle T}} %

\usepackage{relsize}
\def\ifmonospace{\ifdim\fontdimen3\font=0pt }

\def\Cpp{%
\ifmonospace%
    C++%
\else%
    C\nolinebreak[4]\raisebox{0.5ex}{\tiny\textbf{++}}%
\fi%
\spacefactor1000 }

\newcommand{\solver}[1]{#1}

\newcommand{\Gurobi}{\solver{Gurobi}}

\newcommand\NB{\mathcal{N}}

\newcommand{\pointset}{\mathcal{P}}
\newcommand{\rayset}{\mathcal{R}}

\newcommand{\opt}[1]{\bar{#1}}
\newcommand{\lpopt}{\opt x}

\newcommand{\intvars}{\mathcal{I}}
\newcommand{\PI}{P_{I}}

\newtheorem{theorem}{Theorem}
\newtheorem{corollary}[theorem]{Corollary}

\newtheorem{proposition}[theorem]{Proposition}

\newtheorem{definition}[theorem]{Definition}

\newtheorem*{theorem*}{Theorem}
\newtheorem*{corollary*}{Corollary}
\newtheorem*{lemma*}{Lemma}
\newtheorem*{proposition*}{Proposition}
\newtheorem*{claim*}{Claim}
\newtheorem*{example*}{Example}
\newtheorem*{remark*}{Remark}
\newtheorem*{observation*}{Observation}
\newtheorem*{definition*}{Definition}
\newtheorem*{conjecture*}{Conjecture}

\newcommand{\nbspace}[1]{\tilde{#1}}
\newcommand{\epsortho}{\epsilon_{\text{orth}}}
\def\PRLPeq{PRLP\textsuperscript{=}}

\renewcommand{\PI}{\ref{PI}}
\newcommand{\Pt}{\ref{Pt}}

\newcommand{\numgapinst}{332}
\newcommand{\numVgoodinst}{116}
\newcommand{\numBinaryInstances}{65}
\newcommand{\numtimeinst}{282}

\usepackage{makecell}

\begin{document}
\maketitle

\begin{abstract}
We introduce \emph{$\mathcal{V}$-polyhedral disjunctive cuts} (VPCs) for generating valid inequalities from general disjunctions. Cuts are critical to integer programming solvers, but the benefit from many families is only realized when the cuts are applied recursively, causing numerical instability and ``tailing off'' of cut strength after several rounds. To mitigate these difficulties, the VPC framework offers a practical method for generating strong cuts without resorting to recursion. The framework starts with a disjunction whose terms partition the feasible region into smaller subproblems, then obtains a collection of points and rays from the disjunctive terms, from which we build a linear program whose feasible solutions correspond to valid disjunctive cuts. Though a na\"{i}ve implementation would result in an exponentially-sized optimization problem, we show how to efficiently construct this linear program, such that it is much smaller than the one from the alternative higher-dimensional cut-generating linear program. This enables us to test strong multiterm disjunctions that arise from the leaf nodes of a partial branch-and-bound tree. In addition to proving useful theoretical properties of the cuts, we evaluate their performance computationally through an implementation in the open-source COIN-OR framework. In the results, VPCs from a strong disjunction significantly improve the gap closed compared to existing cuts in solvers, and they also decrease some instances' solving time when used with branch and bound.
\end{abstract}

\section{Introduction}
\label{sec:intro}

This paper presents a new framework for generating disjunctive \emph{cutting planes}, or \emph{cuts}, in which a large number of strong cuts can be generated efficiently and nonrecursively.
We are motivated by a crucial drawback of many existing cut techniques, in their reliance on recursion to reach strong cuts, i.e., by computing cuts from previously-derived ones.
This can result in numerical issues (e.g., due to compounding inaccuracies) and a ``tailing off'' of the strength of the cuts in later rounds~\cite{BelFis75,CooDasFukGoy09,BalFisZan10,ZanFisBal11,MilRalSte18_exploring-numerics-branch-and-cut}.
Our primary computational innovation to circumvent recursion is an efficient use of the \emph{$\mathcal{V}$-polyhedral} perspective, which has been avoided in the past because, used na\"{i}vely, it yields intractable representations of instances.
The framework we introduce overcomes this hurdle and facilitates a cut generation scheme formulated in the original dimension of the problem, as opposed to a commonly-used higher-dimensional representation.
This enables us to test large multiterm disjunctions arising from partial branch-and-bound trees, resulting in cuts that are strong (compared to cuts currently deployed in solvers) 
and have the potential to reduce solving time.

One way to derive a general-purpose cut is from a \emph{disjunction}, which we define precisely later in this section.
\citet{Balas79} introduced the prevailing paradigm for disjunctive cuts, via a \emph{cut-generating linear program} (CGLP) with additional variables beyond those defining the cut.
The CGLP is too expensive to work with, despite being polynomially-sized (in the size of the original instance and disjunction).
As a result, solvers only use this idea for the simplest---split---disjunctions~\cite{BalCerCor93,BalCerCor96}.
Successfully applying such \emph{lift-and-project cuts} (L\&PCs) for splits hinges on 
the ability to compute the cuts in the space of the original linear program, which is possible due to a correspondence of basic solutions of the extended formulation to those in the lower-dimensional space~\cite{BalPer03,BalBon09,Bonami12}.
However, although such a correspondence exists for split disjunctions, it does not necessarily exist for more general ones~\cite{AndCorLi05,BalKis16}.
Hence, efficiently producing cuts from stronger disjunctions calls for a different perspective.

Our starting point for developing such a computationally-efficient procedure is a \emph{polar} representation of polyhedra: rather than the inequality description~\!\!---which is how an instance is usually provided and what underpins the CGLP---every polyhedron can also be equivalently represented by its $\mathcal{V}$-polyhedral description, i.e., through a collection of points and rays.
Cuts can be generated by inputting these points and rays into a {``point-ray'' linear program} with the same number of variables as that of the original problem.
The drawback of the point-ray approach is that the number of rows of the linear program will typically be exponentially large in the size of an inequality formulation, and these rows are necessary in the sense that dropping them may result in invalid cuts.
For this reason, prior work that has adopted the $\mathcal{V}$-polyhedral perspective resorts to row generation to guarantee validity~\cite{PerBal01,LouPoiSal15}.

We show how to avoid the expensive row generation step via a properly chosen compact collection of points and rays that we prove suffices to produce valid cuts, albeit a subset of the entire pool of possible disjunctive cuts.
We prove conditions under which the cuts obtained from our $\mathcal{V}$-polyhedral relaxation define facets for the disjunctive hull;
for example, for a split disjunction, every cut we generate is facet-defining when there is no primal degeneracy.

Most existing cut-generation procedures are applied to shallow disjunctions, such as split disjunctions.
Typically, strong cuts are attained by recursive applications, deriving cuts from other cuts, rather than only from the inequalities of the original system,
and several shallow disjunctions are considered in parallel.
In contrast, the efficiency of our technique enables testing larger disjunctions (up to 64 terms in our experiments).
Our disjunctions are derived from generating a partial branch-and-bound tree for each instance.
Theoretically, disjunctive cuts from such trees can solve the original integer programming problem without recursion~\cite{Jorg08_disj-cut-finiteness,CheKucSen11_finite-disjunctive-programming-characterizations}.

Another difference in our approach from usual cut generation paradigms is our choice of objective functions for the point-ray linear program.
Instead of producing cuts that are maximally violated by an optimal solution to the linear programming relaxation,
we aim for a diverse set of cuts with wider coverage of the disjunction: we explicitly aim for inequalities that are supporting on different points and rays.
In the process, we constructively prove that cuts from our relaxed point-ray collection can provide the same objective value as from using the complete (exponential-sized) point-ray collection.
Further, we develop and refine an objective selection scheme, which is a dynamic cut selection procedure, in that the next target direction depends on what cuts have been generated before.
Our theory also aims to avoid infeasibility, unboundedness, and duplicate cuts while solving the point-ray linear program, as these situations impose computational burden without generating new inequalities.

Supporting the theoretical results, our computational experiments indicate that applying our cuts, which we call \emph{$\mathcal{V}$-polyhedral (disjunctive) cuts} (VPCs), significantly improves the percent integrality gap closed over the baseline of \emph{Gomory mixed-integer cuts} (GMICs)~\cite{Gomory58_fractional-cut} and the default cuts for the leading commercial solver \Gurobi{}~\cite{Gurobi10.0.3}.
To complement our empirical evaluation of the strength of VPCs through the integrality gap they close, we also test the effect of VPCs on the solving time of \Gurobi{}.
The branch-and-bound results do not show an improvement in solver performance when using VPCs on average,
but we observe a benefit for a number of instances.
Our investigation leaves open the question of how to identify instances, hyperparameter settings, or solver modifications to take advantage of stronger disjunctive cuts in practice.

\paragraph{Paper organization.}
We describe properties of our cut generation scheme in Section~\ref{sec:PRLP}, such as conditions under which we obtain facet-defining inequalities for the disjunctive hull and theoretical results that benefit our implementation.
Section~\ref{sec:disj} specifies the disjunctions used in our experiments.
Section~\ref{sec:obj-theory} provides theory for the objective directions we consider.
The computational results in Section~\ref{sec:computation} indicate that a vital and challenging question that remains outstanding is efficiently selecting multiterm disjunctions that yield good cuts.

\paragraph{Notation.}
Let \ref{P} denote a polyhedron described by a set of $m$ inequalities:
  \begin{equation}
    P \defas \{x \in \R^n \suchthat A x \ge b\} \rlap{.}
    \tag{\ensuremath{P}}\label{P}
  \end{equation}
Let $\PI$ be the integer-feasible region:
  \begin{equation}
    P_I \defas \{x \in \ref{P} \suchthat x_j \in \Z \text{ for all $j \in \intvars$}\} \rlap{,}
    \tag{\ensuremath{P_I}}\label{PI}
  \end{equation}
where $\intvars \subseteq [n] \defeq \{1,\ldots,n\}$ is the index set of the integer-restricted variables.
We assume that \ref{P} is full dimensional and {pointed},
all data is rational,
and all variable bounds are subsumed by $Ax \ge b$.%
\footnote{The full-dimensionality assumption is made for ease of exposition and the subsequent results generally apply, with minor modifications, when this assumption is relaxed.}
For a given $c \in \R^n$, our goal is to solve the \emph{mixed-integer program}
  \begin{equation}
    \min_{x \in \ref{PI}} c^\T x \rlap{.} \tag{IP}\label{IP}
  \end{equation}
We start by solving the \emph{linear programming relaxation} of \eqref{IP}, obtained by removing the integrality restrictions on the variables:
  \begin{equation}
    \min_{x \in \ref{P}} c^\T x \rlap{.} \tag{LP}\label{LP}
  \end{equation}
This yields an optimal solution, $\lpopt$, which we assume does not belong to \ref{PI}.
To proceed, integer programming solvers next tighten the relaxation \ref{P} by the addition of inequalities that are valid for \ref{PI} but remove some of \ref{P}.
One way these cuts can be generated is via a valid \emph{disjunction}, which creates a partition such that \ref{PI} is contained in the union of the disjunctive terms.
Concretely, a disjunction takes the form
  \begin{equation}
    \bigvee_{t \in \mathcal{T}} \{x \in \R^n \suchthat D^t x \ge D^t_0\} \rlap{.}
    \label{eq:disj}
  \end{equation}
where $\mathcal{T}$ is a finite index set.
We denote \emph{disjunctive term} $t \in \mathcal{T}$ by
  \begin{equation}
    P^t \defas \{x \in \ref{P} \suchthat D^t x \ge D^t_0\} \rlap{.}
    \tag{\ensuremath{P^t}}\label{Pt}
  \end{equation}
Let $P_D \defas \cl\conv(\cup_{t \in \mathcal{T}} \Pt)$ be the \emph{disjunctive hull}, the \emph{closed convex hull} of the elements of \ref{P} satisfying the disjunction.
We assume the disjunction satisfies $\PI \subseteq P_D$ and $\lpopt \notin P_D$.

\section{Point-ray linear program}
\label{sec:PRLP}

Let $\pointset$ and $\rayset$ denote sets of points and rays in $\R^n$.%
\footnote{%
In this paper, \emph{rays} refer to either (1) a direction of unboundedness, or (2) an ``extended edge'' of a polyhedron, i.e., an edge extended (to infinity) from one of its endpoints.
}
Define the 
  \emph{point-ray linear program} (PRLP), 
taking as an input the point-ray collection $(\pointset, \rayset)$ and an objective direction $w \in \R^n$, as follows:
  \begin{align*}
    \min_{\alpha,\beta} \quad &\alpha^\T w \\ \label{PRLP} \tag{PRLP}
    &\begin{aligned}
      \alpha^\T p &\ge \beta &\quad& \rlaptext{for all $p \in \pointset$} \\
      \alpha^\T r &\ge 0 &\quad& \rlaptext{for all $r \in \rayset$.}
    \end{aligned}
  \end{align*}
Every feasible solution $(\alpha,\beta)$ to \eqref{PRLP} is an inequality $\alpha^\T x \ge \beta$; these are what we refer to as VPCs.
Define the \emph{point-ray hull} as $\conv(\pointset) + \cone(\rayset)$.

We immediately address the nature of the cuts obtainable from \eqref{PRLP}.
Theorem~\ref{thm:PRLP-facets} shows that the extreme ray solutions to \eqref{PRLP} correspond to facet-defining inequalities for the point-ray hull (which, in accordance with convention, we simply refer to as \emph{facets} of the point-ray hull).

\begin{theorem}
\label{thm:PRLP-facets}
  The inequality $\alpha^\T x \ge \beta$ is valid for $\conv(\pointset) + \cone(\rayset)$
  if and only if
  $(\alpha,\beta)$ is a feasible solution to \eqref{PRLP}. %
  The inequality defines a facet of the point-ray hull if and only if the solution $(\alpha,\beta)$ is an extreme ray of \eqref{PRLP}.
\end{theorem}
\begin{proof}
  Every point $\hat{x} \in \conv(\pointset) + \cone(\rayset)$ can be expressed as a convex combination of the elements in $(\pointset,\rayset)$.
  For any inequality $\alpha^\T x \ge \beta$ satisfied by all of these points and rays, it follows that $\alpha^\T \hat{x} \ge \beta$.
  Every extreme ray $(\alpha,\beta)$ to \eqref{PRLP} has the additional property that $n$ affinely independent points and rays from the point-ray collection satisfy $\alpha^\T x = \beta$, which means that the inequality defines a facet of $\conv(\pointset) + \cone(\rayset)$.
  The reverse directions follow, respectively, from the definitions of valid and facet-defining inequality.
\end{proof}

There are two primary challenges encountered with the PRLP.
First, the point-ray collection needs to be chosen judiciously, to balance the strength of the obtainable VPCs with the time required to generate them, while ensuring validity of the cuts.
Second, it is critically important to intelligently select the objective directions $w$ for \eqref{PRLP}.
The next sections are devoted to addressing these questions.
We first detail properties of the feasible region of \eqref{PRLP}.
In particular, we 
describe how we normalize the PRLP,
prove necessary and sufficient conditions for VPCs to be valid cuts for $\PI$,
and discuss the conditions under which VPCs are facet-defining for the disjunctive hull.

\subsection{Normalization of the PRLP}
\label{sec:normalization}

As presented, the linear program \eqref{PRLP} does not have a finite objective value for any nonzero objective direction $w$:
if $(\alpha,\beta)$ is feasible, then so is $(\lambda \alpha, \lambda \beta)$ for any nonnegative $\lambda \in \R$.
In theory, each ray $(\alpha,\beta)$ yields a valid cut; 
however, in practice, the ray returned by a solver might not be extreme, which may correspond to a weak cut.
This is the same well-documented issue that arises with the CGLP.
To resolve this, a \emph{normalization} is applied, truncating the cone defining the feasible region of \eqref{PRLP}.
The choice of normalization can have a significant effect on the cuts that are ultimately generated~\cite{FisLodTra11,LodTanVie23_disj-cuts-mixed-integer-conic}.

One solution is to constrain the magnitude of the cut coefficients, e.g., via (a linearization of) $\sum_{i = 1}^n \abs{\alpha_i} \le 1$~\cite{BalCerCor93}.
This normalization has the undesirable characteristic that it may add extreme points to the feasible region of \eqref{PRLP} that do not correspond to facets of the disjunctive hull.

A second normalization proposed by \citet{BalPer02}---also used by \citet{PerBal01} and \citet*{LouPoiSal15}---is to add the constraint 
  $\alpha^\T (p - \lpopt) = 1$,
for some $p \in P_D$.
This idea has been studied by \citet{Serra20_reformulating} and \citet{ConWol19}, when the role of the objective and normalization is swapped.
This guarantees that \eqref{PRLP} is always bounded and has other nice properties,
but depends on a good choice of $p$.

We take a third approach for normalizing \eqref{PRLP} of fixing $\beta$ to a constant value.
As observed by \citet{BalMar13}, it suffices to consider only three values for $\beta$: $\{-1,0,+1\}$.
This might indicate that one would need to solve three linear programs to generate cuts with such a normalization, as discussed by \citet*{LouPoiSal15}.
We avoid this issue by formulating \eqref{PRLP} in the \emph{nonbasic space} relative to $\lpopt$.
In this space, the \eqref{LP} optimal solution $\lpopt$ is the origin.
As a result, if we are looking for inequalities that are violated by $\lpopt$, it suffices to fix $\beta = 1$.
Another advantage of working in the nonbasic space is that \eqref{PRLP} may be much sparser than if it were formulated in the structural space of variables.
This is because the number of nonzero components in every row roughly corresponds to the number of simplex pivots from $\lpopt$ to the point or ray for that row.
When we normalize with $\beta = 1$, every basic feasible solution to \eqref{PRLP} corresponds to an inequality violated by $\lpopt$.

One limitation imposed by our normalization is that the cuts we generate are only those that remove $\lpopt$.
It is clearly necessary to remove this point in order to make progress beyond the linear programming relaxation towards an integer-feasible solution, but it is a ``transient'' point, in the sense that it is no longer truly important after the first cut that removes it
(see, e.g., the discussion by \citet{CadLem13}).
This limitation is easily avoided by generating cuts with $\beta = 0$ or $-1$, but effectively implementing this idea requires an independent in-depth investigation that we leave to future research.

\subsection{Proper point-ray collections: ensuring valid inequalities}

From Theorem~\ref{thm:PRLP-facets}, we have that feasible solutions to \eqref{PRLP} are valid for the point-ray hull.
We further need to ensure that VPCs are valid for \PI.
The point-ray collections with this guarantee will be called \emph{proper}, adapting the definition from \citet{KazNadBalMar20} within the \emph{generalized intersection cut} paradigm~\cite{BalMar13}, which uses a linear program analogous to the PRLP but derives points and rays via a different procedure.
We remove a dependency on $\lpopt$ present in the prior definition, which enables our framework to produce arbitrary disjunctive inequalities that do not necessarily cut $\lpopt$ and can be stronger than \emph{any} intersection cut obtainable from the same disjunction~\cite{BalKis16}.
This modification also simplifies the characterization of proper point-ray collections.

\begin{definition} \label{defn:globally-proper}
  The point-ray collection $(\pointset, \rayset)$
  is called \emph{proper} if $\alpha^\T x \ge \beta$ is valid for $\PI$ whenever
  $(\alpha,\beta)$ is feasible to \eqref{PRLP}.
\end{definition}

As a direct corollary to Theorem~\ref{thm:PRLP-facets}, we obtain a necessary and sufficient condition for a point-ray collection to be proper.

\begin{corollary}
\label{cor:PRLP-globallyproper}
  A point-ray collection $(\pointset,\rayset)$ is proper if and only if $\PI \subseteq \conv(\pointset) + \cone(\rayset)$.
\end{corollary}
\begin{proof}
  Sufficiency of the condition follows from Theorem~\ref{thm:PRLP-facets}: every feasible solution to \eqref{PRLP} is valid for $\conv(\pointset) + \cone(\rayset)$ and hence for $\PI$.
  Necessity is similarly evident as, otherwise, there exists an extreme ray of \eqref{PRLP} that cuts a point in $\PI$.
\end{proof}

Of course, we do not work with the integer hull directly.
The intermediary is the disjunctive hull.
The next corollary is a key result for the development of a practical procedure working with points and rays.
It states that as long as the point-ray hull forms a \emph{$\mathcal{V}$-polyhedral relaxation} of $P_D$, then the point-ray collection is proper.

\begin{corollary}
\label{cor:relaxedPRcollection}
  A point-ray collection $(\pointset,\rayset)$ is proper if $\Pt \subseteq \conv(\pointset) + \cone(\rayset)$ for all $t \in \mathcal{T}$, or, equivalently, if $P_D \subseteq \conv(\pointset) + \cone(\rayset)$.
\end{corollary}

\subsection{Simple VPCs from simple cone relaxations}
\label{sec:simple-vpcs}

To generate valid VPCs from the PRLP, we                                 first have to compute a proper point-ray collection.
A natural starting point is a $\mathcal{V}$-polyhedral description of the disjunctive hull.
For $t \in \mathcal{T}$, let $\pointset^{t*}$ and $\rayset^{t*}$ be the complete set of extreme points and rays of $\Pt$,
and define
  $\pointset^{*} \defas \cup_{t \in \mathcal{T}} \pointset^{t*}$ 
and 
  $\rayset^{*} \defas \cup_{t \in \mathcal{T}} \rayset^{t*}$.
As a corollary of Theorem~\ref{thm:PRLP-facets}, we know not only that $(\pointset^{*},\rayset^{*})$ is proper but also that basic feasible solutions of the normalized \eqref{PRLP}
from this point-ray collection correspond to 
facet-defining inequalities for the disjunctive hull, $P_D$.

\begin{corollary}
\label{cor:completePRLP}
  The point-ray collection $(\pointset^{*},\rayset^{*})$ is proper.
  Every extreme ray solution $(\alpha,\beta)$ of the associated \eqref{PRLP} corresponds to a facet $\alpha^\T x \ge \beta$ of $P_D$.
  Conversely, for every facet $\alpha^\T x \ge \beta$ of $P_D$, %
  the solution $(\alpha,\beta)$ to \eqref{PRLP} is feasible and extreme. %
\end{corollary}

The complete $\mathcal{V}$-polyhedral description of each disjunctive term is, however, impractical, as the number of points and rays can grow exponentially large in $m$ and $n$.
A reasonable alternative would be to use some small subset of the point-ray collection $(\pointset^{*}, \rayset^{*})$.
It is not difficult to see that this could lead to invalid cuts, as we show by example in 
the extended manuscript~\cite{BalKaz22+_vpc-arxiv}.
It is for this reason that \citet{PerBal01} and \citet*{LouPoiSal15} employ constraint generation to obtain valid cuts.

We take an alternative approach, based on Corollary~\ref{cor:relaxedPRcollection}: instead of pursuing all facet-defining inequalities for the disjunctive hull, we use a relaxation of $P_D$ with a compact $\mathcal{V}$-polyhedral description.
One convenient relaxation for each disjunctive term is the \emph{basis cone} $C^t$ at an optimal solution $p^t$ to 
  $\min_{x} \{c^\T x \suchthat x \in \Pt\}$, 
formed from $p^t$ and a \emph{cobasis} $\NB(p^t)$ associated with $p^t$.
This cone is defined as the intersection of the $n$ inequalities corresponding to the nonbasic variables, which are indexed by $\NB(p^t)$, and it has a compact $\mathcal{V}$-polyhedral description, with only one extreme point ($p^t$) and $n$ extreme rays, which we denote $r^{t1},\ldots,r^{tn}$.
We refer to the union of these points and rays across all terms as the \emph{simple point-ray collection} $(\pointset^0,\rayset^0)$, where
$\pointset^0 \defeq \cup_{t \in \mathcal{T}} \{p^t\}$ and
$\rayset^0 \defeq \cup_{t \in \mathcal{T}} \{r^{tj}\}_{j \in [n]}$, 
and we will use the shorthand 
	$P_{D}^0 \defeq \conv(\pointset^0) + \cone(\rayset^0) 
	$ 
to denote the corresponding \emph{simple point-ray hull}.
The specific PRLP used in our experiments is given below as \eqref{PRLP0}.
Recall that we formulate the problem in the nonbasic space.
To make this explicit, for any $x \in \R^n$, we use $\nbspace{x}$ for its representation in the nonbasic space.
  \begin{align*}
    \min_{\nbspace{\alpha}} \quad &\nbspace{\alpha}^\T \nbspace{w} \\ \label{PRLP0} \tag{PRLP$^0$}
    &\begin{aligned}
      &\nbspace{\alpha}^\T \nbspace{p}^t \ge 1 &\quad& \text{for all } \rlaptext{$t \in \mathcal{T}$} \\
      &\nbspace{\alpha}^\T \nbspace{r}^{tj} \ge 0 &\quad& \text{for all } \rlaptext{$t \in \mathcal{T}$, $j \in [n]$}
    \end{aligned}
  \end{align*}
The cuts from \eqref{PRLP0} will be called \emph{simple} VPCs.
We state their validity as Proposition~\ref{prop:VPC1}.

\begin{proposition}
\label{prop:VPC1}
  The simple point-ray collection $(\pointset^0,\rayset^0)$ is proper.
\end{proposition}
\begin{proof}
  The result follows from Corollary~\ref{cor:relaxedPRcollection} and the fact that $P_D \subseteq P_D^0$.
\end{proof}

It is useful at this point to make a theoretical comparison between the PRLP and extended formulation used for the CGLP.
Any valid disjunctive cut can be theoretically found with a CGLP having size polynomial in the dimensions of the original problem and the number of disjunctive terms.
In particular, the CGLP with a fixed right-hand side $\beta$ has $(n+1) \cdot \card{\mathcal{T}}$ constraints and $n + (m + m_t) \cdot \card{\mathcal{T}}$ variables (where $m_t$ denotes the number of rows of $D^t$).
It is similarly possible to produce all valid disjunctive cuts with the $\mathcal{V}$-polyhedral framework, when $(\pointset,\rayset) = (\pointset^{*},\rayset^{*})$, as we proved in Corollary~\ref{cor:completePRLP}.
The disadvantage is that the PRLP may now have exponentially many constraints, which is why we turned to the relaxation-based generator.
This yields \eqref{PRLP0}, whose feasible region is defined by only $(n+1) \cdot \card{\mathcal{T}}$ constraints (the same as for the CGLP) and only $n$ variables.
Moreover, working in the original dimension of the problem confers significant computational efficiency over the CGLP framework, as we discuss in Section~\ref{sec:disj} via the experiments of \citet{PerBal01}, but the tradeoff is being able to generate merely a subset of all valid disjunctive cuts.
Nevertheless, the subsequent theoretical results of this section and the later computational results indicate that this subset already captures strong disjunctive cuts.

\subsection{Simple VPCs corresponding to facets of the disjunctive hull}
\label{sec:vpc-facets}

\begin{figure}
  \centering
    \begin{subfigure}{.5\textwidth}%
    \centering
      \begin{tikzpicture}[line join=round,line cap=round,x={(\xX cm, \xY cm)},y={(\yX cm, \yY cm)},z={(\zX cm,\zY cm)},>=stealth,scale=2.9]
        \coordinate (barx) at (1/2,1/2,1);
        
        \coordinate (p11) at (0,1/4,0);
        \coordinate (p12) at (0,1/2,1/2);
        \coordinate (p13) at (0,3/4,0);
        
        \coordinate (p21) at (1,0,0);
        \coordinate (p22) at (1,1/4,1/2);
        \coordinate (p23) at (1,3/4,1/2);
        \coordinate (p24) at (1,1,0);
        
        \draw [polyhedron_edge] (barx) -- (p12) -- (p11) -- (p21) -- (p22) -- cycle;
        \draw [polyhedron_edge,dashed] (p12) -- (p13) -- (p11);
        \draw [polyhedron_edge] (p21) -- (p22) -- (p23) -- (p24) -- cycle;
        \draw [polyhedron_edge] (barx) -- (p22) -- (p23) -- cycle;
        \draw [polyhedron_edge,dashed] (p13) -- (p24);        
        
        \node [draw, point, fill=black, label={[label distance=0pt]above: $\lpopt$}] at (barx) {};
      \end{tikzpicture}
    \end{subfigure}%
    \begin{subfigure}{.5\textwidth}%
    \centering
      \begin{tikzpicture}[line join=round,line cap=round,x={(\xX cm, \xY cm)},y={(\yX cm, \yY cm)},z={(\zX cm,\zY cm)},>=stealth,scale=2.9]
        \fill [polyhedron_fill] (p11) -- (p12) -- (p13) -- cycle;
        \fill [polyhedron_fill] (p21) -- (p22) -- (p23) -- (p24) -- cycle;
        
        \draw [polyhedron_edge,gray] (p12) -- (p11);
        \draw [polyhedron_edge,gray] (p21) -- (p22);
        \draw [polyhedron_edge,gray] (p13) -- (p11);
        \draw [polyhedron_edge,gray] (p21) -- (p22) -- (p23) -- (p24) -- cycle;
        
        \draw [ray,line width=2pt] (p12) -- (0,1/8,-1/4);
        \draw [ray,line width=2pt] (p12) -- (0,7/8,-1/4);
        \draw [ray,line width=2pt] (p22) -- (1,-1/16,-1/8);
        
        \draw [cut_edge,dashed] (p13) -- (p12) -- (barx) -- (p23) -- (p24) -- cycle;
        
        \draw [ray,line width=2pt] (p22) -- (1,7/8,1/2);
        
        \node [draw, point, fill=black, label={[label distance=0pt]above: $\lpopt$}] at (barx) {};
        \node [draw, point, fill=black, label={[label distance=-6pt]above left: $C^1$}] at (p12) {};
        \node [draw, point, fill=black, label={[label distance=0pt]below left: $C^2$}] at (p22) {};
                
        \node [label={[label distance=-13pt]0:  $P^1$}] at (0,11/20,1/8) {};
        \node [label={[label distance=-13pt]0:  $P^2$}] at (1,1/2,1/4) {};
      \end{tikzpicture}
    \end{subfigure}
  \caption{%
		This example shows how the simple point-ray collection could limit the set of obtainable inequalities valid for the disjunctive hull.
		The right panel shows that the dashed inequality (that is valid for $P_I$) would violate a ray of the cone $C^2$.
		The cones are shown as two dimensional for ease of illustration.
  }
  \label{fig:non-facet}
  \end{figure}
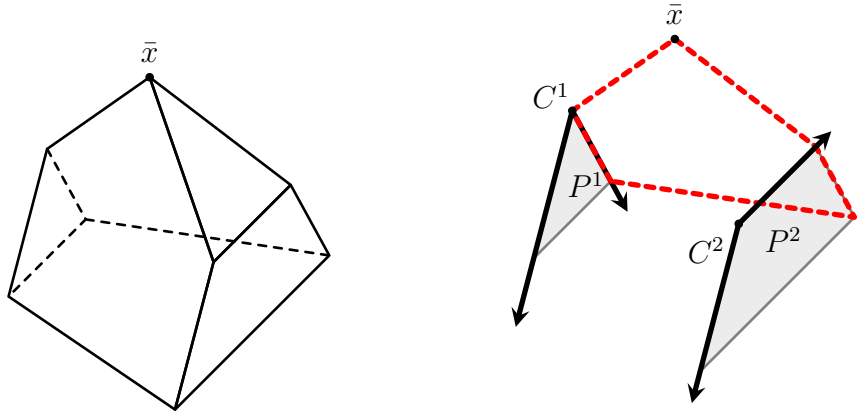

In this section, we compare the disjunctive hull $P_D$ with the relaxation $P_D^0$.
Although $P_D^0$ is a drastic relaxation of $P_D$, in that it is defined by a small fraction of the inequalities defining $P_D$, we show that $P_D^0$ can tightly approximate $P_D$ in the region of interest to us.
We distinguish between two types of facets of $P_D^0$: those that are facets of $P_D$, and those that are not.
Clearly not all facets of $P_D$ are captured by $P_D^0$, which we illustrate in Figure~\ref{fig:non-facet}.
The results below concern precisely which facets of $P_D$ exist as facets of $P_D^0$.

Consider a facet-defining inequality of $P_D^0$ that is violated by $\lpopt$
and a corresponding basic feasible solution to \eqref{PRLP0},
after adding slack variables on the constraints.
Since the cut coefficient variables $\nbspace{\alpha}$ are unrestricted in sign,
the nonbasic variables in this solution are all slack variables,
and the corresponding tight constraints of \eqref{PRLP0} identify
$n$ affinely independent points and rays of $(\pointset^0,\rayset^0)$
that lie on the inequality and certify that it defines a facet of $P_D^0$.
We call these the ``\emph{nonbasic points}'' and ``\emph{nonbasic rays}'' with respect to the basic feasible solution.

By construction of the simple point-ray collection,
each of the nonbasic points also belongs to $P_D$.
Now assume that, for each of the nonbasic rays $r$,
there is a point $p \in \pointset^0$ tight for the cut and a $\lambda > 0$ such that $p_r \defeq p + \lambda r \in P_D$,
which is thus also on the cut.
Then the nonbasic points and the points $p_r$ for the nonbasic rays certify that the inequality also defines a facet of $P_D$.

This assumption is, unfortunately, not generally satisfied:
for some nonbasic rays, we might not find a point in $P_D$ corresponding to that ray.
One difficulty, illustrated in the example shown in Figure~\ref{fig:non-facet2}, is that although a ray originates from some particular point and term (when building the point-ray collection), it is ultimately added to \emph{all} points in $\pointset^0$ to calculate $P_D^0$.
      In this example, a four-term disjunction is taken.
      The first panel shows \ref{P};
      the second panel shows $P_D$, as well as $\Pt$ and the cones $C^t$ for each $t \in [4]$, with $C^1$ labeled;%
      \footnote{For simplicity in making the example, these $C^t$ are not basis cones, due to degeneracy.}
      and the third panel shows $P_D^0$
      and the points and rays (as wavy arrows) that are tight for each of the facets of this point-ray hull.
      Consider the ray $r$ of $C^1$ labeled in two places in the third panel; it is added to the collection from one disjunctive term but impacts two facets.
      The effect of $r$ is it causes $F_1$ (a facet-defining inequality for $P_D$) to be invalid for the point-ray hull, and it adds a facet ($F_2$) to the point-ray hull that is redundant for \ref{P}.

Consider a basic feasible solution of \eqref{PRLP0} and associated cut $\alpha^\T x \ge \beta$.
If $r^{tj}$ is a nonbasic ray in this solution, %
but there is no term $t' \in \mathcal{T}$ for which $r^{tj}$ is an extreme ray of $C^{t'}$ and $\alpha^\T p^{t'} = \beta$,
then we call this a \emph{stray ray} (for the facet).
Note that $t' \ne t$ is possible if a ray is extreme for multiple terms.
With this, we say that a facet of $P_D^0$ is \emph{standard} if there is a corresponding basis of \eqref{PRLP0} with no stray rays,
i.e., when
$\nbspace{\alpha}^\T \nbspace{r}^{tj} = 0$ (with the row's slack variable nonbasic)
implies that $\nbspace{\alpha}^\T \nbspace{p}^{t'} = 1$, for some $t' \in \mathcal{T}$ for which $r^{tj}$ is an extreme ray of $C^{t'}$.
Thus, facet $F_2$ in Figure~\ref{fig:non-facet2} is not standard, due to stray ray $r$ from $C^1$.
We apply this concept in Theorem~\ref{thm:simple-facets} to state a sufficient condition for a facet of the simple point-ray hull to be a facet of the disjunctive hull. %

  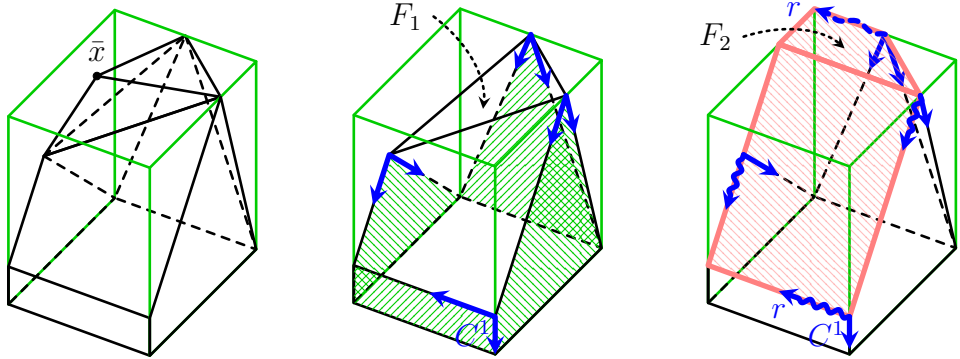
\begin{figure}
    \centering
    \begin{subfigure}{.33\linewidth}
    \centering
      \begin{tikzpicture}[line join=round,line cap=round,x={(\xX cm, \xY cm)},y={(\yX cm, \yY cm)},z={(\zX cm,\zY cm)},>=stealth,scale=2] 
        \coordinate (p11) at (0,0,-1/4);  
        \coordinate (p12) at (0,0,0);
        \coordinate (p13) at (0,1/3,1/2);
        \coordinate (p14) at (0,1,-1/4);

        \coordinate (p21) at (1,0,-1/4);  
        \coordinate (p22) at (1,0,0);
        \coordinate (p23) at (1,2/3,1);
        \coordinate (p24) at (1,1,-1/4);
        
        \coordinate (p0) at (1/4,1/2,1);
        \coordinate (p3) at (1/2, 1, 1);
        
        \draw [fill opacity=0.5,fill=none, color=green!80!black, line width=1pt] (0,0,-1/4) -- (0,1,-1/4) -- (0,1,1) -- (0,0,1) -- cycle;
                        
        \draw [polyhedron_edge] (p11) -- (p12) -- (p13) -- (p23);
        \draw [polyhedron_edge, dashed] (p13) -- (p14) -- (p11);
        \draw [polyhedron_edge] (p11) -- (p21) -- (p22) -- (p12) -- cycle;
        \draw [polyhedron_edge] (p13) -- (p23);
        \draw [polyhedron_edge, dashed] (p13) -- (p3);
        \draw [polyhedron_edge, dashed] (p3) -- (p24) -- (p14) -- cycle;
        \draw [polyhedron_edge] (p0) -- (p13);
        \draw [polyhedron_edge] (p0) -- (p3) -- (p23) -- cycle;
        
        \draw [fill opacity=0.5,fill=none, color=green!80!black, line width=1pt] (1,0,-1/4) -- (1,1,-1/4) -- (1,1,1) -- (1,0,1) -- cycle;
        \draw [polyhedron_edge] (p21) -- (p22) -- (p23) -- (p24) -- cycle;
        \draw [fill opacity=0.5,fill=none, color=green!80!black, line width=1pt] (0,0,1) -- (1,0,1) -- cycle;
        \draw [fill opacity=0.5,fill=none, color=green!80!black, line width=1pt] (0,1,1) -- (1,1,1) -- cycle;
        \node [draw, point, fill=black, label={[label distance=0pt]90: $\lpopt$}] at (p0) {};
      \end{tikzpicture}
    \end{subfigure}%
    \begin{subfigure}{.34\linewidth}
    \centering
      \begin{tikzpicture}[line join=round,line cap=round,x={(\xX cm, \xY cm)},y={(\yX cm, \yY cm)},z={(\zX cm,\zY cm)},>=stealth,scale=2]
        \draw [fill opacity=0.5,fill=none, color=green!80!black, line width=1pt] (0,0,-1/4) -- (0,1,-1/4) -- (0,1,1) -- (0,0,1) -- cycle;
        \draw [fill opacity=0.5,fill=none, color=green!80!black, line width=1pt] (0,1,1) -- (1,1,1) -- cycle;
                        \path [pattern=north east lines, pattern color=green!80!black] (p14) -- (p3) -- (p24) -- cycle;
        
        \path [pattern=north west lines, pattern color=green!80!black] (p11) -- (p12) -- (p13) -- (p14) -- cycle;
        \draw [polyhedron_edge] (p11) -- (p12) -- (p13);
        \draw [polyhedron_edge, dashed] (p13) -- (p14) -- (p11);
        \draw [polyhedron_edge] (p11) -- (p21) -- (p22) -- (p12) -- cycle;
        \draw [polyhedron_edge] (p13) -- (p3) -- (p23) -- cycle;
        \draw [polyhedron_edge, dashed] (p3) -- (p24) -- (p14) -- cycle;
        
        \draw [fill opacity=0.5,fill=none, color=green!80!black, line width=1pt] (1,0,-1/4) -- (1,1,-1/4) -- (1,1,1) -- (1,0,1) -- cycle;
                \draw [line width=1pt, pattern=north west lines, pattern color=green!80!black] (p21) -- (p22) -- (p23) -- (p24) -- cycle;
                        \path [pattern=north east lines, pattern color=green!80!black] (p11) -- (p12) -- (p22) -- (p21) -- cycle;
        \draw [ray, line width=2pt, blue,<->] (1,0,-1/4) -- (p22) node [pos=1,below left=-5pt] {$C^1$} -- (1/2,0,0); %
        \draw [ray, line width=2pt, blue,<->] (0,2/3,1/8) -- (p13) -- (0,1/6,1/4);
        \draw [ray,blue,line width=2pt,<->] (1,9/12,11/16) -- (p23) -- (1,1/2,3/4);
        \draw [ray,blue,line width=2pt,<->] (5/8,1,11/16) -- (p3) -- (3/8,1,11/16);
        \draw [fill opacity=0.5,fill=none, color=green!80!black, line width=1pt] (0,0,1) -- (1,0,1) -- cycle;
         \path
          (0,3/4,9/8) edge [->, shorten >= 0.25cm, line width=1pt, black, dotted, bend left=30] node [pos=0,left] {$F_1$} (1/3,3/4,1/2);
      \end{tikzpicture}
    \end{subfigure}%
    \begin{subfigure}{.33\linewidth}
    \centering
      \begin{tikzpicture}[line join=round,line cap=round,x={(\xX cm, \xY cm)},y={(\yX cm, \yY cm)},z={(\zX cm,\zY cm)},>=stealth,scale=2]
        \draw [fill opacity=0.5,fill=none, color=green!80!black, line width=1pt] (0,0,-1/4) -- (0,1,-1/4) -- (0,1,1) -- (0,0,1) -- cycle;
        \draw [fill opacity=0.5,fill=none, color=green!80!black, line width=1pt] (0,1,1) -- (1,1,1) -- cycle;
        
        \draw [polyhedron_edge] (p11) -- (p12) -- (p13);
        \draw [polyhedron_edge, dashed] (p13) -- (p14) -- (p11);
        \draw [polyhedron_edge] (p11) -- (p21) -- (p22) -- (p12) -- cycle;
        \draw [line width=1pt, dashed] (p3) -- (p24) -- (p14) -- cycle;

        \draw [fill opacity=0.5,fill=none, color=green!80!black, line width=1pt] (1,0,-1/4) -- (1,1,-1/4) -- (1,1,1) -- (1,0,1) -- cycle;
        \draw [polyhedron_edge] (p21) -- (p22) -- (p23) -- (p24) -- cycle;
        \draw [ray,blue,line width=2pt,<->] (5/8,1,11/16) -- (p3) -- (3/8,1,11/16);
        \draw [fill opacity=0.5,fill=none, color=red!50, line width=2pt,pattern=north west lines, pattern color=red!50] (p3) -- (p23) -- ($(p23) - (1,0,0)$) -- (0,1,1) -- cycle;
        \draw [fill opacity=0.5,fill=none, color=red!50, line width=2pt,pattern=north west lines, pattern color=red!50] (p23) -- ($(p23) - (1,0,0)$) -- (p12) -- (p22) -- cycle;
        \draw [%
             ->,
            decorate,%
            decoration={snake,amplitude=.4mm,segment length=2mm,post length=2mm},
            blue,line width=2pt] %
          (p22) -- (1/2,0,0) node [pos=1,below] {$r$};
        \draw [ray,blue,line width=2pt] (p22) -- (1,0,-1/4)  node [pos=0,below left=-4.5pt] {$C^1$} ;
        \draw [%
             <-,
            decorate,%
            decoration={snake,amplitude=.4mm,segment length=2mm,pre length=2mm},
            blue,line width=2pt] %
          (0,1/6,1/4) -- (p13) edge [->] (0,2/3,1/8);
        \draw [%
             ->,
            decorate,%
            decoration={snake,amplitude=.4mm,segment length=2mm,post length=2mm},
            blue,line width=2pt] %
          (p23) -- (1,1/2,3/4);
        \draw [ray,line width=2pt,blue] (p23) -- (1,9/12,11/16);
        \draw [%
             ->,
            decorate,%
            decoration={snake,amplitude=.4mm,segment length=5mm,post length=1mm},
            blue,line width=2pt,dashed] %
          (p3) -- (0,1,1) node [pos=1,left] {$r$};
        \draw [fill opacity=0.5,fill=none, color=green!80!black, line width=1pt] (0,0,1) -- (1,0,1) -- cycle;
         \path
          (-1/8,1/2,9/8) edge [->, shorten >= 0.25cm, line width=1pt, black, dotted, bend left=30] node [pos=0,left] {$F_2$} (1/2,3/4,1);
      \end{tikzpicture}
    \end{subfigure}
    \caption{
	  Rays in the point-ray collection impact the set of cuts that can be generated.
      The dashed wavy line in the third panel corresponds to a ray $r$ that is added to $\rayset^0$ from one term of the disjunction, but affects the point-ray hull when it originates from a point from a different term.
    }
    \label{fig:non-facet2}
  \end{figure}

\begin{theorem}
\label{thm:simple-facets}
	Suppose the basis defining $p^t$ is unique for each $t \in \mathcal{T}$.
	If a facet of $P_D^0$ is standard,
	then it is a facet of $P_D$ that cuts $\lpopt$.
\end{theorem}
\begin{proof}
	Given a standard facet of $P_D^0$ and a corresponding basic feasible solution of \eqref{PRLP0} with no stray rays,
	we construct $n$ affinely independent points from $P_D$ that lie on the facet
	based on the nonbasic points and rays in this solution.
	For each nonbasic ray $r \in \rayset^0$, there is a term $t \in \mathcal{T}$ such that $r$ is an extreme ray of $C^t$ and the point $p^t$ lies on the facet.
	Since the basis defining $p^t$ is not primal degenerate, for small enough $\lambda > 0$, the point $p_r \defeq p + \lambda r$ is in $\Pt$ (and thus in $P_D$), as otherwise the ray would be cut by a hyperplane of $\Pt$ that is tight at $p^t$ but not defining $C^t$.
	The nonbasic points---all belonging to $P_D$---and the points $p_r$ provide the desired $n$ affinely independent points.
\end{proof}

The requirement that the facet of $P_D$ has to cut $\lpopt$ is due to the normalization of \eqref{PRLP0} by $\nbspace{\beta}=1$.

The case of a split disjunction deserves special attention given its importance in prior work, especially in the context of L\&PCs.
Although a facet of $P_D$ is not necessarily a simple VPC (whereas it exists as an L\&PC),
we can conclude the converse using our much more compact formulation and Theorem~\ref{thm:simple-facets}, i.e., that all simple VPCs are facet-defining for the disjunctive hull.

\begin{theorem}
\label{thm:split-facets}
	Suppose that \eqref{eq:disj} is a split disjunction and the bases defining $p^1$ and $p^2$ are unique.
	Then every facet of $P_D^0$ that is tight on both $p^1$ and $p^2$ is a facet of $P_D$.
	Moreover, every facet of $P_D$ that is tight on both $p^1$ and $p^2$ and cuts $\lpopt$ exists as a facet of $P_D^0$.
\end{theorem}
\begin{proof}
  For the first statement, our assumptions imply that every facet of $P_D^0$ that is tight on both $p^1$ and $p^2$ is standard, because then all points of $\pointset$ have zero slack, so there can be no stray rays.
  The second statement follows from convexity, in that a facet of $P_D$ that is tight on $p^t$ will be valid for $C^t$, $t \in \{1,2\}$, as otherwise that facet cuts a ray of $C^t$ and the corresponding point from $P_D$ in the ray's relative interior.
\end{proof}

Theorem~\ref{thm:split-facets} can be extended to more general disjunctions for facets of $P_D^0$ that are tight on all points $p^t$, $t \in \mathcal{T}$.

\section{Choosing strong disjunctions}
\label{sec:disj}

The previous section introduces our computationally-viable way to generate disjunctive cuts through the PRLP via a $\mathcal{V}$-polyhedral relaxation of the given disjunctive hull.
To complete the setup of the constraints of \eqref{PRLP}, it remains to specify which class of disjunctions we will use in our experiments.

Much of the focus in the recent literature on cutting planes has been on generating stronger cuts from shallow disjunctions, i.e., those that utilize relatively few (one or two) integer variables, based on indices of integer variables that are fractional in $\lpopt$.
This includes disjunctions based on the complements of triangles, quadrilaterals, and crosses.
To compensate for the weakness of the disjunction, typically cuts are generated from several disjunctions from the same class in each round.
These families of disjunctions do contain cuts that outperform Gomory cuts, but one needs to carefully select which disjunctions to test within each family,
and the computational cost associated with finding the stronger cuts often outweighs their benefit (see, e.g., \cite{DasGunVie14}).

We circumvent some of these difficulties by expending additional effort to generate one strong ``deep'' disjunction per instance for cut generation.
Specifically, the disjunctions we define come from the set of leaf nodes of a partial branch-and-bound tree.%
\footnote{Besides intuitive appeal, our choice of disjunctions is further bolstered by the results in 
Appendix~G of the extended manuscript~\cite{BalKaz22+_vpc-arxiv},
indicating that VPCs from a multitude of split or cross disjunctions are weaker.}
A partial branch-and-bound tree has the advantage of conveying additional information about \eqref{IP} 
(with respect to an alternative disjunction with the same number of terms obtained without branching).
For instance, the partial tree may be asymmetric and include pruning by infeasibility, by integrality, and by bound.
We demonstrate this by example in Figure~\ref{fig:bbtree-example}, contrasting a cross disjunction generated from two integer variables $x_1$ and $x_2$ to a four-term disjunction that might be obtained using the branch-and-bound process.

\begin{figure}
\floatsetup{valign=t, heightadjust=all}
\ffigbox{%
\begin{subfloatrow}
	\ffigbox{%
    \centering
      \begin{tikzpicture}[line join=round,line cap=round,>=stealth,scale=5]
		\coordinate (root) at (5/10, 5/5);
		\coordinate (x0) at (2/10,4/5);
		\coordinate (x1) at (8/10,4/5);
		\coordinate (x00) at (0/10,3/5);
		\coordinate (x01) at (4/10,3/5);
		\coordinate (x10) at (6/10,3/5);
		\coordinate (x11) at (10/10,3/5);
        
        \coordinate (translate) at (0,0);
        \coordinate (root2) at ($(root)+(translate)$);
		\coordinate (y0) at ($(x0)+(translate)$);
		\coordinate (y00) at ($(x00)+(translate)$);
		\coordinate (y01) at ($(x01)+(translate)$);
		\coordinate (y1) at ($(x1)+(translate)$);
		\coordinate (y10) at ($(x10)+(translate)$);
		\coordinate (y11) at ($(x11)+(translate)$);
		\coordinate (y110) at ($(9/10,2/5)+(translate)$);
		\coordinate (y111) at ($(11/10,2/5)+(translate)$);
		\coordinate (y1100) at ($(8.5/10,1/5)+(translate)$);
		\coordinate (y1101) at ($(9.5/10,1/5)+(translate)$);
		
        \draw [line width=1pt] (root) -- node [pos=0.4,left=0.25em] {\small $x_1 \le 0$ } (x0);
        \draw [line width=1pt] (root) -- node [pos=0.4,right=0.25em] {\small $x_1 \ge 1$ } (x1);
        \draw [line width=1pt] (x0) -- node [pos=0.4,left=0.1em] {\small \llap{$x_2 \le 0$} } (x00);
        \draw [line width=1pt] (x0) -- node [pos=0.2,right=0.1em] {\small $x_2 \ge 1$ } (x01);
        \draw [line width=1pt] (x1) -- node [pos=0.6,left=0.1em] {\small $x_2 \le 0$ } (x10);
        \draw [line width=1pt] (x1) -- node [pos=0.4,right=0.1em] {\small \rlap{$x_2 \ge 1$} } (x11);
		
        \node [draw, point, fill=black, label={[label distance=0pt]90: }] at (root2) {};
        \node [draw, point, fill=black, label={[label distance=-3pt]90: }] at (x0) {};
        \node [draw, point, fill=black, label={[label distance=-3pt]45: }] at (x1) {};
        \node [draw, point, fill=black, label={[label distance=0pt]-90: }] at (x00) {};  
        \node [draw, point, fill=black, label={[label distance=0pt]-90: }] at (x01) {};
        \node [draw, point, fill=black, label={[label distance=0pt]-90: }] at (x10) {};
        \node [draw, point, fill=black, label={[label distance=0pt]-90: }] at (x11) {};
		\end{tikzpicture}}{\caption{The disjunction is all possible assignments of $x_1$ and $x_2$.}}%
	\ffigbox{%
    \centering
      \begin{tikzpicture}[line join=round,line cap=round,>=stealth,scale=5]
        \draw [line width=1pt] (root2) -- node [pos=0.4,left=0.25em] {\small $x_1 \le 0$ } (y0);
        \draw [line width=1pt] (root2) -- node [pos=0.4,right=0.25em] {\small $x_1 \ge 1$ } (y1);
        \draw [line width=1pt] (y0) -- node [pos=0.4,left=0.1em] {\small \llap{$x_2 \le 0$} } (y00);
        \draw [line width=1pt] (y0) -- node [pos=0.2,right=0.1em] {\small $x_2 \ge 1$ } (y01);
        \draw [line width=1pt] (y1) -- node [pos=0.6,left=0.1em] {\small $x_3 \le 0$ } (y10);
        \draw [line width=1pt] (y1) -- node [pos=0.4,right=0.1em] {\small $x_3 \ge 1$ } (y11);
        \draw [line width=1pt] (y11) -- node [pos=0.4,left=0em] {\small $x_2 \le 0$ } (y110);
        \draw [line width=1pt] (y11) -- node [pos=0.4,right=0em] {\small \rlap{$x_2 \ge 1$} } (y111);
        \draw [line width=1pt] (y110) -- node [pos=0.5,left=0em] {\small $x_4 \le 0$ } (y1100);
        \draw [line width=1pt] (y110) -- node [pos=0.5,right=0em] {\small $x_4 \ge 1$ } (y1101);
		
        \node [draw, point, fill=black, label={[label distance=0pt]90: }] at (root) {};
        \node [draw, point, fill=black, label={[label distance=0pt]90: }] at (y0) {};
        \node [draw, point, fill=black, label={[label distance=0pt]0: }] at (y00) {};
        \node [draw, point, fill=black, label={[label distance=0pt]0: }] at (y01) {};
        \node [draw, point, fill=black, label={[label distance=0pt]90: }] at (y1) {};
        \node [draw, point, fill=red, label={[label distance=0pt]0: }] at (y10) {};
        \node [draw, cross=3pt, fill=red] at (y10) {};
        \node [draw, point, fill=black, label={[label distance=0pt]90: }] at (y11) {};
        \node [draw, point, fill=black, label={[label distance=0pt]90: }] at (y110) {};
        \node [draw, point, fill=red, label={[label distance=0pt]0: }] at (y111) {};
        \node [draw, cross=3pt, fill=red] at (y111) {};
        \node [draw, point, fill=black, label={[label distance=0pt]90: }] at (y1100) {};
        \node [draw, point, fill=black, label={[label distance=0pt]90: }] at (y1101) {};
      \end{tikzpicture}}{\caption{A different variable is branched on, resulting in some pruned nodes and two stronger disjunctive terms.}}
\end{subfloatrow}}
	{\caption{Two four-term disjunctions (the leaf nodes of the trees).}
	\label{fig:bbtree-example}}
\end{figure}
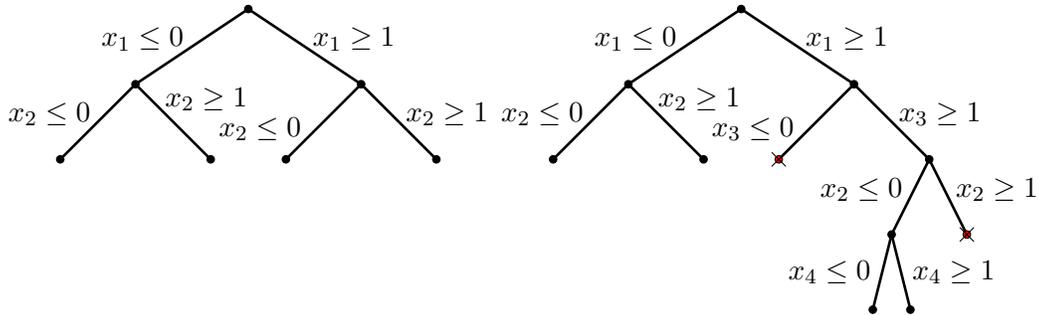

Disjunctive cuts coming from partial branch-and-bound trees have previously been proposed and tested in several contexts, indicating the potential impact of a more efficient method for obtaining such cuts.
The majority of these previous experiments rely on variants of the higher-dimensional CGLP, e.g., in the context of the cutting plane tree algorithm~\cite{CheKucSen11_finite-disjunctive-programming-characterizations,CheKucSen12_computational-study-cutting-plane-tree} and in stochastic mixed-integer programming applications~\cite{SenShe06_decomposition-with-branch-and-cut-for-two-stage-stochastic-mip,NtaTan08,YuaSen09,Ntaimo10,GadKucSen14,QiSen17}.
A famous example is the computational experience of solving the \emph{seymour} problem with the aid of L\&PCs, as documented by \citet*{FerPatSch01}. %

Methods resembling the VPC approach include the aforementioned cuts from the $\mathcal{V}$-polyhedral perspective based on row generation~\cite{PerBal01,LouPoiSal15}.
One takeaway from the paper by \citet{PerBal01} is a comparison to generating disjunctive cuts via solving the higher-dimensional CGLP: %
the authors conclude that solving the CGLP becomes relatively much slower as the number of terms of the disjunction grows, a result that would be more pronounced if the expense of row generation were avoided.
In other closely related work, \citet*{ChvCooEsp13} experiment with partial branch-and-bound trees as part of local or target cut algorithms~\cite{AppBixChvCoo95,BucLieOsw08},
which are dual to the separation schemes in this paper and that of \cite{PerBal01,LouPoiSal15}.
Most recently, \citet{CheLue23_sparse-multiterm-disj-cuts} generate objective cuts from multiterm disjunctions in the context of two-stage stochastic programs.

We refer the interested reader to the dissertation of \citet{Kazachkov18} for a more in-depth treatment of related literature.

\section{Choosing appropriate objectives}
\label{sec:obj-theory}

Having set up the constraints of the PRLP, we now analyze the theoretical strength of VPCs to drive our choices of objective functions $\nbspace{w}$ for \eqref{PRLP0}.
Choosing these carefully is critical to the success of any VPC algorithm, as the objectives directly determine the nature of the VPCs obtained.
Aside from the type of cuts obtained, it is also important to make the cut-generating process efficient.
We say that every time we solve \eqref{PRLP} and a new cut is not generated, a \emph{failure} occurs.
In an early implementation of VPCs, failures were frequent (over 85\% of the objectives tried).

One reason for failure is that \eqref{PRLP0} may be infeasible for a particular point-ray collection.
This occurs, for instance, if $\lpopt$ belongs to $P_D^0$ (feasible solutions of \eqref{PRLP0} are inequalities that separate $\lpopt$, the origin in the nonbasic space).
Figure~\ref{fig:non-facet2} actually illustrates such a situation. %
The next proposition gives a sufficient condition for the feasibility of \eqref{PRLP0} that we use in our implementation, though it can be extended to apply to \eqref{PRLP} more generally.
Let $\ref{p*}$ denote a point from $\pointset^0$ with minimum objective value, i.e.,
  \begin{equation}
    p^* \in \argmin_{t} \{c^\T p^t \suchthat t \in \mathcal{T}\} \rlap{.}
    \tag{\ensuremath{p^*}}\label{p*}
  \end{equation}

\begin{proposition} 
\label{prop:PRLP-feasible}
  If $\nbspace{c}^\T \nbspace{\ref{p*}} > 0$, then \eqref{PRLP0} is feasible.
\end{proposition}
\begin{proof}
  It suffices to observe that $\nbspace{\alpha} = \nbspace{c} / \nbspace{c}^\T \nbspace{\ref{p*}}$ is a feasible solution to \eqref{PRLP0}, corresponding to the objective cut $\tilde{\alpha}^\T \tilde{x} \ge 1$, equivalently $c^\T x \ge c^\T \ref{p*}$.
  For $t \in \mathcal{T}$, $\nbspace{\alpha}^\T \nbspace{p}^t = \nbspace{c}^\T \nbspace{p}^t / \nbspace{c}^\T \nbspace{\ref{p*}} \ge 1$ by definition of $\ref{p*}$.
  For $r \in \rayset^0$, we need to show that ${c}^\T {r} \ge 0$.	
  This follows from the fact that $p^t \in \argmin_{x} \{c^\T x \suchthat x \in \Pt\}$,
  which implies that $p^t$ is also optimal when minimizing over $C^t$.
  For any ray $r \in C^t$, for all $\lambda > 0$, the point $p^t + \lambda r$ belongs to $C^t$.
  Hence, $c^\T p^t + \lambda c^\T r \ge c^\T p^t$.
\end{proof}
\noindent
This condition is satisfied, for example, whenever each disjunctive term's LP relaxation has an optimal objective value worse than that of the original LP.

The two other primary reasons for failures we observed were that, for a given objective direction $\nbspace{w}$: 
  (1) \eqref{PRLP0} was feasible but unbounded,
  or (2)~\eqref{PRLP0} had a finite optimal solution but the corresponding cut was a duplicate of a previously generated cut.
We will not be able to completely eliminate failures, but in the remainder of this section, we work towards choosing objectives that help reduce the failure rate.
At the same time, we will target generating strong VPCs, while mostly ignoring the potential effect of the cuts within branch and bound;
this latter goal is poorly understood and hence difficult to target directly.

The first candidate for an objective direction is to maximize the violation by $\lpopt$, as is done in the case of L\&PCs.
Unfortunately, in the nonbasic space and with $\beta = 1$, $\lpopt$ is simply the origin and all cuts have violation equal to $1$.
As proxies, we use two other objectives.
First, we try the all-ones objective, $\nbspace{w} = e$.
The interpretation is that we seek an inequality that puts equal weight on cutting each of the rays of the basis cone at $\lpopt$.
Second, we add to \ref{P} a round of GMICs, separated from $\lpopt$, and calculate a new optimal solution ${x}'$;
we then use $\nbspace{w} = \nbspace{{x}}'$, which finds a cut maximizing the violation with respect to ${x}'$.

Finding cuts that maximize violation with respect to points not in $P_D$ is a paradigm that may place too much emphasis on cutting away irrelevant parts of the relaxation.
The alternative is to find inequalities that minimize the slack with respect to points that do belong to $P_D$.
Within this latter perspective, we are able to utilize whatever structural information we possess about the disjunctive hull.
We will now discuss precisely what kind of information can be inferred from our $\mathcal{V}$-polyhedral relaxations of the disjunctive hull.

The next result states that, despite the vastly relaxed simple point-ray collection, the optimal value over the disjunctive hull can be obtained by optimizing over the points in $\pointset^0$.

\begin{proposition}
\label{prop:disjLB}
  The point $\ref{p*} \in \argmin_p \{c^\T p \suchthat p \in \pointset^0\}$ is also an optimal solution to $\min_{x} \{ c^\T x \suchthat x \in P_D^0 \}$ and to $\min_{x} \{ c^\T x \suchthat x \in P_D \}$.
  Moreover, there are $n$ facets of $P_D^0$ that can be added to \ref{P} such that $\ref{p*}$ is also an optimal solution to the resulting relaxation.
\end{proposition}
\begin{proof}
  This is a direct consequence of the fact that $p^t \in \argmin_{x} \{c^\T x \suchthat x \in \Pt\}$ for all $t \in \mathcal{T}$.
  The $n$ inequalities are simply the ones determined by the cobasis of $\ref{p*}$ from solving $\min_{x} \{ c^\T x \suchthat x \in P_D^0 \}$.
\end{proof}

We say that $c^\T \ref{p*}$ is the \emph{disjunctive lower bound} and examine whether we can achieve it via VPCs.
Note that we can always add the inequality $c^\T x \ge c^\T \ref{p*}$, but this is generally counterproductive, as such objective cuts tend to create multiple optimal solutions to the subsequent relaxation, which cause difficulties for solvers.

By Proposition~\ref{prop:disjLB}, we know that we can attain the disjunctive lower bound via $n$ facets of the point-ray hull that are tight at $\ref{p*}$.
One way to generate a cut tight at $\ref{p*}$ is to use the objective direction $\nbspace{w} = \nbspace{\ref{p*}}$.
Absent numerical issues, the optimal solution will be some $\bar{\alpha}$ such that $\nbspace{\bar{\alpha}}^\T \nbspace{\ref{p*}} = 1$.
Though this is only one cut, it can be used to find other objective directions.
We will work with a modified \eqref{PRLP0}, which we refer to as \emph{\PRLPeq{}},
in which the constraint $\nbspace{\alpha}^\T \nbspace{\ref{p*}} \ge 1$ is changed to $\nbspace{\alpha}^\T \nbspace{\ref{p*}} = 1$.
Let $\overline{\rayset}$ be the set of rays from $\rayset^0$ that are not tight for the cut $\nbspace{\bar{\alpha}}^\T \nbspace{x} \ge 1$, i.e., 
  $\overline{\rayset} \defas \{r \in \rayset^0 \suchthat \bar{\alpha}^\T r > 0\}$.

\begin{proposition}
\label{prop:PRLPeq}
  \PRLPeq{} with objective direction $\nbspace{w} = \nbspace{{r}}$, ${r} \in \overline{\rayset}$, has a finite optimal solution.
  The optimal value is strictly less than $\nbspace{\bar{\alpha}}^\T \nbspace{{r}}$ only if the resulting cut is distinct from $\nbspace{\bar{\alpha}}^\T \nbspace{x} \ge 1$.
  The optimal value is zero if and only if there exists a facet of $P_D^0$ that cuts $\lpopt$ and is tight on ${r}$.
\end{proposition}
\begin{proof}
  The fact that \PRLPeq{} is finite and bounded is a direct result of the constraint $\nbspace{\alpha}^\T \nbspace{{r}} \ge 0$ and the feasibility of \PRLPeq{}.
  The second statement is obvious.
  The last statement comes from the one-to-one correspondence between basic feasible solutions to \PRLPeq{} and facets of $P_D^0$ that cut away $\lpopt$.
\end{proof}

Propostion~\ref{prop:PRLPeq} resolves issue~{(1)} mentioned above, of having a feasible PRLP that is unbounded.
Unfortunately, we may still get failures from issue~{(2)}, meaning the optimal solution to \PRLPeq{} corresponds to a cut we previously generated.
For example, it may be the case that there exists $r \in \overline{R}$ such that, for all $\nbspace{\alpha}$ feasible to \PRLPeq{}, $\nbspace{\alpha}^\T \nbspace{r} \ge \nbspace{\bar{\alpha}}^\T \nbspace{r}$.
Using such an $r$ as the objective for \PRLPeq{} could reproduce the solution $\bar{\alpha}$.
To prevent such phenomena from excessively slowing down cut generation, we add a failure rate parameter, which detects when there is an unacceptably small percent of the objectives successfully producing new cuts, triggering early termination of the procedure.
We discuss this further in
the extended manuscript~\cite{BalKaz22+_vpc-arxiv}.

Lastly, we address a practical question: will we always attain the disjunctive lower bound $c^\T \ref{p*}$ via simple VPCs?
On one hand, Proposition~\ref{prop:disjLB} states that this bound is attainable via $n$ facets of the point-ray hull tight at $\ref{p*}$ (which may not be facets of $P_D$).
However,
the answer to this question can be no, given our algorithmic choices.
To see why, we need to understand which inequalities can be generated from \eqref{PRLP0} in our setup.
The specific modifications we have made from the general case are that we fix $\beta > 0$ and work in the nonbasic space in which $\lpopt$ is the origin.
This implies that we will never generate any facet-defining inequalities for $P_D^0$ that are satisfied by $\lpopt$.
One might initially assume that these inequalities are not necessary in order to attain the bound $c^\T \ref{p*}$.
  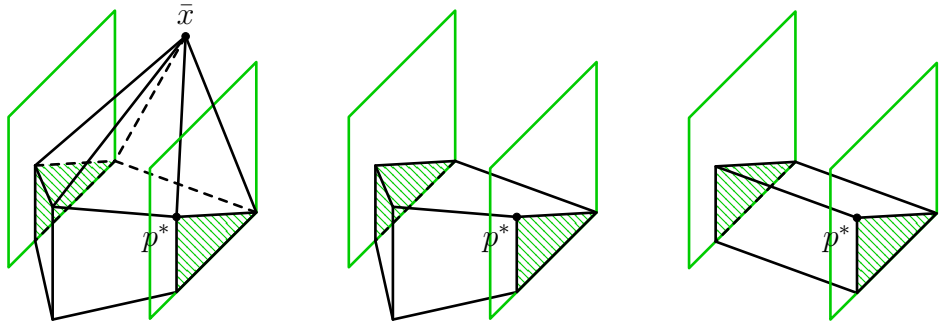
\begin{figure}
    \centering
    \begin{subfigure}{.33\linewidth}
    \centering
      \begin{tikzpicture}[line join=round,line cap=round,x={(\xX cm, \xY cm)},y={(\yX cm, \yY cm)},z={(\zX cm,\zY cm)},>=stealth,scale=2] 
        \coordinate (p11) at (0,1/4,0);  
        \coordinate (p12) at (0,1/4,1/2);
        \coordinate (p13) at (0,1,0);
        
        \coordinate (p1) at (1/2, -1/4, 0);
        \coordinate (p2) at (1/2, -1/4, 3/4);
        \coordinate (p3) at (1/2, 1, 1);

        \coordinate (p21) at (1,1/4,0);  
        \coordinate (p22) at (1,1/4,1/2);
        \coordinate (p23) at (1,1,0);
        
        \draw [fill opacity=0.5,fill=none, color=green!80!black, line width=1pt] (0,0,0) -- (0,1,0) -- (0,1,1) -- (0,0,1) -- cycle;
        
        \path [pattern=north west lines, pattern color=green!80!black] (p11) -- (p12) -- (p13) -- cycle;
        \draw [line width=1pt] (p11) -- (p12);
        \draw [line width=1pt, dashed] (p12) -- (p13) -- (p11);
        \draw [line width=1pt] (p1) -- (p2) -- (p3);
        \draw [line width=1pt] (p11) -- (p1) -- (p21);
        \draw [line width=1pt] (p12) -- (p2) -- (p22);
        \draw [line width=1pt] (p3) -- (p12);
        \draw [line width=1pt] (p3) -- (p22);
        \draw [line width=1pt, dashed] (p3) -- (p13);
        \draw [line width=1pt, dashed] (p13) -- (p23);
        
        \draw [fill opacity=0.5,fill=none, color=green!80!black, line width=1pt] (1,0,0) -- (1,1,0) -- (1,1,1) -- (1,0,1) -- cycle;
        \draw [line width=1pt] (p3) -- (p23);
        \draw [line width=1pt, pattern=north west lines, pattern color=green!80!black] (p21) -- (p22) -- (p23) -- cycle;
        
        \node [draw, point, fill=black, label={[label distance=-4pt]215: $\ref{p*}$}] at (p22) {};
        \node [draw, point, fill=black, label={[label distance=-1pt]above: $\lpopt$}] at (p3) {};
      \end{tikzpicture}
    \end{subfigure}%
    \begin{subfigure}{.33\linewidth}
    \centering
      \begin{tikzpicture}[line join=round,line cap=round,x={(\xX cm, \xY cm)},y={(\yX cm, \yY cm)},z={(\zX cm,\zY cm)},>=stealth,scale=2]
        \draw [fill opacity=0.5,fill=none, color=green!80!black, line width=1pt] (0,0,0) -- (0,1,0) -- (0,1,1) -- (0,0,1) -- cycle;
        
        \path [pattern=north west lines, pattern color=green!80!black] (p11) -- (p12) -- (p13) -- cycle;
        \draw [line width=1pt] (p11) -- (p12) -- (p13);
        \draw [line width=1pt, dashed] (p13) -- (p11);
        \draw [line width=1pt] (p1) -- (p2);
        \draw [line width=1pt] (p11) -- (p1) -- (p21);
        \draw [line width=1pt] (p12) -- (p2) -- (p22);
        \draw [line width=1pt] (p13) -- (p23);
        
        \draw [fill opacity=0.5,fill=none, color=green!80!black, line width=1pt] (1,0,0) -- (1,1,0) -- (1,1,1) -- (1,0,1) -- cycle;
        \draw [line width=1pt, pattern=north west lines, pattern color=green!80!black] (p21) -- (p22) -- (p23) -- cycle;
        
        \node [draw, point, fill=black, label={[label distance=-4pt]215: $\ref{p*}$}] at (p22) {};
      \end{tikzpicture}
    \end{subfigure}%
    \begin{subfigure}{.33\linewidth}
    \centering
      \begin{tikzpicture}[line join=round,line cap=round,x={(\xX cm, \xY cm)},y={(\yX cm, \yY cm)},z={(\zX cm,\zY cm)},>=stealth,scale=2]
        \draw [fill opacity=0.5,fill=none, color=green!80!black, line width=1pt] (0,0,0) -- (0,1,0) -- (0,1,1) -- (0,0,1) -- cycle;
        
        \path [pattern=north west lines, pattern color=green!80!black] (p11) -- (p12) -- (p13) -- cycle;
        \draw [line width=1pt] (p11) -- (p12) -- (p13);
        \draw [line width=1pt, dashed] (p13) -- (p11);
        \draw [line width=1pt] (p11) -- (p21);
        \draw [line width=1pt] (p12) -- (p22);
        \draw [line width=1pt] (p13) -- (p23);
        
        \draw [fill opacity=0.5,fill=none, color=green!80!black, line width=1pt] (1,0,0) -- (1,1,0) -- (1,1,1) -- (1,0,1) -- cycle;
        \draw [line width=1pt, pattern=north west lines, pattern color=green!80!black] (p21) -- (p22) -- (p23) -- cycle;
        \node [draw, point, fill=black, label={[label distance=-4pt]215: $\ref{p*}$}] at (p22) {};
      \end{tikzpicture}
    \end{subfigure} 
    \caption{Example that shows an inequality tight at $\protect{\ref{p*}}$ that does not cut away $\lpopt$ may be necessary for achieving the bound $c^\T \protect{\ref{p*}}$. In this example, we assume we are maximizing along the vertical axis. The first panel shows the original polytope. The second panel is the polytope after adding the only split inequality that cuts away $\lpopt$. The third panel shows the polytope after adding all the valid split cuts.}
    \label{fig:non-cut}
  \end{figure}
We dispel that notion in Figure~\ref{fig:non-cut}.
The example demonstrates that ignoring inequalities that do not cut away $\lpopt$ may lead to an optimal value (after adding cuts) that is strictly better than $c^\T \ref{p*}$.
Note that in this example, the point-ray collection uses the complete $\mathcal{V}$-polyhedral description of each $\Pt$, so, unlike the example in Figure~\ref{fig:non-facet}, this situation is not a consequence of using a relaxation of $P_D$.
This may partially explain why, in our experiments, despite our careful objective choices, we do not always obtain the bound $c^\T \ref{p*}$.

\section{Computational results}
\label{sec:computation}

Our computational experiments have two goals:
(1) assess the strength of VPCs by the percent root gap closed by one round of the cuts, which we discuss in Section~\ref{sec:vpc-gap-closed}, and 
(2) evaluate the effectiveness of VPCs when added at the root and used as part of branch and bound, covered in Section~\ref{sec:bb-time}.
Before presenting our results, we review our algorithmic choices in Section~\ref{sec:computational-setup}.

\subsection{Computational setup}
\label{sec:computational-setup}
The $\mathcal{V}$-polyhedral framework we have introduced is quite general, and there are many possibilities for implementing it.
We experiment with only a small subset of the possible parameters, 
so a more thorough tuning may improve upon our reported results.
Algorithm~\ref{alg:VPC1} summarizes our choices for generating VPCs;
more details can be found in 
the extended manuscript~\cite{BalKaz22+_vpc-arxiv}.
All experiments are performed on computers equipped with i9-13900K CPUs.
Our implementation is in \Cpp{} in the COIN-OR framework~\cite{COIN-OR} using \solver{Clp}~\cite{Clp} and \solver{Cbc}~\cite{Cbc} as the underlying linear programming and branch-and-bound solvers to generate VPCs.
\solver{Gurobi 10.0.3}~\cite{Gurobi10.0.3} is used to test the effectiveness of VPCs when embedded in branch and bound, averaged across 7 random seeds.

{
\begin{algorithm}[t]
\caption{Type 1 $\mathcal{V}$-Polyhedral Cuts}\label{alg:VPC1}
\begin{algorithmic}[1]
\algnotext{EndIf}
\algnotext{EndFor}
\algnotext{EndFunction}
\Input Polyhedron \ref{P}; objective direction $c$; disjunction $\vee_{t \in \mathcal{T}} \Pt$. %
  \For{$t \in \mathcal{T}$}
    \State $p^t \gets$ optimal solution to $\min_{x} \{c^\T x \suchthat x \in \Pt\}$.
    \State $C^t \gets$ basis cone with respect to a cobasis of $p^t$.
    \State $\pointset^t \gets \{p^t\}$ and $\rayset^t \gets$ extreme rays of $C^t$.
  \EndFor
  \State $\pointset \gets \cup_{t \in \mathcal{T}} \pointset^t$, $\rayset \gets \cup_{t \in \mathcal{T}} \rayset^t$, $\mathcal{C} \gets \emptyset$.
  \State Solve \eqref{PRLP0} with all-ones objective $\nbspace{w} = e$; add resulting cut to $\mathcal{C}$.
  \State Add GMICs to \ref{P}; let ${x}'$ be an optimal solution to the new relaxation.
  \State Solve \eqref{PRLP0} with $\nbspace{w} = \nbspace{x}'$ and add the resulting cut to $\mathcal{C}$. \label{step:VPC1:postGMIC}
  \State $\ref{p*} \gets$ point from $\argmin_{p} \{c^\T p \suchthat p \in \pointset\}$.
  \State Solve \eqref{PRLP0} with $\nbspace{w} = \nbspace{\ref{p*}}$ and add the resulting cut, $\nbspace{\bar{\alpha}}^\T \nbspace{x} \ge 1$, to $\mathcal{C}$.
  \State \PRLPeq{} $\gets$ \eqref{PRLP0} with added constraint $\nbspace{\alpha}^\T \nbspace{\ref{p*}} = 1$.
  \State $\overline{\pointset} \gets \{p \in \pointset \suchthat \nbspace{\bar{\alpha}}^\T \nbspace{p} > 1\}$, sorted in order of decreasing angle with $c$.
  \State $\overline{\rayset} \gets \{r \in \rayset \suchthat \nbspace{\bar{\alpha}}^\T \nbspace{r} > 0\}$, sorted in order of decreasing angle with $c$.
  \ForAll{${w}' \in \overline{\pointset} \cup \overline{\rayset}$} \label{step:VPC1:DB}
    \State Solve \PRLPeq{} with $\nbspace{w} = \nbspace{{w}}'$ and add the resulting cut $\nbspace{\bar{\alpha}}^\T \nbspace{x} \ge 1$ to $\mathcal{C}$. 
    \State Remove from $\overline{\pointset} \cup \overline{\rayset}$ all points and rays that are tight for $\nbspace{\bar{\alpha}}^\T \nbspace{x} \ge 1$.
    \State If number of objectives tried is two times the cut limit, then break. %
  \EndFor
  \State\Return Set $\mathcal{C}$ of generated cuts. \label{step:VPC1:return}
\end{algorithmic}
\end{algorithm}
} %

\paragraph{Instance selection.}
Our \numgapinst{} test instances come from the union of the MIPLIB~\cite{MIPLIB,MIPLIB3,MIPLIB2003,MIPLIB2010,MIPLIB2017}, CORAL~\cite{CORAL}, and NEOS
sets, restricted to those with at most 5,000 rows and columns and by other criteria detailed in 
the extended manuscript~\cite{BalKaz22+_vpc-arxiv}.
Every instance is preprocessed once by \Gurobi{}'s presolve.

\paragraph{Generation of disjunction.}
The disjunctive terms provided as input to Algorithm~\ref{alg:VPC1} are the leaf nodes of a partial branch-and-bound tree terminated
after reaching $2^\ell$, $\ell \in [6]$, leaf nodes of the partial tree, which form the disjunction that we input to Algorithm~\ref{alg:VPC1}.
The $\mathcal{V}$-polyhedral relaxation for each term is the simple point-ray collection defined in Section~\ref{sec:simple-vpcs}.
The partial branch-and-bound tree is generated by the default node, variable, and branch selection rules for \texttt{Cbc}.
Generating the partial tree can at times be expensive, and we make no claim that our enumeration technique is the best for cut generation;
other, perhaps weaker but less costly, strategies also merit consideration~\cite{SheSmi12}.

\paragraph{Evaluation within branch and bound.}
The VPCs are given to \Gurobi{} at the root as \emph{user cuts}, which allows \Gurobi{} to use its internal cut selection criteria.
All of \Gurobi{}'s default parameters are used except the following:
   \begin{itemize}
     \item Set random seed to $i \cdot 628$, for $i \in \{1,\ldots,7\}$.
     \item Set a time limit of 3600 seconds.
     \item Set maximum number of threads to 1.
     \item Disable presolve. (Each instance is already presolved during preprocessing.)
     \item Set \texttt{PreCrush} to 1. (Required when adding user cuts.) %
   \end{itemize}

\paragraph{Cut generation limits.}
We only use one round of cuts with a generation limit of one hour, all cuts are rank one with respect to \ref{P}, and VPCs are unstrengthened (in the sense of coefficient modularization).%
\footnote{Efficiently strengthening general disjunctive cuts currently poses nontrivial barriers~\cite{KazBal23_strengthening-ipco}.}
Cut generation is abandoned when \eqref{PRLP0} is infeasible or fails to solve to optimality within a minute when using no objective, i.e., just the feasibility problem.
We generate at most as many VPCs as the number of integer variables that are fractional at $\lpopt$, i.e., the same as the limit on the number of Gomory cuts.
This is in order to enhance comparability, but we have no evidence that this is a good choice.

\mbox{}\\[1ex]
In summary, we do not vary the $\mathcal{V}$-polyhedral relaxation of each term, and we do not impose limits on cut orthogonality or maximum density.
For one of our experiments, we double the cut limit and use two cut rounds, but
these parameters merit further exploration; prior work has repeatedly demonstrated that appropriately choosing such values can mean the difference between an algorithm that works in practice and one that seems to produce negative results (see, e.g., \cite{BalCerCor96,FerPatSch01} and the discussion in \citet[Chapter~3]{Karamanov06_branch-and-cut}).

\subsection{Percent root gap closed}
\label{sec:vpc-gap-closed}
Table~\ref{tab:gap-closed-summary} provides a summary of the average percent gap closed by GMICs, VPCs, and VPCs used together with GMICs,
as well as the percent gap closed by one round of cuts at the root by \Gurobi{}
and after the last round of cuts added by \Gurobi{} at the root.
The extended manuscript~\cite{BalKaz22+_vpc-arxiv}
contains the values for all the instances.

In the tables,
``G'' refers to GMICs, 
``V'' refers to VPCs,
``$\max$(G,V)'' refers to the best result (per instance) between GMICs and VPCs,
``GurF'' refers to \Gurobi{} after one round of cuts at the root,
``GurL'' refers to \Gurobi{} after the last round of cuts at the root,
and
``DB'' refers to the value of the disjunctive lower bound $c^\T \ref{p*}$ from the partial branch-and-bound tree for each instance with 64 leaf nodes (an upper bound on the gap we can close using VPCs on their own).
Unless otherwise stated, the result shown for VPCs is the best across all partial tree sizes tested for that instance.
A finer level of analysis in this regard is given in Appendix~\ref{app:results-partial-tree-sizes}.

Column~1 indicates which instances are being considered in the corresponding row:\ 
	the first pair of rows concerns the \numgapinst{} instances for which the disjunctive lower bound is strictly greater than the LP optimal value, 
	the second pair of rows pertains to the subset of \numVgoodinst{} instances for which VPCs close at least 10\% of the integrality gap with respect to GMICs,
    while the third pair of rows reports on the subset of \numBinaryInstances{} pure binary instances.
The first row for each set gives the average for the percent gap closed across the instances.
The second row for each set shows the number of ``wins'',
where wins
for columns ``DB'', ``V'', and ``V+G'' are relative to column ``G'';
wins for ``V+GurF'' are counted with respect to column ``GurF'';
and wins for ``V+GurL'' are with respect to column ``GurL''.
An instance counts as a win when at least $10^{-3}$ percent more integrality gap is closed compared to the appropriate reference column.

Column~2 gives the number of instances in each set.
Next is given the percent gap closed by
    GMICs when they are added to the LP relaxation (column~4);
    the disjunctive lower bound from the partial tree with 64 leaf nodes (column~5;
    VPCs (column~6);
    the maximum per instance between GMICs and VPCs (column~7);
    GMICs and VPCs used together (column~8).
Columns~9 and 10 show the percent gap closed by 
    \Gurobi{} cuts from one round at the root, first without and then with VPCs added as user cuts.
Columns~11 and 12 show the same, but after the last round of cuts at the root.
The Gurobi-related columns use the average percent gap closed by \Gurobi{} across the 7 random seeds tested.

{
\sisetup{
    table-alignment-mode = format,
    table-number-alignment = center,
    table-format = 2.2,
}
\begin{table}
\renewcommand{\tabcolsep}{4pt}
\centering
\caption{
            Summary statistics for percent gap closed by VPCs.
            The wins row reports how many instances close at least $\epsilon$ more gap when comparing DB, V, V+G to G on its own, V+GurF to GurF, and V+GurL to GurL.
        }
\label{tab:gap-closed-summary}
\begin{adjustbox}{max width=1\textwidth}
\begin{tabular}{@{}l@{}
        S[table-format=2.0,table-auto-round,table-number-alignment=center]
        l
        *{1}{S[table-auto-round]}
        H
        *{8}{S[table-auto-round]}
        @{}}
\toprule
{Set} & {\# inst} & {Metric} & {G} & {R} & {DB} & {V} & {max(G,V)} & {V+G} & {GurF} & {V+GurF} & {GurL} & {V+GurL} \\
\midrule
{\multirow[c]{2}{*}{All}} & {\multirow[c]{2}{*}{\tablenum[table-format=3]{332}}} & Avg (\%) & 16.508148432312172 & 0.25674423491955295 & 18.369422595516696 & 12.00826462388178 & 23.00736907616349 & 23.91643511610843 & 28.92974415029431 & 34.8014543982641 & 48.319147189366504 & 52.80602695234524 \\
 &  & Wins &  &  & {\tablenum[table-format=3.0]{178}} & {\tablenum[table-format=3.0]{132}} &  & {\tablenum[table-format=3.0]{226}} &  & {\tablenum[table-format=3.0]{250}} &  & {\tablenum[table-format=3.0]{234}} \\
\midrule
{\multirow[c]{2}{*}{$\ge$10\%}} & {\multirow[c]{2}{*}{\tablenum[table-format=3]{116}}} & Avg (\%) & 16.804983554247737 & 0.6927611181643151 & 38.10016548909812 & 29.795586323445303 & 34.21955710959231 & 35.90407310132605 & 26.410744357400425 & 38.58808083881424 & 45.26830170027557 & 55.39673498133272 \\
 &  & Wins &  &  & {\tablenum[table-format=3.0]{100}} & {\tablenum[table-format=3.0]{92}} &  & {\tablenum[table-format=3.0]{112}} &  & {\tablenum[table-format=3.0]{101}} &  & {\tablenum[table-format=3.0]{107}} \\
\midrule
{\multirow[c]{2}{*}{Binary}} & {\multirow[c]{2}{*}{\tablenum[table-format=3]{65}}} & Avg (\%) & 13.641808347397628 & 0.0 & 21.79883756420864 & 17.30614615821823 & 25.477278114075148 & 26.393399749582038 & 19.229515393723048 & 31.255047689501886 & 32.93591266160529 & 42.79545328908983 \\
 &  & Wins &  &  & {\tablenum[table-format=3.0]{50}} & {\tablenum[table-format=3.0]{37}} &  & {\tablenum[table-format=3.0]{53}} &  & {\tablenum[table-format=3.0]{48}} &  & {\tablenum[table-format=3.0]{48}} \\
\bottomrule
\end{tabular}
\end{adjustbox}
\end{table}
}

The results indicate that VPCs are strong compared to existing cuts.
Namely, using VPCs and GMICs together leads the average percent gap closed at the root to increase from 16.5\% to 23.9\%.
VPCs on their own close strictly more gap than GMICs for 132 instances.
In comparison, for 178 instances, the disjunctive lower bound is greater than the optimal value after adding GMICs, 
so there are only 46 additional instances for which VPCs on their own could have gotten stronger results.
For 24 of those 46 instances, we achieve the cut limit, implying that a higher percent gap might be achieved if we permit more cuts to be generated.
VPCs used with GMICs together outperform GMICs for 226 of the \numgapinst{} instances.
Of the 106 instances in which VPCs and GMICs combined do not improve over GMICs alone,
	for 6 instances, GMICs already close \textasciitilde{}100\% of the gap;
	and for 33 instances, VPCs and GMICs together close 0\% of the gap.
The cut limit is achieved for 24 of the 106 nonimproving instances.

Perhaps even more indicative of the strength of VPCs is when VPCs are used as user cuts within \Gurobi{}, which may employ a variety of cut classes, not only GMICs.
For the first round of cuts at the root, the percent gap closed goes from 28.9\% (without VPCs) to 34.8\% (with them),
with strictly better outcomes for 250 of the \numgapinst{} instances.
For the last round of cuts at the root, the percent gap closed increases from 48.3\% to 52.8\% when using VPCs.

On average, VPCs without GMICs close less of the integrality gap than GMICs do on their own; this occurs for 165 of the \numgapinst{} instances.
We offer two plausible explanations for this phenomenon.
First, for 60 of the 165 instances, we generate very few VPCs compared to GMICs, which makes it difficult to compare the two families directly.
Another 61 of these instances hit the cut limit, so there is further potential for VPCs to improve.
Second, in these results, no \latin{a posteriori} strengthening techniques, such as modularization, are applied to VPCs, while GMICs do take advantage of modularization.

VPCs and GMICs seem to be strong for different types of instances.
One indication is from column ``$\max$(G,V)'': using the best result of only GMICs or only VPCs, per instance, the percent gap closed is 23\%, only 1\% less than combining both families together.
Further evidence comes from the ``$\ge$10\%'' set of instances.
VPCs and GMICs together close over double the percent gap closed by GMICs alone, with improvements for 112 of the \numVgoodinst{} instances in this set,
and VPCs provide a 22\% improvement in the gap closed after the last round of cuts at the root node of Gurobi (55.4\% compared to 45.3\%).
VPCs also outperform GMICs on pure binary instances, offering 30\% improvement in average root percent gap closed for Gurobi.

For the columns including VPCs, the result reported is the maximum percent gap closed across all partial tree sizes tested.
One may initially assume that the strongest cuts would always come from the partial tree with 64 leaf nodes.
This is indeed true for the disjunctive lower bound, but it does not always hold for VPCs for particular instances, though the \emph{average} gap closed across instances steadily increases.
One reason, meaningful in conjunction with the fact that we generate a fixed number of cuts, is that there are likely to be more facet-defining (i.e., essential) inequalities for the disjunctive hull from stronger disjunctions.
As a result, achieving the disjunctive lower bound may become more difficult, in particular given our relatively conservative cut limit.
Another reason, on an intuitive level, is that more of the facet-defining inequalities for the deeper disjunctions may not cut away $\lpopt$, which are cuts we do not generate in these experiments.
Finally, the rate of numerical issues goes up as the disjunctions get larger;
we investigate this more in Appendix~\ref{app:results-objectives}.

{
\sisetup{
    table-alignment-mode = format,
    table-number-alignment = center,
    table-format = 2.2,
}
\begin{table}
\small
\centering
\caption{
            Average percent gap closed broken down by the number of leaf nodes used to construct the partial branch-and-bound tree,
            for VPCs with and without GMICs, as well as at the root by \Gurobi{} after the first and last round of cuts. 
            ``Best'' refers to the maximum gap closed across all partial tree sizes.
        }
\label{tab:depth}
\begin{tabular}{@{}l
        *{6}{S[table-auto-round]}
        @{}}
\toprule
{} & {DB} & {V} & {max(G,V)} & {V+G} & {V+GurF} & {V+GurL} \\
\midrule
VPCs disabled & 0.0 & 0.0 & 16.508148432312172 & 16.508148432312172 & 28.92974415029431 & 48.319147189366504 \\
\midrule
2 leaves & 3.0145069679615464 & 2.213289365939913 & 16.78974579780376 & 17.206578963213257 & 30.96117050218173 & 49.260529886136254 \\
4 leaves & 5.270241662205749 & 3.5924977627287134 & 17.227532496610067 & 17.785030887929768 & 31.16440150982158 & 49.25214246288348 \\
8 leaves & 8.117860021905244 & 4.8474765479675845 & 17.828119270231515 & 18.533394239712695 & 31.526033648713845 & 49.484617007679496 \\
16 leaves & 11.173444037822007 & 6.456210140416377 & 19.373154541924333 & 20.029255265864183 & 32.470786341284295 & 50.2361881728273 \\
32 leaves & 14.671543975982372 & 8.741057007396707 & 20.957474003695218 & 21.77325829522081 & 33.41355016123777 & 51.14084061024237 \\
64 leaves & 18.369422595446917 & 10.233224438179658 & 22.139000791250773 & 22.887076103430644 & 33.82505312597258 & 51.638833320581234 \\
Best & 18.369422595516863 & 12.008264623888747 & 23.00736907616366 & 23.916435116108584 & 34.80145439826474 & 52.80602695234555 \\
\midrule
Combined & 17.795569396793535 & 12.487033115007888 & 23.20799743570185 & 24.123908672499724 & 34.46147366395937 & 52.07456654886541 \\
Rounds & 18.425291259114417 & 15.282946859099736 & 25.557770049311674 & 26.306358243597213 & 35.21185947277031 & 52.67563534101169 \\
\bottomrule
\end{tabular}
\end{table}

}

Table~\ref{tab:depth} shows how the average percent gap closed increases with disjunction size.
In this table, %
row ``Best'' corresponds to the same values as in Table~\ref{tab:gap-closed-summary},
i.e., the best value per instance is used across all partial tree sizes tested.
We show the same metrics as in Table~\ref{tab:gap-closed-summary}.
In this table, VPCs close more gap as stronger disjunctions are used.
However, we also see that there is much room for improvement in our implementation of the framework, as the gap closed by VPCs grows increasingly farther from the disjunctive lower bound with the use of stronger disjunctions.

We also report results from two additional experiments in Table~\ref{tab:depth}.
The penultimate row, ``Combined'', is obtained by applying all VPCs, across all disjunction sizes, 
though still constrained to an overall one-hour cut generation limit,
which is why the value in column ``DB'' is smaller than the corresponding value from the preceding row ``Best''.
The last row, ``Rounds'', reports the outcome of using ``Combined'' for two cut rounds, with doubling the cut limit per disjunction and per round.
The ``Rounds'' setting only contains results for 322 instances, so it cannot be exactly compared to the \numgapinst{} of the preceding rows.
Modulo this caveat, it can be seen that the percent gap closed by VPCs alone increases from 12.5\% to 15.3\%, but the impact on Gurobi in column ``V+GurL'' appears to be relatively minor.
For comparison, two rounds of GMICs correspondingly increases from 16.5\% to 22.3\%.

An important conclusion from Table~\ref{tab:depth} is that our procedure may help in avoiding the ``tailing-off'' effect from recursive applications of cuts: without requiring recursion, by simply using a (sufficiently) stronger disjunction, we make relatively steady progress toward the optimal value of \eqref{IP}.
However, this is only in terms of percent gap closed; as we discuss in the next section, the story when using the cuts within branch and bound is completely different, in which seemingly weaker cuts may lead to better performance.

\subsection{Effect with branch and bound}
\label{sec:bb-time}

We now turn to the second metric: the effect of our cuts on branch and bound in terms of time and number of nodes.
We compare two solvers: ``Gur'' is the baseline of default Gurobi, and ``V'' denotes  \Gurobi{} with VPCs added as user cuts.
Table~\ref{tab:bb-summary-depth} contains a summary of the statistics for several instance sets, called
    ``Combined'', ``Rounds'', and ``$\ell$ leaves'' for $\ell \in \{2,4,8,16,32,64\}$,
where
    \begin{itemize}
        \item
            ``Combined'' refers to all VPCs generated within an hour across all disjunction sizes considered;
        \item
            ``Rounds'' refers to two rounds of ``Combined'' VPC generation with a cut limit of twice the number of fractional integer variables in $\lpopt$;
            and
        \item
            ``$\ell$ leaves'' refers to terminating partial branch-and-bound tree generation when there are $\ell$ open (leaf) nodes.
    \end{itemize}
Each instance is solved with 7 random seeds per solver.
For each seed that times out, for solving time, 7200 seconds (twice the time limit) is used, and for calculating nodes, the number of processed nodes is also multiplied by 2.
Within each set, we create four bins,
where bin [$t$,3600) contains the subset of the \numgapinst{} instances used in Section~\ref{sec:vpc-gap-closed} for which the average solution time (across 7 random seeds) is at least $t$ seconds for both solvers and under 3600 seconds for at least one solver.

The first column of Table~\ref{tab:bb-summary-depth} indicates which set and bin is being considered.
The second column is the number of instances in that subset.
The next column indicates the two summary statistics presented for each subset.
The first statistic row, ``Gmean'', for each subset is a shifted geometric mean, with a shift of 1 for time and 1000 for nodes, modified from \citet{Achterberg07}.
The second row, ``Wins'', reports the number of instances for which each of the two solver options ``Gur'' and ``V'' perform better, where
an instance counts as a win by time for a solver if the other solver has an average running time that is at least 10\% slower,
and an instance counts as a win by nodes for a solver when the number of nodes is strictly better.
The remaining columns contain the corresponding statistics for each solver and metric,
as well as column ``Gen'' that gives the geometric mean of cut generation time for instances within each set and bin.

From Table~\ref{tab:bb-summary-depth}, we conclude that adding VPCs tends to slightly degrade the running time of Gurobi,
even though the number of processed nodes frequently decreases.
For both the ``Combined'' and ``Rounds'' sets, the default ``Gur'' solver runs faster, especially when considering the running time reported in column ``Gen'', though there is a reduction in nodes for the $[1000,3600)$ bucket.
In the remaining sets, there are isolated cases of promise.
For example, under ``4 leaves'', we see that the geometric mean of the number of processed nodes decreases for the first three buckets when using VPCs,
though the corresponding geometric mean running times are about the same for both solvers.
There is a small improvement in running time for the ``harder'' instances, in the $[1000,3600)$ bucket, for several parameter settings: ``2 leaves'', ``8 leaves'', and ``16 leaves'' (even when considering cut generation time).
We caution that, despite our best efforts, we might not have completely removed confounding factors such as machine variability, which may cause inconsistent results when repeating the experiments.
Nevertheless, these results indicate that there is a significant portion of instances from the dataset that may benefit from VPCs.
At the same time, it is not clear how to identify such instances, effectively select VPC hyperparameters, or incorporate VPCs within Gurobi in general.

{%
\sisetup{
   table-alignment-mode = format,
   table-number-alignment = center,
   table-format = 4.2,
}
\setlength\LTleft{-21.038pt}%
 \begin{tabularx}{\textwidth}{@{}l    %
       c       %
       l       %
       *{3}{S[table-auto-round,table-format=4.2]}
       *{2}{S[table-auto-round,table-format=7.0]}
       @{}}
	\caption{%
	  Summary statistics when solving instances with branch and bound.
	}
    \label{tab:bb-summary-depth}\\[0pt]

\toprule
& & & \multicolumn{3}{c}{Time (s)}  & \multicolumn{2}{c}{Nodes (\#)}    \\ 
\cmidrule(lr){4-6} \cmidrule(l){7-8}
{Set} & {\# inst} & {Metric} & {Gur} & {V} & {Gen} & {Gur} & {V} \\
\midrule
\endfirsthead

\toprule
& & & \multicolumn{3}{c}{Time (s)}  & \multicolumn{2}{c}{Nodes (\#)}    \\ 
\cmidrule(lr){4-6} \cmidrule(l){7-8}
{Set} & {\# inst} & {Metric} & {Gur} & {V} & {Gen} & {Gur} & {V} \\
\midrule
\endhead

\endfoot

\endlastfoot
\multirow{2}{*}{\shortstack[l]{Combined\\\relax [0,3600)}} & {\multirow[c]{2}{*}{\tablenum[table-format=3]{288}}} & Gmean & 11.256821 & 11.839758 & 27.617049 & 7681.913224 & 7793.801266 \\
 &  & Wins & {\tablenum[table-format=4.0]{117.0}\phantom{.00}} & {\tablenum[table-format=4.0]{62.0}\phantom{.00}} &  & {\tablenum[table-format=6.0]{128.0}} & {\tablenum[table-format=6.0]{131.0}} \\
\multirow{2}{*}{\shortstack[l]{Combined\\\relax [10,3600)}} & {\multirow[c]{2}{*}{\tablenum[table-format=3]{119}}} & Gmean & 157.263482 & 168.453590 & 55.301806 & 75193.842935 & 76955.769766 \\
 &  & Wins & {\tablenum[table-format=4.0]{44.0}\phantom{.00}} & {\tablenum[table-format=4.0]{32.0}\phantom{.00}} &  & {\tablenum[table-format=6.0]{54.0}} & {\tablenum[table-format=6.0]{65.0}} \\
\multirow{2}{*}{\shortstack[l]{Combined\\\relax [100,3600)}} & {\multirow[c]{2}{*}{\tablenum[table-format=3]{69}}} & Gmean & 505.119151 & 550.557326 & 43.695307 & 244802.368364 & 255028.026777 \\
 &  & Wins & {\tablenum[table-format=4.0]{32.0}\phantom{.00}} & {\tablenum[table-format=4.0]{15.0}\phantom{.00}} &  & {\tablenum[table-format=6.0]{38.0}} & {\tablenum[table-format=6.0]{31.0}} \\
\multirow{2}{*}{\shortstack[l]{Combined\\\relax [1000,3600)}} & {\multirow[c]{2}{*}{\tablenum[table-format=3]{24}}} & Gmean & 1972.953117 & 2073.896266 & 35.858057 & 1008339.705759 & 996884.899183 \\
 &  & Wins & {\tablenum[table-format=4.0]{10.0}\phantom{.00}} & {\tablenum[table-format=4.0]{5.0}\phantom{.00}} &  & {\tablenum[table-format=6.0]{11.0}} & {\tablenum[table-format=6.0]{13.0}} \\
\midrule
\multirow{2}{*}{\shortstack[l]{Rounds\\\relax [0,3600)}} & {\multirow[c]{2}{*}{\tablenum[table-format=3]{279}}} & Gmean & 11.621243 & 12.470732 & 128.876724 & 7845.546959 & 7919.763758 \\
 &  & Wins & {\tablenum[table-format=4.0]{144.0}\phantom{.00}} & {\tablenum[table-format=4.0]{61.0}\phantom{.00}} &  & {\tablenum[table-format=6.0]{129.0}} & {\tablenum[table-format=6.0]{123.0}} \\
\multirow{2}{*}{\shortstack[l]{Rounds\\\relax [10,3600)}} & {\multirow[c]{2}{*}{\tablenum[table-format=3]{117}}} & Gmean & 155.225984 & 165.437991 & 216.395625 & 74002.424463 & 74706.008320 \\
 &  & Wins & {\tablenum[table-format=4.0]{45.0}\phantom{.00}} & {\tablenum[table-format=4.0]{34.0}\phantom{.00}} &  & {\tablenum[table-format=6.0]{58.0}} & {\tablenum[table-format=6.0]{59.0}} \\
\multirow{2}{*}{\shortstack[l]{Rounds\\\relax [100,3600)}} & {\multirow[c]{2}{*}{\tablenum[table-format=3]{64}}} & Gmean & 526.969209 & 558.173541 & 193.578409 & 235024.497883 & 241114.762824 \\
 &  & Wins & {\tablenum[table-format=4.0]{22.0}\phantom{.00}} & {\tablenum[table-format=4.0]{16.0}\phantom{.00}} &  & {\tablenum[table-format=6.0]{36.0}} & {\tablenum[table-format=6.0]{28.0}} \\
\multirow{2}{*}{\shortstack[l]{Rounds\\\relax [1000,3600)}} & {\multirow[c]{2}{*}{\tablenum[table-format=3]{23}}} & Gmean & 1868.614466 & 1982.474085 & 166.964579 & 1041127.748276 & 1032613.929929 \\
 &  & Wins & {\tablenum[table-format=4.0]{8.0}\phantom{.00}} & {\tablenum[table-format=4.0]{5.0}\phantom{.00}} &  & {\tablenum[table-format=6.0]{13.0}} & {\tablenum[table-format=6.0]{10.0}} \\
\midrule
\multirow{2}{*}{\shortstack[l]{2 leaves\\\relax [0,3600)}} & {\multirow[c]{2}{*}{\tablenum[table-format=3]{261}}} & Gmean & 9.264507 & 9.505446 & 0.561007 & 6309.540478 & 6452.580675 \\
 &  & Wins & {\tablenum[table-format=4.0]{64.0}\phantom{.00}} & {\tablenum[table-format=4.0]{52.0}\phantom{.00}} &  & {\tablenum[table-format=6.0]{102.0}} & {\tablenum[table-format=6.0]{135.0}} \\
\multirow{2}{*}{\shortstack[l]{2 leaves\\\relax [10,3600)}} & {\multirow[c]{2}{*}{\tablenum[table-format=3]{102}}} & Gmean & 129.679451 & 139.349050 & 1.188230 & 59630.814167 & 64032.410898 \\
 &  & Wins & {\tablenum[table-format=4.0]{33.0}\phantom{.00}} & {\tablenum[table-format=4.0]{22.0}\phantom{.00}} &  & {\tablenum[table-format=6.0]{44.0}} & {\tablenum[table-format=6.0]{58.0}} \\
\multirow{2}{*}{\shortstack[l]{2 leaves\\\relax [100,3600)}} & {\multirow[c]{2}{*}{\tablenum[table-format=3]{50}}} & Gmean & 503.496056 & 534.573401 & 1.003378 & 186171.780626 & 199675.888781 \\
 &  & Wins & {\tablenum[table-format=4.0]{13.0}\phantom{.00}} & {\tablenum[table-format=4.0]{10.0}\phantom{.00}} &  & {\tablenum[table-format=6.0]{18.0}} & {\tablenum[table-format=6.0]{32.0}} \\
\multirow{2}{*}{\shortstack[l]{2 leaves\\\relax [1000,3600)}} & {\multirow[c]{2}{*}{\tablenum[table-format=3]{15}}} & Gmean & 1974.557425 & 1894.638111 & 0.296151 & 799031.956077 & 729535.070710 \\
 &  & Wins & {\tablenum[table-format=4.0]{4.0}\phantom{.00}} & {\tablenum[table-format=4.0]{4.0}\phantom{.00}} &  & {\tablenum[table-format=6.0]{4.0}} & {\tablenum[table-format=6.0]{11.0}} \\
\midrule
\multirow{2}{*}{\shortstack[l]{4 leaves\\\relax [0,3600)}} & {\multirow[c]{2}{*}{\tablenum[table-format=3]{266}}} & Gmean & 8.813739 & 8.802836 & 1.188708 & 6017.400732 & 5963.029381 \\
 &  & Wins & {\tablenum[table-format=4.0]{70.0}\phantom{.00}} & {\tablenum[table-format=4.0]{55.0}\phantom{.00}} &  & {\tablenum[table-format=6.0]{112.0}} & {\tablenum[table-format=6.0]{127.0}} \\
\multirow{2}{*}{\shortstack[l]{4 leaves\\\relax [10,3600)}} & {\multirow[c]{2}{*}{\tablenum[table-format=3]{104}}} & Gmean & 117.380840 & 115.831482 & 2.375498 & 55206.090918 & 53955.814521 \\
 &  & Wins & {\tablenum[table-format=4.0]{26.0}\phantom{.00}} & {\tablenum[table-format=4.0]{29.0}\phantom{.00}} &  & {\tablenum[table-format=6.0]{49.0}} & {\tablenum[table-format=6.0]{55.0}} \\
\multirow{2}{*}{\shortstack[l]{4 leaves\\\relax [100,3600)}} & {\multirow[c]{2}{*}{\tablenum[table-format=3]{48}}} & Gmean & 466.634515 & 458.663953 & 2.557509 & 166328.206782 & 162475.081286 \\
 &  & Wins & {\tablenum[table-format=4.0]{12.0}\phantom{.00}} & {\tablenum[table-format=4.0]{12.0}\phantom{.00}} &  & {\tablenum[table-format=6.0]{19.0}} & {\tablenum[table-format=6.0]{29.0}} \\
\multirow{2}{*}{\shortstack[l]{4 leaves\\\relax [1000,3600)}} & {\multirow[c]{2}{*}{\tablenum[table-format=3]{13}}} & Gmean & 1802.023961 & 1875.014178 & 0.962807 & 790760.039514 & 792726.306079 \\
 &  & Wins & {\tablenum[table-format=4.0]{4.0}\phantom{.00}} & {\tablenum[table-format=4.0]{4.0}\phantom{.00}} &  & {\tablenum[table-format=6.0]{7.0}} & {\tablenum[table-format=6.0]{6.0}} \\
\midrule
\multirow{2}{*}{\shortstack[l]{8 leaves\\\relax [0,3600)}} & {\multirow[c]{2}{*}{\tablenum[table-format=3]{259}}} & Gmean & 9.186959 & 9.681584 & 2.211241 & 6717.703857 & 7034.712212 \\
 &  & Wins & {\tablenum[table-format=4.0]{89.0}\phantom{.00}} & {\tablenum[table-format=4.0]{46.0}\phantom{.00}} &  & {\tablenum[table-format=6.0]{127.0}} & {\tablenum[table-format=6.0]{107.0}} \\
\multirow{2}{*}{\shortstack[l]{8 leaves\\\relax [10,3600)}} & {\multirow[c]{2}{*}{\tablenum[table-format=3]{99}}} & Gmean & 139.393297 & 153.050294 & 4.716767 & 74523.482792 & 82415.763481 \\
 &  & Wins & {\tablenum[table-format=4.0]{42.0}\phantom{.00}} & {\tablenum[table-format=4.0]{24.0}\phantom{.00}} &  & {\tablenum[table-format=6.0]{58.0}} & {\tablenum[table-format=6.0]{41.0}} \\
\multirow{2}{*}{\shortstack[l]{8 leaves\\\relax [100,3600)}} & {\multirow[c]{2}{*}{\tablenum[table-format=3]{52}}} & Gmean & 510.319081 & 557.747570 & 5.163097 & 213765.198022 & 241449.818517 \\
 &  & Wins & {\tablenum[table-format=4.0]{24.0}\phantom{.00}} & {\tablenum[table-format=4.0]{13.0}\phantom{.00}} &  & {\tablenum[table-format=6.0]{29.0}} & {\tablenum[table-format=6.0]{23.0}} \\
\multirow{2}{*}{\shortstack[l]{8 leaves\\\relax [1000,3600)}} & {\multirow[c]{2}{*}{\tablenum[table-format=3]{16}}} & Gmean & 2046.171910 & 1965.601753 & 2.644587 & 750097.415912 & 709701.100125 \\
 &  & Wins & {\tablenum[table-format=4.0]{7.0}\phantom{.00}} & {\tablenum[table-format=4.0]{6.0}\phantom{.00}} &  & {\tablenum[table-format=6.0]{7.0}} & {\tablenum[table-format=6.0]{9.0}} \\
\midrule
\multirow{2}{*}{\shortstack[l]{16 leaves\\\relax [0,3600)}} & {\multirow[c]{2}{*}{\tablenum[table-format=3]{261}}} & Gmean & 9.216628 & 9.462807 & 3.909299 & 6964.502129 & 7019.523970 \\
 &  & Wins & {\tablenum[table-format=4.0]{78.0}\phantom{.00}} & {\tablenum[table-format=4.0]{63.0}\phantom{.00}} &  & {\tablenum[table-format=6.0]{104.0}} & {\tablenum[table-format=6.0]{133.0}} \\
\multirow{2}{*}{\shortstack[l]{16 leaves\\\relax [10,3600)}} & {\multirow[c]{2}{*}{\tablenum[table-format=3]{100}}} & Gmean & 139.444260 & 147.970481 & 6.658429 & 80057.848664 & 83603.489488 \\
 &  & Wins & {\tablenum[table-format=4.0]{30.0}\phantom{.00}} & {\tablenum[table-format=4.0]{28.0}\phantom{.00}} &  & {\tablenum[table-format=6.0]{45.0}} & {\tablenum[table-format=6.0]{55.0}} \\
\multirow{2}{*}{\shortstack[l]{16 leaves\\\relax [100,3600)}} & {\multirow[c]{2}{*}{\tablenum[table-format=3]{50}}} & Gmean & 535.232775 & 562.072782 & 7.832284 & 226992.478009 & 236017.937245 \\
 &  & Wins & {\tablenum[table-format=4.0]{15.0}\phantom{.00}} & {\tablenum[table-format=4.0]{14.0}\phantom{.00}} &  & {\tablenum[table-format=6.0]{21.0}} & {\tablenum[table-format=6.0]{29.0}} \\
\multirow{2}{*}{\shortstack[l]{16 leaves\\\relax [1000,3600)}} & {\multirow[c]{2}{*}{\tablenum[table-format=3]{18}}} & Gmean & 1891.260914 & 1824.106626 & 4.820112 & 995349.746343 & 889380.207280 \\
 &  & Wins & {\tablenum[table-format=4.0]{5.0}\phantom{.00}} & {\tablenum[table-format=4.0]{7.0}\phantom{.00}} &  & {\tablenum[table-format=6.0]{4.0}} & {\tablenum[table-format=6.0]{14.0}} \\
\midrule
\multirow{2}{*}{\shortstack[l]{32 leaves\\\relax [0,3600)}} & {\multirow[c]{2}{*}{\tablenum[table-format=3]{246}}} & Gmean & 8.354017 & 8.660765 & 7.333595 & 6474.155040 & 6411.111746 \\
 &  & Wins & {\tablenum[table-format=4.0]{85.0}\phantom{.00}} & {\tablenum[table-format=4.0]{53.0}\phantom{.00}} &  & {\tablenum[table-format=6.0]{104.0}} & {\tablenum[table-format=6.0]{117.0}} \\
\multirow{2}{*}{\shortstack[l]{32 leaves\\\relax [10,3600)}} & {\multirow[c]{2}{*}{\tablenum[table-format=3]{91}}} & Gmean & 128.969218 & 136.264026 & 13.085646 & 78583.258429 & 77145.758037 \\
 &  & Wins & {\tablenum[table-format=4.0]{30.0}\phantom{.00}} & {\tablenum[table-format=4.0]{26.0}\phantom{.00}} &  & {\tablenum[table-format=6.0]{35.0}} & {\tablenum[table-format=6.0]{56.0}} \\
\multirow{2}{*}{\shortstack[l]{32 leaves\\\relax [100,3600)}} & {\multirow[c]{2}{*}{\tablenum[table-format=3]{43}}} & Gmean & 519.646221 & 567.703026 & 12.132020 & 259501.616701 & 261929.256792 \\
 &  & Wins & {\tablenum[table-format=4.0]{14.0}\phantom{.00}} & {\tablenum[table-format=4.0]{10.0}\phantom{.00}} &  & {\tablenum[table-format=6.0]{15.0}} & {\tablenum[table-format=6.0]{28.0}} \\
\multirow{2}{*}{\shortstack[l]{32 leaves\\\relax [1000,3600)}} & {\multirow[c]{2}{*}{\tablenum[table-format=3]{14}}} & Gmean & 1876.718030 & 1940.795200 & 10.713437 & 613961.860202 & 547304.160732 \\
 &  & Wins & {\tablenum[table-format=4.0]{6.0}\phantom{.00}} & {\tablenum[table-format=4.0]{5.0}\phantom{.00}} &  & {\tablenum[table-format=6.0]{3.0}} & {\tablenum[table-format=6.0]{11.0}} \\
\midrule
\multirow{2}{*}{\shortstack[l]{64 leaves\\\relax [0,3600)}} & {\multirow[c]{2}{*}{\tablenum[table-format=3]{229}}} & Gmean & 7.878246 & 8.349617 & 10.393096 & 6673.861463 & 6817.203112 \\
 &  & Wins & {\tablenum[table-format=4.0]{90.0}\phantom{.00}} & {\tablenum[table-format=4.0]{47.0}\phantom{.00}} &  & {\tablenum[table-format=6.0]{107.0}} & {\tablenum[table-format=6.0]{99.0}} \\
\multirow{2}{*}{\shortstack[l]{64 leaves\\\relax [10,3600)}} & {\multirow[c]{2}{*}{\tablenum[table-format=3]{82}}} & Gmean & 135.620735 & 150.597977 & 18.134402 & 92629.949412 & 96976.105799 \\
 &  & Wins & {\tablenum[table-format=4.0]{35.0}\phantom{.00}} & {\tablenum[table-format=4.0]{18.0}\phantom{.00}} &  & {\tablenum[table-format=6.0]{42.0}} & {\tablenum[table-format=6.0]{40.0}} \\
\multirow{2}{*}{\shortstack[l]{64 leaves\\\relax [100,3600)}} & {\multirow[c]{2}{*}{\tablenum[table-format=3]{41}}} & Gmean & 483.705222 & 558.892444 & 22.095454 & 240939.349494 & 257322.059059 \\
 &  & Wins & {\tablenum[table-format=4.0]{14.0}\phantom{.00}} & {\tablenum[table-format=4.0]{10.0}\phantom{.00}} &  & {\tablenum[table-format=6.0]{20.0}} & {\tablenum[table-format=6.0]{21.0}} \\
\multirow{2}{*}{\shortstack[l]{64 leaves\\\relax [1000,3600)}} & {\multirow[c]{2}{*}{\tablenum[table-format=3]{12}}} & Gmean & 1804.236845 & 2254.478262 & 20.507652 & 546738.715700 & 549838.851544 \\
 &  & Wins & {\tablenum[table-format=4.0]{6.0}\phantom{.00}} & {\tablenum[table-format=4.0]{2.0}\phantom{.00}} &  & {\tablenum[table-format=6.0]{6.0}} & {\tablenum[table-format=6.0]{6.0}} \\
\bottomrule
\end{tabularx}
}

\section{Conclusion \& open problems}
\label{sec:conclusion}

This paper presents a step toward merging cut-generation and branching in integer programming solvers by providing a computationally tractable method for generating cuts from partial branch-and-bound trees.
The framework we introduce is to (1) select a disjunction, (2) choose a (compact) $\mathcal{V}$-polyhedral relaxation for each disjunctive term, 
and (3) selectively generate cuts by judiciously choosing objective directions to optimize over \eqref{PRLP} formed from the point-ray collection.

Our investigation touches on each of these aspects.
The first is the disjunction choice, which in our experiments is the set of leaf nodes of a partial branch-and-bound tree.
We only experiment with the size of the disjunction, but we do not extensively test alternative disjunctions or claim to prescribe the best (or even good) disjunctions for the purposes of cut generation, which is left as an open problem for future work.
As discussed in 
Appendix~G of the extended manuscript~\cite{BalKaz22+_vpc-arxiv},
just one sufficiently strong disjunction chosen this way can lead to cuts that are stronger than those from all possible split and cross disjunctions.
Although using stronger disjunctions does lead to better gap closed, our results indicate that this monotonicity does not hold when embedding the cuts within branch and bound in a solver.
A better understanding of the interaction between the branch-and-bound process and cutting planes remains an open problem meriting future research~\cite{ConLodiTra22_cuts-from-branch-and-bound-tree, KazLeBSan22_abstractbc-mpb}.

The relaxation for each disjunctive term that we use is quite simple, but therein lies its advantage.
Nevertheless, we do show examples highlighting the weakness of our simple relaxations---that only a subset of all the valid disjunctive cuts can be produced---and the computational results do, at times, reflect this weakness.
Thus, there is an opportunity to improve the quality of the generated VPCs by considering tighter relaxations generated from structural information about each instance.

For the objective directions, we provide theoretical support for objective directions to \eqref{PRLP} that yield new and strong VPCs more frequently than previous approaches (reducing the percent of objectives failing to produce a cut from 80\% in early experiments, to around 30\% in the current implementation).

Overall, the cuts we generate are strong, as evidenced by the percent integrality gap they close in our experiments (compared to both GMICs and the default cut setting of \Gurobi{}).
Moreover, %
the integrality gap closed by VPCs increases steadily with the use of stronger disjunctions, 
which may help to avoid the common tailing off of strength experienced by other cut families that require recursive applications to reach strong cuts.
In addition, for some instances, our results show that the extra computational effort pays off in reduced branch-and-bound time,
but on average across the instances tested, employing VPCs causes slower performance.
We conclude that, though VPCs may not yet improve most solvers, the VPC framework has theoretical advantages with the potential for practical impact on cut generation.

{
\paragraph{Acknowledgments.}
\small
This work was supported in part by NSF grant CMMI1560828 and ONR contract N00014-15-12082.
We also greatly appreciate the reviewers' valuable feedback.
}

{
\paragraph{Conflicts of interest.}
\small
The authors have no conflicts of interest to declare.
}

\biblio

\newpage
\appendix

\section{Analysis of effect of disjunction size}
\label{app:results-partial-tree-sizes}

One aspect that is hidden in the results of Section~\ref{sec:computation} is the number of leaf nodes used for the partial branch-and-bound tree to obtain the reported gap closed and branch-and-bound times.
Table~\ref{tab:gap-closed-summary} reports the best result across all the tested sizes of the partial tree, i.e., with number of leaf nodes $\ell \in \{2,4,8,16,32,64\}$.
We next disaggregate the analysis to see how the different size partial trees perform alone.
For data including VPCs, a partial tree size may be specified (either as ``V ($\ell$)'' or ``$\ell$ leaves''), indicating that these results concern only the runs for partial trees with $\ell$ leaf nodes.

{
\sisetup{
    table-alignment-mode = format,
    table-number-alignment = center,
    table-format = 3.0,
}
\begin{table}[b]
  \captionsetup{font=small}
\small
\centering
\caption{Number of leaf nodes yielding the best result for each experiment per instance.}
\label{tab:size}
\begin{tabular}{@{}l*4{S}*2{S}@{}}
\toprule
 & \multicolumn{4}{c}{Gap} 
 & {Time} & {Nodes}
 \\
   \cmidrule(lr){2-5} 
   \cmidrule(l){6-6}
   \cmidrule(l){7-7}
 & {V} & {V+G} & {V+GurF} & {V+GurL}
& {All} & {All} %
 \\
\midrule
No improvement & 200 & 106 & 82 & 98 & 178 & 79 \\
2 leaves & 0 & 10 & 53 & 39 & 27 & 38 \\
4 leaves & 2 & 8 & 65 & 42 & 37 & 37 \\
8 leaves & 5 & 20 & 60 & 37 & 23 & 35 \\
16 leaves & 4 & 18 & 79 & 44 & 25 & 41 \\
32 leaves & 16 & 37 & 72 & 44 & 32 & 40 \\
64 leaves & 105 & 143 & 122 & 96 & 27 & 38 \\
\bottomrule
\end{tabular}
\end{table}
}

Table~\ref{tab:size} gives %
the number of instances for which the best result over the appropriate baseline is achieved when restricted to a particular number of leaves.
We use the same tolerances to count wins as in Section~\ref{sec:computation}.
The first column is the tree size, including the option of no VPCs, %
the second through fifth columns refer to gap closed for the \numgapinst{} instances used in the strength experiments,
while the last columns give solving time and nodes for the \numtimeinst{} instances solved within 3600 seconds on average for either the default Gurobi solver or for Gurobi with VPCs for at least one setting for $\ell$.
We see, as in Table~\ref{tab:depth}, that the tree with the most leaf nodes produces the best percent gap closed quite often, but not always; there are even instances for which $\ell=2$, i.e., a single split disjunction, suffices to achieve the best result.
In contrast, there is no clear winner in terms of branch-and-bound metrics.

This phenomenon of stronger cuts not being directly correlated to better branch and bound performance is not easy to remedy, due to the hard-to-predict effect of cuts on the branch-and-bound process.
Nevertheless, there is one aspect of Table~\ref{tab:bb-summary-depth} that suggests a possible explanation.
Observe that the number of wins with VPCs in terms of number of nodes is typically higher than in terms of time.
For each instance in which the number of nodes decreases but the time increases, 
solving the relaxation at each node of the branch-and-bound tree might be slower with additional cuts added.
It is commonly known that making the coefficient matrix denser will slow down the solution of a linear program.
We next look at how the density of VPCs changes as deeper disjunctions are used.

Table~\ref{tab:density} gives the average cut density statistics across the different partial tree sizes for the set ``All''.
The density of a cut is defined as the number of nonzero cut coefficients divided by the total number of coefficients.
The first row is the number of instances having VPCs for each tree size.
The second row is the number of instances where ``V'' wins on time compared to ``Gur''.
The next three rows give the average of the minimum, maximum, and average densities of VPCs for each instance.
The last two rows give the average of the average cut densities for (1) the instances counted in the second row, i.e., those where VPCs improve time with respect to ``Gur'', and (2) the instances for which ``Gur'' wins over ``V''.

{
\sisetup{
    table-alignment-mode = format,
    table-number-alignment = center,
    table-format = 0.3,
}
\begin{table}
  \captionsetup{font=small}
\small
\centering
\caption{Statistics about the density of generated cuts broken down by partial tree size.}
\label{tab:density}
\begin{adjustbox}{max width=1\textwidth}
\begin{tabular}{@{}l*{6}{S[table-format=0.3,table-auto-round,table-number-alignment=center]}@{}}
\toprule
{} & {V (2)} & {V (4)} & {V (8)} & {V (16)} & {V (32)} & {V (64)} \\
\midrule
\# inst w/VPCs and time < 3600s & {\tablenum[table-format=3.0]{261.0}} & {\tablenum[table-format=3.0]{266.0}} & {\tablenum[table-format=3.0]{259.0}} & {\tablenum[table-format=3.0]{261.0}} & {\tablenum[table-format=3.0]{246.0}} & {\tablenum[table-format=3.0]{229.0}} \\
\# wins by time & {\tablenum[table-format=3.0]{52.0}} & {\tablenum[table-format=3.0]{55.0}} & {\tablenum[table-format=3.0]{46.0}} & {\tablenum[table-format=3.0]{63.0}} & {\tablenum[table-format=3.0]{53.0}} & {\tablenum[table-format=3.0]{47.0}} \\
Avg min cut density & 0.216980 & 0.236145 & 0.274479 & 0.301382 & 0.349120 & 0.391576 \\
Avg max cut density & 0.343876 & 0.382824 & 0.404883 & 0.426546 & 0.484096 & 0.487735 \\
Avg avg cut density & 0.279550 & 0.310963 & 0.343381 & 0.373185 & 0.426838 & 0.453378 \\
Avg avg cut density (win by time) & 0.289946 & 0.298028 & 0.212830 & 0.276679 & 0.360823 & 0.398727 \\
Avg avg cut density (lose by time) & 0.278624 & 0.311245 & 0.346565 & 0.398178 & 0.486850 & 0.501303 \\
\bottomrule
\end{tabular}
\end{adjustbox}
\end{table}
}

The first observation is that the number of instances with VPCs decreases from 261 for the setting with 2 leaf nodes to 229 for the 64 leaf node setting.
The average cut density goes from 0.280 to 0.453 on average, which could help explain the fact that the stronger cuts from stronger disjunctions can yield worse branch and bound times.
In addition, the last two rows suggest that density is correlated to whether VPCs help for an instance.
Namely, for each column, the average cut density is always lower for those instances that win by time.
Future experiments may benefit from doing cut filtering by density, or possibly reducing the density of generated cuts \latin{a posteriori} while sacrificing strength. %

\section{Analysis of objective function choices}
\label{app:results-objectives}
Next, via Table~\ref{tab:objectives}, we discuss our objective function choices,
including statistics on the frequency of failures as a function of disjunction size.
A caveat is that our analysis of objectives functions is somewhat limited by our relatively conservative strategy in selecting objectives to use, in order to limit time spent on cut generation and the types of failures discussed in Section~\ref{sec:obj-theory}.
The objectives we choose utilize the results of \citet{KazNadBalMar20}, suggesting that a successful class of objectives is the set of points and rays of the point-ray collection, but, motivated by Proposition~\ref{prop:disjLB}, we only do this for cuts tight at $\ref{p*}$ (the loop starting at step~\ref{step:VPC1:DB} of Algorithm~\ref{alg:VPC1}).
It is natural to consider cuts that only lie on deeper points, but this, as well as cuts that are satisfied by $\lpopt$, remains a topic for future research.

{
\sisetup{
    table-alignment-mode = format,
    table-number-alignment = center,
    table-format = 2.2,
}
\begin{table}[b]
  \captionsetup{font=small}
\small
\centering
\caption{Statistics about objectives leading to failures, by partial tree size.}
\label{tab:objectives}
\begin{tabular}{@{}l*{6}{S[table-format=2.2,table-auto-round,table-number-alignment=center]}@{}}\toprule
{} & {V (2)} & {V (4)} & {V (8)} & {V (16)} & {V (32)} & {V (64)} \\
\midrule
\# inst w/obj & {\tablenum[table-format=3.0]{311.0}} & {\tablenum[table-format=3.0]{318.0}} & {\tablenum[table-format=3.0]{304.0}} & {\tablenum[table-format=3.0]{306.0}} & {\tablenum[table-format=3.0]{289.0}} & {\tablenum[table-format=3.0]{271.0}} \\
\# inst w/succ obj & {\tablenum[table-format=3.0]{308.0}} & {\tablenum[table-format=3.0]{315.0}} & {\tablenum[table-format=3.0]{300.0}} & {\tablenum[table-format=3.0]{302.0}} & {\tablenum[table-format=3.0]{285.0}} & {\tablenum[table-format=3.0]{263.0}} \\
\# inst no obj & {\tablenum[table-format=3.0]{21.0}} & {\tablenum[table-format=3.0]{14.0}} & {\tablenum[table-format=3.0]{28.0}} & {\tablenum[table-format=3.0]{26.0}} & {\tablenum[table-format=3.0]{43.0}} & {\tablenum[table-format=3.0]{61.0}} \\
\# inst all obj fail & {\tablenum[table-format=3.0]{3.0}} & {\tablenum[table-format=3.0]{3.0}} & {\tablenum[table-format=3.0]{4.0}} & {\tablenum[table-format=3.0]{4.0}} & {\tablenum[table-format=3.0]{4.0}} & {\tablenum[table-format=3.0]{8.0}} \\
\# inst all obj succ & {\tablenum[table-format=3.0]{30.0}} & {\tablenum[table-format=3.0]{41.0}} & {\tablenum[table-format=3.0]{31.0}} & {\tablenum[table-format=3.0]{34.0}} & {\tablenum[table-format=3.0]{27.0}} & {\tablenum[table-format=3.0]{24.0}} \\
\midrule
\% obj fails & 28.572839 & 25.279546 & 26.495366 & 30.774286 & 31.833289 & 34.084307 \\
\midrule
\% fails dup & 50.672298 & 39.337152 & 41.105995 & 48.784992 & 51.021596 & 56.847203 \\
\% fails unbdd & 38.737356 & 47.589754 & 45.409934 & 35.735784 & 32.210717 & 27.495963 \\
\% fails tilim & 1.379326 & 3.673926 & 3.607452 & 5.130045 & 6.267595 & 5.910538 \\
\% fails dyn & 9.075878 & 8.736189 & 9.609179 & 9.418319 & 9.312240 & 9.377769 \\
\midrule
\% fails all ones & 26.746944 & 30.729987 & 25.846494 & 20.687139 & 17.435861 & 14.174297 \\
\% fails post-GMIC obj & 14.594832 & 20.446940 & 22.883502 & 21.245874 & 20.508121 & 17.376528 \\
\% fails DB & 58.658225 & 48.823073 & 51.270005 & 58.066987 & 62.056018 & 68.449175 \\
\midrule
\# obj / cut & 2.572361 & 2.307938 & 2.380017 & 2.761221 & 2.969413 & 3.444549 \\
(s) / obj & 0.290980 & 4.418293 & 0.771466 & 0.920539 & 1.931006 & 3.016811 \\
(s) / cut & 0.360224 & 1.678482 & 6.324945 & 4.976613 & 11.674281 & 22.383200 \\
\bottomrule
\end{tabular}
\end{table}
}

Table~\ref{tab:objectives} summarizes objective failures using all \numgapinst{} instances from the gap closed experiments.
The columns of the table are the same as Table~\ref{tab:density}.
The rows are divided into several blocks.
The first block gives statistics on the number of instances for each column for which: (1) objectives were tried, (2) VPCs were generated, (3) objectives were not tried, (4) none of the objectives yielded VPCs, and (5) all of the objectives led to distinct VPCs.
The next block gives the average percent of the objectives that were failures.
The subsequent block of rows looks at the cause of these failures, which fall into one of four categories:
\begin{itemize}[itemsep=0pt]
  \item ``Dup'': the optimal solution to the PRLP is an exact duplicate of an existing cut
  \item ``Unbdd'': \eqref{PRLP} does not have a finite solution for that objective
  \item ``Tilim'': the time limit for \eqref{PRLP} is attained for that objective
  \item ``Dyn'': the dynamism of the new cut is too high
\end{itemize}
The following block of rows looks at the percent of failures for each class of objectives:\ ``all ones'' ($w = e$), ``post-GMIC'' (step~\ref{step:VPC1:postGMIC} of Algorithm~\ref{alg:VPC1}), and ``DB'' (step~\ref{step:VPC1:DB} of Algorithm~\ref{alg:VPC1}).
The last block of rows looks at the average number of objectives required to generate each distinct VPC, as well as the average number of seconds taken per objective and per cut.

This table shows that failures become more frequent when using disjunctions with more terms, with average failure rate increasing from 29\% to 34\%.
The primary reason for this is that more objectives lead to cuts that were previously generated, leading to, on average, up to 57\% of the failures.
The cause is that there are more unsuccessful objectives being tried for the ``DB'' class of objectives.
The last set of rows of the table also show that cut generation can be extremely costly for stronger disjunctions,
which could be mitigated by reducing the failure rate.
Table~\ref{app:tab:obj-and-time-best}
provides objective statistics just for the best run per instance (across all partial tree sizes).

Lastly, we look at which classes of objective functions are more likely to lead to active cuts (after the addition of all cuts),
as a different measure of the effect of our cuts.
Table~\ref{tab:activity} gives averages for which cuts are active at the optimal solution after adding both GMICs and VPCs to \ref{P}, for GMICs and VPCs, as well as individually based on the objective producing each VPC.
The first row is the percent of GMICs that are active.
The second row is the percent of VPCs that are active, averaged across those instances per each column for which VPCs were generated.
The next rows come in pairs and give the average percent of cuts that come from the four subclasses of VPCs within our procedure, 
as well as the average percent of these cuts that are active (across instances for which there exist cuts from that class).
The first class concerns a set of (rarely encountered) cuts that is added in our procedure, which we call ``one-sided''.
In the process of generating VPCs, while selecting variables for strong branching, 
we occasionally detect that one of the two possible branches is infeasible.
In this case, for a variable $x_k$, $k \in \intvars$,
we generate the ``one-sided cut''
	$x_k \le \floor{\lpopt_k}$ or $x_k \ge \ceil{\lpopt_k}$.
These cuts are generated for 5 of the \numgapinst{} instances, with a total of only 6 cuts.

\begin{table}[h]
  \captionsetup{font=small}
\small
\centering
\caption{Statistics about when generated cuts are active, broken down by partial tree size.}
\label{tab:activity}
\begin{tabular}{@{}l*{6}{S[table-format=3.2,table-auto-round,table-number-alignment=center]}@{}}
\toprule
{} & {\makecell[c]{{V+G}\\ {(2)}}} & {\makecell[c]{{V+G}\\ {(4)}}} & {\makecell[c]{{V+G}\\ {(8)}}} & {\makecell[c]{{V+G}\\ {(16)}}} & {\makecell[c]{{V+G}\\ {(32)}}} & {\makecell[c]{{V+G}\\ {(64)}}} \\
\midrule
\% active GMIC & 44.088577 & 43.515863 & 42.170715 & 41.598561 & 40.879039 & 40.675613 \\
\% active VPC & 30.196454 & 30.591121 & 31.897898 & 35.499314 & 34.478032 & 33.401584 \\
\% cuts one-sided & 0.767677 & 0.722595 & 0.741979 & 1.107456 & 0.815851 & 0.792411 \\
\% active one-sided & 100.000000 & 100.000000 & 100.000000 & 100.000000 & 100.000000 & 100.000000 \\
\% cuts all ones & 11.822462 & 6.723898 & 7.806594 & 9.095018 & 8.404258 & 8.995381 \\
\% active all ones & 91.447368 & 84.782609 & 81.818182 & 77.018634 & 79.629630 & 79.268293 \\
\% cuts post-GMIC opt & 2.615979 & 3.397466 & 2.420045 & 2.186830 & 1.703320 & 2.739358 \\
\% active post-GMIC opt & 85.897436 & 71.428571 & 67.924528 & 62.745098 & 63.829787 & 61.904762 \\
\% cuts DB & 84.793883 & 89.156040 & 89.031382 & 87.610696 & 89.076571 & 87.472850 \\
\% active DB & 63.745545 & 58.323106 & 56.359239 & 51.850133 & 49.306592 & 43.000601 \\
\bottomrule
\end{tabular}
\end{table}

Table~\ref{tab:activity} shows that the proportion of active VPCs somewhat increases while the percent of active GMICs decreases as VPCs from stronger disjunctions are used.
Aside from the one-sided cuts, which are always active in our results, the objectives ``all ones'' and ``post-GMIC opt'' lead to cuts that are frequently active, though these objectives yield at most two cuts in total per instance.
Though the ``DB'' class of objectives leads to a smaller percentage of active cuts, it is the source for the majority of the cuts that we generate.

\clearpage
{
\setlength{\tabcolsep}{6pt} %
\renewcommand{\arraystretch}{0.98} %
\setlength\LTleft{-20.27893pt}%
\centering
\begin{tabularx}{\textwidth}{@{}
        l
        *{3}{S[table-format=4.0,table-auto-round,table-number-alignment=center]}
        *{1}{S[table-format=2.1,table-auto-round,table-number-alignment=center]}
        *{1}{S[table-format=4.1,table-auto-round,table-number-alignment=center]}
        *{1}{S[table-format=2.1,table-auto-round,table-number-alignment=center]}
        *{1}{S[table-format=3.1,table-auto-round,table-number-alignment=center]}
        @{}}
\caption{Information about objectives and time to generate cuts corresponding to the parameters leading to the best gap closed per instance, as reported in Table~\ref{app:tab:gap-closed}.}
\label{app:tab:obj-and-time-best}\\
\toprule
{} & \multicolumn{4}{c}{Objectives} & \multicolumn{3}{c}{Time (s)} \\
    \cmidrule(lr){2-5} \cmidrule(l){6-8}
{Instance} & {Obj} & {Succ} & {Fails} & {\% fails} & {Total} & {(s)/obj} & {(s)/cut} \\
\midrule
\endfirsthead
\toprule
{} & \multicolumn{4}{c}{Objectives} & \multicolumn{3}{c}{Time (s)} \\
    \cmidrule(lr){2-5} \cmidrule(l){6-8}
{Instance} & {Obj} & {Succ} & {Fails} & {\% fails} & {Total} & {(s)/obj} & {(s)/cut} \\
\midrule
\endhead

\endfoot

\endlastfoot
10teams & 214.0 & 16.0 & 198.0 & 92.5233644859813 & 3604.56 & 16.84373831775701 & 225.285 \\
23588 & 76.0 & 75.0 & 1.0 & 1.3157894736842106 & 28.6 & 0.3763157894736842 & 0.38133333333333336 \\
30n20b8 & 193.0 & 190.0 & 3.0 & 1.5544041450777202 & 39.99 & 0.2072020725388601 & 0.2104736842105263 \\
50v-10 & 30.0 & 29.0 & 1.0 & 3.3333333333333335 & 4.45 & 0.14833333333333334 & 0.15344827586206897 \\
a1c1s1 & 3.0 & 3.0 & 0.0 & 0.0 & 0.54 & 0.18000000000000002 & 0.18000000000000002 \\
a2c1s1 & 20.0 & 18.0 & 2.0 & 10.0 & 0.16 & 0.008 & 0.008888888888888889 \\
aflow30a & 28.0 & 25.0 & 3.0 & 10.714285714285714 & 6.39 & 0.2282142857142857 & 0.2556 \\
aflow40b & 38.0 & 36.0 & 2.0 & 5.2631578947368425 & 82.01 & 2.158157894736842 & 2.2780555555555555 \\
\bottomrule
\end{tabularx}
} %

\section{Example of invalid cuts from a point-ray collection}
\label{app:counterex-vpc-neighbors}

This example shows that the using as the point-ray collection the optimal points $p^t$ on each term $t \in \mathcal{T}$ along with the neighbors of $p^t$ may lead to the generation of invalid cuts from the associated \eqref{PRLP}.
\begin{align*}
    \max_{x_1,x_2,x_3}\quad &\phantom{-}x_3\\
    &\begin{aligned}
      -&x_3 \le -{1}/{2} \\
      -&({7}/{4})x_1 + 5 x_2 - 2 x_3 \le 1 \\
      -&x_1 - 5 x_2 + 2 x_3 \le -1 \\
      -&x_1 - ({20}/{3}) x_2 + ({7}/{3}) x_3 \le - {3}/{2} \\
      &x_1 - x_2 + ({3}/{2}) x_3 \le {3}/{2} \\
      &2x_1 - x_2 + 3 x_3 \le {7}/{2} \\
      -&x_1 + 4 x_2 + 2x_3 \le {7}/{2} \\
      -&x_1 + 4 x_2 \le 2 \\
      &x_1, x_2, x_3 \in [0,1] \\
      &x_1 \phantom{{}\le{}} \mbox{integer}
    \end{aligned}
  \end{align*}

Let $P$ denote the feasible region of the linear relaxation of the above integer program.
Figure~\ref{app:fig:counterex-vpc-neighbors} shows the feasible region of $P$.
Point~$q^1$ denotes the optimal solution to the linear programming relaxation.
The vertices of $P$ are:
  \begin{align*}
    q^1 &= \{1/2,1/2,1\} \\
    q^2 &= \{1, 3/4, 3/4\} \\
    q^3 &= \{1, 3/4, 1/2\} \\
    q^4 &= \{1, 1/2, 2/3\} \\
    q^5 &= \{1, 1/4, 1/2\} \\
    q^6 &= \{0, 1/2, 3/4\} \\
    q^7 &= \{0, 2/5, 1/2\}.
  \end{align*}

We use as the valid disjunction the elementary split on $x_1$.
If we solve $\max\{x_3 \suchthat x \in P,\ x_1 = 0\}$, the maximum is achieved by $q^6$.
Solving $\max\{x_3 \suchthat x \in P,\ x_1 = 1\}$, the maximum is achieved by $q^2$.
Consider using $q^2$ and $q^6$ and their neighbors as the collection of points given to \eqref{PRLP}.
Thus $\pointset = \{q^2, q^3, q^4,  q^6, q^7\}$ and $\rayset = \emptyset$.

One cut that can be obtained from this set of points is $\bar{\alpha}^\T x \ge 1$ where
$\bar{\alpha} = (-1/6, 5, -2)$,
which goes through $q^6$, $q^7$, and $q^4$, leaving $q^2$ and $q^3$ on the feasible side, 
and cutting off not only $q^1$, but also $q^5$.
Indeed, $\bar{\alpha}^\T q^1 = -1/12 + 5/2 - 2 = 5/12 < 1$, while $\bar{\alpha}^\T q^4 =
\bar{\alpha}^\T q^7 = \bar{\alpha}^\T q^6 = 1$, and $\bar{\alpha}^{\T} q^2 = 25/12 > 1$ and $\bar{\alpha}^\T
q^3 = 31/12 > 1$.
However, as seen in Figure~\ref{app:fig:counterex-vpc-neighbors}, 
$\bar{\alpha}^\T q^5 = 1/12 < 1$, so that point of $P \cap \{x \suchthat x_1 = 1\}$ is cut off,
making the cut invalid.

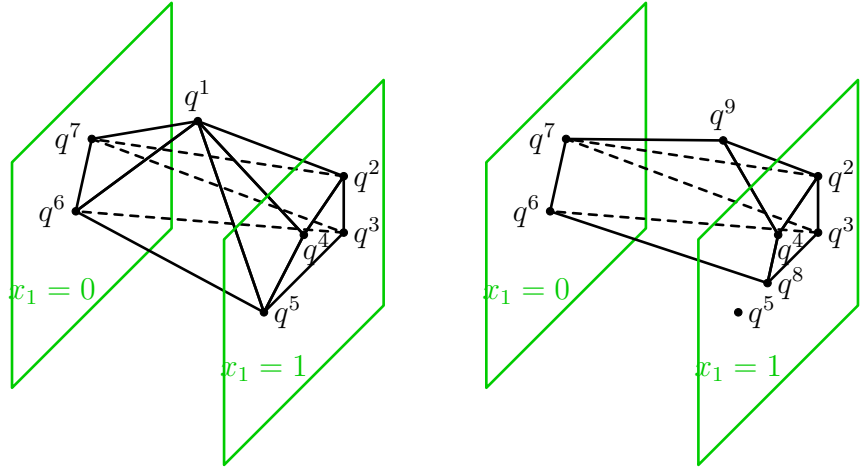
\begin{figure} %
\centering
  \begin{subfigure}{.45\textwidth}
   \centering
    \begin{tikzpicture}[line join=round,line cap=round,x={(\xX cm, \xY cm)},y={(\yX cm, \yY cm)},z={(\zX cm,\zY cm)},>=stealth,scale=3] 
      \pgfmathsetmacro{\scaleRays}{1}
      \tikzstyle{ray_style} = [->,>= stealth,shorten >=2pt,line width=1.5pt]
      \tikzstyle{point_style} = [circle, inner sep=1pt, fill=black]
     
      \coordinate (p1) at (1/2,1/2,1);
      \coordinate (p2) at (1,3/4,3/4);
      \coordinate (p3) at (1,3/4,1/2);
      \coordinate (p4) at (1,1/2,2/3);
      \coordinate (p5) at (1,1/4,1/2);
      \coordinate (p6) at (0,2/5,1/2);
      \coordinate (p7) at (0,1/2,3/4);
      
      \coordinate (p8) at (1,13/30,1/2);
      \coordinate (p9) at (27/40,47/80,73/80);
      
      \draw [fill opacity=0.5,fill=none, color=green!80!black, line width=1pt] (0,0,0) -- (0,1,0) -- (0,1,1) -- (0,0,1) -- cycle;
      
      \draw [line width=1pt] (p1) -- (p2) -- (p4) -- cycle;
      \draw [line width=1pt] (p1) -- (p4) -- (p5) -- cycle;
      \draw [line width=1pt] (p1) -- (p5) -- (p6) -- cycle;
      \draw [line width=1pt] (p1) -- (p6) -- (p7) -- cycle;
      \draw [line width=1pt] (p2) -- (p3) -- (p5) -- (p4) -- cycle;

      \draw [line width=1pt,dashed] (p2) -- (p7) -- (p3) -- (p6);
      
      \node [draw, point_style, fill=black, label={[label distance=-2pt]90: $q^1$}] at (p1) {};
      \node [draw, point_style, fill=black, label={[label distance=-2pt]0: $q^2$}] at (p2) {};
      \node [draw, point_style, fill=black, label={[label distance=-2pt]0: $q^3$}] at (p3) {};
      \node [draw, point_style, fill=black, label={[label distance=-8pt]-45: $q^4$}] at (p4) {};
      \node [draw, point_style, fill=black, label={[label distance=-2pt]0: $q^5$}] at (p5) {};
      \node [draw, point_style, fill=black, label={[label distance=-2pt]180: $q^6$}] at (p6) {};
      \node [draw, point_style, fill=black, label={[label distance=-2pt]180: $q^7$}] at (p7) {};
      \node [circle, inner sep=1.5pt, fill=none] at (0,.25,.25) {\color{green!80!black} $x_1 = 0$};
      \node [circle, inner sep=1.5pt, fill=none] at (1,.25,0.25) {\color{green!80!black} $x_1 = 1$};
      \draw [fill opacity=0.5,fill=none, color=green!80!black, line width=1pt] (1,0,0) -- (1,1,0) -- (1,1,1) -- (1,0,1) -- cycle;
     \end{tikzpicture}
    \end{subfigure}
    \begin{subfigure}{.45\textwidth}
      \centering
    \begin{tikzpicture}[line join=round,line cap=round,x={(\xX cm, \xY cm)},y={(\yX cm, \yY cm)},z={(\zX cm,\zY cm)},>=stealth,scale=3] 
      \pgfmathsetmacro{\scaleRays}{1}
      \tikzstyle{ray_style} = [->,>= stealth,shorten >=2pt,line width=1.5pt]
      \tikzstyle{point_style} = [circle, inner sep=1pt, fill=black]
     
      \coordinate (p1) at (1/2,1/2,1);
      \coordinate (p2) at (1,3/4,3/4);
      \coordinate (p3) at (1,3/4,1/2);
      \coordinate (p4) at (1,1/2,2/3);
      \coordinate (p5) at (1,1/4,1/2);
      \coordinate (p6) at (0,2/5,1/2);
      \coordinate (p7) at (0,1/2,3/4);
      
      \coordinate (p8) at (1,13/30,1/2);
      \coordinate (p9) at (27/40,47/80,73/80);
      
      \draw [fill opacity=0.5,fill=none, color=green!80!black, line width=1pt] (0,0,0) -- (0,1,0) -- (0,1,1) -- (0,0,1) -- cycle;
      
      \draw [line width=1pt] (p9) -- (p2) -- (p4) -- cycle;
      \draw [line width=1pt] (p9) -- (p4) -- (p8) -- (p6) -- (p7) -- cycle;
      \draw [line width=1pt] (p2) -- (p3) -- (p8) -- (p4) -- cycle;

      \draw [line width=1pt,dashed] (p2) -- (p7) -- (p3) -- (p6);
      
      \node [draw, point_style, fill=black, label={[label distance=-2pt]0: $q^2$}] at (p2) {};
      \node [draw, point_style, fill=black, label={[label distance=-2pt]0: $q^3$}] at (p3) {};
      \node [draw, point_style, fill=black, label={[label distance=-8pt]-45: $q^4$}] at (p4) {};
      \node [draw, point_style, fill=black, label={[label distance=-2pt]0: $q^5$}] at (p5) {};
      \node [draw, point_style, fill=black, label={[label distance=-2pt]180: $q^6$}] at (p6) {};
      \node [draw, point_style, fill=black, label={[label distance=-2pt]180: $q^7$}] at (p7) {};
      \node [draw, point_style, fill=black, label={[label distance=-2pt]0: $q^8$}] at (p8) {};
      \node [draw, point_style, fill=black, label={[label distance=-2pt]90: $q^9$}] at (p9) {};
      \node [circle, inner sep=1.5pt, fill=none] at (0,.25,.25) {\color{green!80!black} $x_1 = 0$};
      \node [circle, inner sep=1.5pt, fill=none] at (1,.25,0.25) {\color{green!80!black} $x_1 = 1$};
      \draw [fill opacity=0.5,fill=none, color=green!80!black, line width=1pt] (1,0,0) -- (1,1,0) -- (1,1,1) -- (1,0,1) -- cycle;
    \end{tikzpicture}
    \end{subfigure}
\caption{The LP polytope \ref{P} for the counter-example showing an optimal point on each disjunctive term and its neighbors as the point-ray collection may lead to invalid cuts.}\label{app:fig:counterex-vpc-neighbors}
\end{figure}

There is a relatively simple resolution for the above counterexample.
If we require the generated cuts to be tight
at the optimal solutions on each facet, $q^6$ and $q^2$, then all generated cuts
will be valid for $\conv(P_D)$.
We state this in Theorem~\ref{app:thm:proper-neighbors}.

\begin{theorem}
\label{app:thm:proper-neighbors}
  For each $t \in \mathcal{T}$, set $\pointset^t$ to be an optimal solution $p^t$ to $\argmin_x \{c^\T x \suchthat x \in \Pt\}$ as well as one point on each of the edges emanating from $p^t$ within $\Pt$.
  Let $\pointset \defeq \cup_{t \in \mathcal{T}} \pointset^t$ and $\rayset \defeq \emptyset$.
  Any feasible solution to \eqref{PRLP} formulated from these points and amended with the condition $\alpha^\T p^t = \beta$ for all $t \in \mathcal{T}$ yields a valid cut for $\PI$.
\end{theorem}
\begin{proof}
  Let $(\alpha,\beta)$ be a feasible solution to \eqref{PRLP} such that $\alpha^\T p^t = \beta$ for all $t \in \mathcal{T}$.
  Since $\alpha^\T p \ge \beta$ for all $p \in \pointset^t$, it holds that $\alpha^\T (p - p^t) \ge 0$.
  This means that, for each $t \in \mathcal{T}$, the generated cut is valid for the cone with apex at $p^t$ and rays going through each of the points $p \in \pointset^t$.
  By convexity, this cone is a relaxation of $\Pt$.
  It follows, by Corollary~\ref{cor:relaxedPRcollection}, that the cut is valid for $\PI$.
\end{proof}

\section{Sample partial branch-and-bound tree}
\label{app:sec:sample-tree}

The computational experiments in the paper use a partial branch-and-bound tree as the source of the disjunction for cut generation.
The partial trees are produced from the branching strategy described in Section~\ref{sec:computational-setup}.
In particular, there may exist nodes of the tree that are pruned, and the tree may be very asymmetric.
We illustrate this with one sample tree, shown in Figure~\ref{app:fig:sample-tree}, constructed from the instance \texttt{bm23} and terminated after finding 64 leaf nodes.
This tree includes two pruned nodes (the leftmost leaf node, from branching on $x_{15}$, and an adjacent leaf on the same level, from branching on $x_{8}$).

\begin{figure}[ht]
  \centering
  \includegraphics[width=\textwidth]{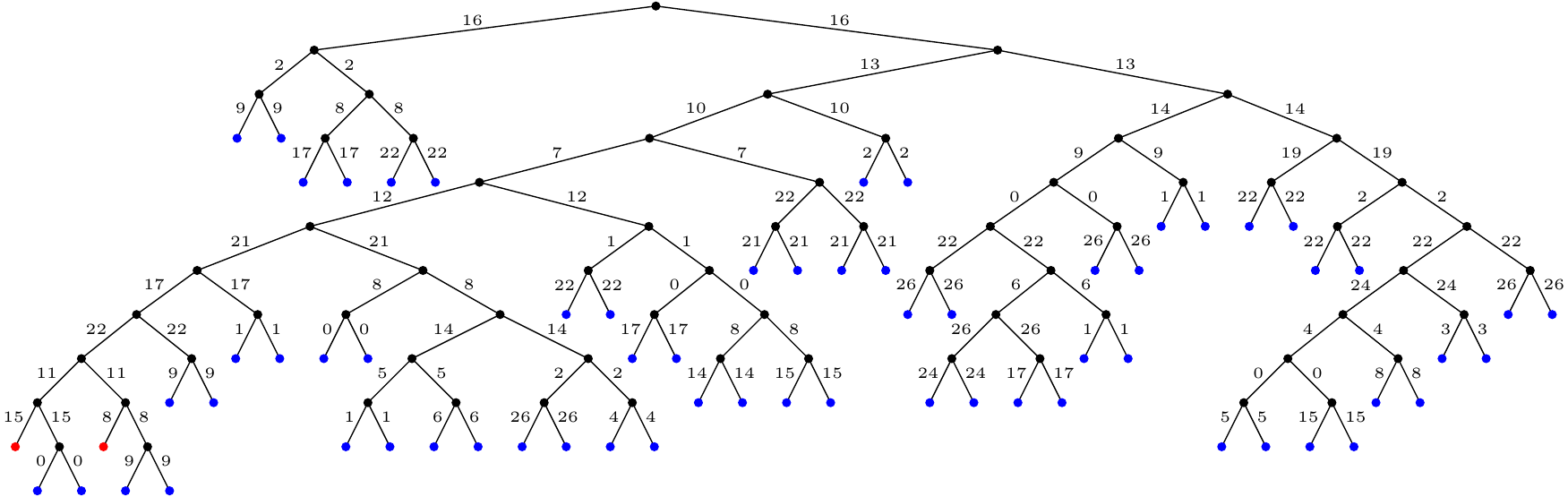}
  \caption{Partial branch-and-bound tree with 64 leaf nodes for instance \texttt{bm23}.}
  \label{app:fig:sample-tree}
\end{figure}

\section{Computational setup: additional details}
\label{app:additional-setup}

This section expands on Section~\ref{sec:computational-setup} with additional details on the tolerances and settings used in our VPC implementation.

\subsection{Evaluation}

We evaluate cuts from two different perspectives: strength and effect on branch and bound.
The strength of the cuts is assessed by the \emph{percent integrality (root) gap closed} by one round of VPCs.
Let $x_I$ denote an optimal solution to \eqref{IP}, and let ${x}'$ be an optimal solution to \eqref{LP} after a set of cuts have been added.
We measure the quantity
  \[
    \text{\% integrality gap closed} \defeq 100 \times \frac{c^\T {x}' - c^\T \lpopt}{c^\T x_I - c^\T \lpopt}.
  \]
As a baseline, we also report the percent gap closed by adding one round of GMICs, as well as the percent gap closed by using both VPCs and GMICs together.
In addition, we report, both with and without the use of VPCs, the root gap closed by \Gurobi{} after one cut pass and after the last round of cuts is added at the root.
The effect on branch and bound is measured by the time \Gurobi{} takes to solve the problem with VPCs added as user cuts;
this is compared to the time taken without VPCs.

\looseness=-1
Note that there are two sources of cut strength, one from the disjunction from which the cut is produced, and one from \emph{modularization}
applied to variables other than those involved in the disjunction~\cite{BalJer80}.
In contrast to GMICs, VPCs do not use this second approach, as applying the technique to VPCs is more complex and requires the use of additional information about the cut when it is derived from a disjunction that is not simple.

\subsection{Generating a partial branch-and-bound tree}
The partial branch-and-bound tree is generated by the node, variable, and branch selection rules that follow, which are the defaults for \texttt{Cbc}.
Node selection is roughly by the \emph{best-first} or \emph{best-bound} rule, in which the next node to explore will be the one with the minimum objective value (though the number of fractional variables at each node is also considered, it is to a much lesser extent).
Variable selection utilizes the outcome of strong branching on up to five fractional variables at each node.
Between the two possible children of the node when branching on the variable that is selected, the direction is chosen by a similar rule to the node selection criteria, i.e., typically in the direction of the child with a lower optimal value.

\subsection{Instance preprocessing}
Every instance is first preprocessed by \Gurobi{}'s presolve.
This procedure is used in order to improve the fairness of the testing environment.
It allows VPCs to be generated from the same version of the instance that would be used internally by the branch-and-bound solver.
This is also a reason for turning presolve off during the subsequent branch and bound tests, as one round of presolve has already been applied.
At the same time, for some instances, preprocessing closes a significant portion of the integrality gap, which could make the process of finding strong cuts more difficult.
For reproducibility, one must be aware that not only might \Gurobi{}'s presolve algorithms change with a new version (release of the software), but also they depend on the random seed given to the solver.
We do not experiment with this latter variability (we presolve with the random seed 628 only).

\subsection{ Setting up \texorpdfstring{\eqref{PRLP}}{PRLP} }
When constructing the associated \eqref{PRLP} (in the nonbasic space defined by the cobasis at $\lpopt$), we remove all duplicate rows.
In addition, any rows that are actually bounds on the $\alpha$ variables are removed as explicit constraints and kept as bounds instead.
Henceforth, we assume that $(\pointset,\rayset)$ has no duplicates.
We proceed with generating cuts from a given \eqref{PRLP} if it is feasible and solves to optimality within a minute when using no objective, i.e., just the feasibility problem.

\subsection{Normalization}
\label{sec:setup-normalization}
As mentioned, we formulate the PRLP in the nonbasic space with respect to $\lpopt$ (in which $\lpopt$ is represented as the origin), and we normalize \eqref{PRLP} to have $\beta = 1$, which restricts the set of obtainable inequalities to those that cut $\lpopt$.
In practice, we will actually use some scaled positive constant determined by the input, as $\beta = 1$ could lead to numerical issues, due to how it causes the cut coefficients scale.
To illustrate this, suppose $c^\T p^t \ge 10^{8}$ for all $t \in \mathcal{T}$.
Then, if $\beta = 1$, $\alpha = c / 10^{8}$ is a feasible solution to \eqref{PRLP}.
However, coefficients less than $10^{-7}$ are often regarded as zero by solvers, so we may end up generating cuts incorrectly with improper scaling.

\subsection{Cut processing and objective failures}
\label{sec:cut-processing}
Not every objective function we try for \eqref{PRLP} leads to a new cut.
As we discussed in Section~\ref{sec:obj-theory}, \eqref{PRLP} can be unbounded or lead to a duplicate cut.
Other failures are imposed by conditions that we set.
If the time to solve \eqref{PRLP} for an objective is greater than 5 seconds, we abandon the objective.
If we do obtain a solution, since \eqref{PRLP} is formulated in the nonbasic space, we first convert the cut to the structural space, yielding a cut $\gamma^\T x \ge \gamma_0$.
Next, we remove small coefficients: for $j \in [n]$, if $\abs{\gamma_j} < \eps$, we ignore the coefficient, and if $\eps \le \abs{\gamma_j} < \eps_{\text{coeff}}$, we replace $x_j$ by either its lower or upper bound and adjust $\gamma_0$.
In our experiments, $\eps = 10^{-7}$, and $\eps_{\text{coeff}} = 10^{-5}$.
Assume that $\gamma^\T x \ge \gamma_0$ has been processed in this way.
The cut is rejected if it is a duplicate of or dominated by a previously generated VPC.
It will also be rejected if its dynamism ($\max_{j \in [n]} \abs{\gamma_j} / \min_{j \in [n]} \{\abs{\gamma_j} \suchthat \gamma_j \ne 0\}$) is higher than $10^8$.
In addition, if there exists some previously generated VPC $\alpha^\T x \ge \beta$ that is nearly parallel to $\gamma^\T x \ge \gamma_0$, i.e., if $\alpha \cdot \gamma / (\norm{\alpha} \cdot \norm{\gamma}) < \epsortho$ ($\epsortho = 10^{-3}$ in our setup), then we keep only one of these two cuts (the one that separates $\lpopt$ by a greater Euclidean distance, or if these are equal, the sparser cut).

We solve \eqref{PRLP} until we exhaust all objectives or reach one of these stopping criteria:
  (1) Numerical difficulties are encountered while solving \eqref{PRLP}. %
  (2) The time limit for cut generation is reached.
  (3) The cut limit is reached.
  (4) The failure limit is reached.

An example of numerical difficulties we have encountered is when \eqref{PRLP} solves to optimality for one objective but is deemed primal infeasible for another.
The time limit for cut generation is 900 seconds (the time to set up the partial tree and build \eqref{PRLP} is not counted against this).
The cut limit is equal to the number $k$ of fractional variables at the LP optimal solution.

The failure limit we use comes from some nontrivial experimentation.
It varies based on several parameters: there are different maximum failure rates depending on whether ``few'' or ``many'' cuts have been generated, and whether ``many'' objective functions have been attempted.
Let $\phi_{\text{few\_cuts}} \defas 0.95$, $\phi_{\text{many\_cuts}} \defas 0.90$, $\phi_{\text{many\_obj}} \defas 0.80$.
We define ``few'' cuts as $n_{\text{few\_cuts}} \defas 1$, ``many'' cuts as $n_{\text{many\_cuts}} \defas \ceil{k/4}$, and ``many'' objectives as 
  \[
    n_{\text{many\_obj}} 
      \defas \max\{ \ceil{n_{\text{few\_cuts}} / (1 - \phi_{\text{few\_cuts}})}, \ceil{n_{\text{many\_cuts}} / (1-\phi_{\text{many\_cuts}})} \}.
  \]
Hence, the default for $n_{\text{many\_obj}}$ is $\max \{ 20, 10\ceil{k/4}\}$.
After each cut, we test whether the current failure ratio (number of unsuccessful objectives as a proportion of the total number of objectives attempted) is greater than the appropriate threshold ($\phi_{\text{few\_cuts}}$, $\phi_{\text{many\_cuts}}$, or $\phi_{\text{many\_obj}}$);
if it is, then we return that we have reached the failure limit.
We also say we reached the failure limit if the first $\ceil{n_{\text{few\_cuts}} / (1 - \phi_{\text{few\_cuts}})}$ objectives all lead to failures;
this is often an indicator of numerical issues with the instance.
As more cuts are generated and more objectives are tried, the acceptable failure rate decreases in this setup, as there is likely to be diminishing marginal benefit for additional cuts and we wish to avoid spending excessive amounts of time attempting to generate cuts unsuccessfully.

\section{Instance selection}
\label{sec:instance-selection}
Instances were selected by the following criteria:
\begin{enumerate}[label={Criterion~\arabic*.},ref={{\arabic*}},leftmargin=*,]
	\item \label{selection-criterion:ip-opt-known}
		The IP optimal value is known and is not equal to the LP optimal value.
	\item \label{selection-criterion:max-instance-size}
		There are at most 5,000 rows and 5,000 columns in the presolved instance.
	\item \label{selection-criterion:partial-tree-does-not-find-opt}
		No partial branch-and-bound tree used finds an IP optimal solution.
	\item \label{selection-criterion:cuts-are-generated}
		For at least one partial branch-and-bound tree, 
  			\begin{enumerate*}[label={(\alph*)},ref={{4(\alph*)}}]
		  		\item \label{selection-criterion:cuts-are-generated:not_lp=dlb=dub}
  					either the disjunctive lower bound, $c^\T \ref{p*}$, is strictly less than the maximum objective value of any leaf node, 
		  			or it is strictly greater than $c^\T \lpopt$,
		  		\item \label{selection-criterion:cuts-are-generated:PRLP-primal-feasible}
		  			the corresponding \eqref{PRLP} is primal feasible,
		  		\item \label{selection-criterion:cuts-are-generated:PRLP-time-limit}
		  			it is proved feasible within a one minute time limit,
		  			and
		  		\item \label{selection-criterion:cuts-are-generated:cuts-are-generated}
		  			at least one VPC is added to cut pool after processing according to Section~\ref{sec:cut-processing}.
	  		\end{enumerate*}
\end{enumerate}
Criterion~\ref{selection-criterion:cuts-are-generated}
is put in place because infeasible instantiations of \eqref{PRLP} typically occur when 
  $\max_{p \in \pointset} c^\T p = c^\T \ref{p*} = c^\T \lpopt$.
When we report gap closed, we will also filter by another condition, that $c^\T \ref{p*} > c^\T \lpopt$ from the 64-leaf partial tree,
as instances in which the disjunctive lower bound and optimal value of the LP relaxation coincide are not good candidates for evaluation by gap closed.
The branch and bound results only include those instances that are solved in under an hour with \Gurobi{} (either with or without VPCs).

We modify the instances \texttt{stein27} and \texttt{stein45} from the versions in MIPLIB; in particular, we remove a cardinality constraint enforcing a lower bound for the objective value, which is not present in the original formulation of the problem~\cite{FulNemTro74_set-covering}.

In total, there are 1,458 instances across the MIPLIB, CORAL, and NEOS sets.
Many of these are ultimately not considered.
Initially, we eliminate
	42 instances that have indicator constraints from MIPLIB 2017, and
	81 nonindicator instances that are infeasible or unbounded.
Then, from the 1,335 instances that are left, we eliminate
	349 instances with unknown IP optimal value.
We then remove
	45 instances that have no objective (feasibility instance).
From the 941 instances remaining,
	295 instances have either more than 20,000 rows or more than 20,000 columns,
so these are never preprocessed and removed from consideration, as we deemed it unlikely that the resulting presolved instance would satisfy Criterion~\ref{selection-criterion:max-instance-size}.
From the remaining 646 instances,
	53 are eliminated because of less than 0.001 integrality gap after presolve,
and
	another 153 are eliminated due to having too many rows or columns after presolve,
Further, 4 more instances (bley\_xs2, control20-5-10-5, ej, gen-ip016) encountered numerical issues.
For example, for bley\_xs2, Gurobi presolve declares the instance unbounded,
and for control20-5-10-5, Gurobi will declare the instance unbounded or infeasible depending on seed.
The above exclusions count neos-3661949-lesse, which appears to have no integrality gap after presolve, though the issue may be numerical as the optimal value found after presolve is 689,000,000, compared to the listed value 688,995,225.
This leaves 436 instances.

Of these 436 instances, Table~\ref{app:tab:discarded-instances} lists the 104 that were removed from consideration and the reason for removal.
As the table shows, one of the most common reasons for discarding an instance was that no cuts were generated for that instance due either to $\max_{p \in \pointset} c^\T p = c^\T \lpopt$ (the optimal value over the best leaf node equals the optimal value of the LP) or \eqref{PRLP} being primal infeasible.
The situations are treated together because the former anyway typically results in \eqref{PRLP} being infeasible,
though it also implies that our metric of gap closed is not reasonable for that instance.
There is potential for these instances to either generate inequalities that do not cut away $\lpopt$ or to try to find a different partial branch-and-bound tree,
but that has been beyond the scope of our investigation.
There are an additional 50 instances for which branch and bound results are not shown, because these instances were unable to be solved within an hour by \Gurobi{} (neither with or without cuts, averaged across all random seeds).
Finally, we comment that results were not obtained with instance neos-3734794-moppy because of a slight mismatch between the objective values computed for the disjunctive terms within Cbc versus outside of Cbc'
this issue could be resolved with loosening the associated tolerance, but we refrained from this manual adjustment for these experiments.

{\small
\centering
\renewcommand{\arraystretch}{0.99} %
\begin{tabularx}{\linewidth}{@{}*{2}{l}*{6}{c}@{}}
\caption{Instances not considered along with violated selection criteria.}
\label{app:tab:discarded-instances}\\
\toprule
& & \multicolumn{6}{c}{Number of leaves} \\*
\cmidrule{3-8}
Instance & Set & 2 & 4 & 8 & 16 & 32 & 64
\\* \midrule
\endfirsthead
\toprule
& & \multicolumn{6}{c}{Number of leaves} \\*
\cmidrule{3-8}
Instance & Set & 2 & 4 & 8 & 16 & 32 & 64                                                                                     
\\* \midrule
\endhead
\bottomrule
\endfoot
\endlastfoot
22433 & miplib2017 &  &  & \ref{selection-criterion:partial-tree-does-not-find-opt} & \ref{selection-criterion:partial-tree-does-not-find-opt} & \ref{selection-criterion:partial-tree-does-not-find-opt} & \ref{selection-criterion:partial-tree-does-not-find-opt} \\
air01 & miplib2 & \ref{selection-criterion:partial-tree-does-not-find-opt} & \ref{selection-criterion:partial-tree-does-not-find-opt} & \ref{selection-criterion:partial-tree-does-not-find-opt} & \ref{selection-criterion:partial-tree-does-not-find-opt} & \ref{selection-criterion:partial-tree-does-not-find-opt} & \ref{selection-criterion:partial-tree-does-not-find-opt} \\
app1-1 & miplib2017 &  &  &  &  & \ref{selection-criterion:partial-tree-does-not-find-opt} & \ref{selection-criterion:partial-tree-does-not-find-opt} \\
b-ball & miplib2017 & \ref{selection-criterion:cuts-are-generated:not_lp=dlb=dub} & \ref{selection-criterion:cuts-are-generated:not_lp=dlb=dub} & \ref{selection-criterion:cuts-are-generated:not_lp=dlb=dub} & \ref{selection-criterion:cuts-are-generated:PRLP-primal-feasible} & \ref{selection-criterion:cuts-are-generated:PRLP-primal-feasible} & \ref{selection-criterion:cuts-are-generated:PRLP-primal-feasible} \\
bppc4-08 & miplib2017 & \ref{selection-criterion:cuts-are-generated:not_lp=dlb=dub} & \ref{selection-criterion:cuts-are-generated:not_lp=dlb=dub} & \ref{selection-criterion:cuts-are-generated:PRLP-primal-feasible} & \ref{selection-criterion:cuts-are-generated:PRLP-primal-feasible} & \ref{selection-criterion:cuts-are-generated:PRLP-primal-feasible} & \ref{selection-criterion:cuts-are-generated:PRLP-time-limit} \\
\bottomrule
\end{tabularx}
}

\section{Additional results for partial branch-and-bound tree experiments}
\label{app:sec:extra-tables}

This section contains additional details for the experimental results, left out of the main text due to length.

Table~\ref{app:tab:gap-closed} shows the number of rows and columns for each of the \numgapinst{} instances used in the gap closed experiments, the number of cuts produced to yield the ``best'' VPC objective value (across partial tree sizes), and the percent gap closed for each instance.
Columns~2 and 3 give, for each instance, the number of constraints and variables after preprocessing.
The next two columns show the number of cuts generated.
Column~6 is the percent gap closed by GMICs when they are added to the LP relaxation.
Column~7 is the percent gap closed as implied by the disjunctive lower bound from the partial tree with 64 leaf nodes.
Column~8 is the percent gap closed by VPCs.
Column~9 is the best percent gap closed for each instance between the value with GMICs alone in column ``G'' and the value with VPCs alone in column ``V''.
Column~10 is the percent gap closed when GMICs and VPCs are used together.
Columns~11 and 12 show the percent gap closed by \Gurobi{} cuts from one round at the root, first without and then with VPCs added as user cuts.
Columns~13 and 14 show the same, but after the last round of cuts at the root.
The values in columns~11 and 13 are the maximum percent gap closed across seven random seeds.
The last two rows give the average and number of wins, reproducing the summary data in Table~\ref{tab:gap-closed-summary}.

Table~\ref{app:tab:bb} provides the time and number of nodes taken by each instance for the partial tree size per instance that led to the best outcome for Gurobi with VPCs across the six partial tree sizes tested.
Columns~2 and 3 show the number of terms and number of cuts that led to the time for ``V''.
The table is in increasing order by column~5, ``V''.

{\clearpage\footnotesize
\sisetup{
    table-alignment-mode = format,
    table-number-alignment = center,
    table-format = 3.2,
}
\setlength\LTleft{-81.97917pt}%
\setlength{\tabcolsep}{2.5pt} %
\begin{tabularx}{1.2\textwidth}{
	@{}l*{2}{S[table-format=4.0,table-auto-round,table-number-alignment=center]}
	*{2}{S[table-format=4.0,table-auto-round,table-number-alignment=center]}
	*{2}{S[table-auto-round]}
	*{8}{S[table-auto-round]}
	@{}
}
\caption{Percent gap closed by instance for GMICs (G), VPCs (V), both VPCs and GMICs used together, and the bound implied by the partial branch-and-bound tree with 64 leaf nodes (DB).
Also shown are the sizes of the instances, the number of cuts added, and the percent gap closed by \Gurobi{} at the root (after one round (GurF) and after the last round (GurL)). 
}
\label{app:tab:gap-closed}\\
\toprule
& & & \multicolumn{2}{c}{\# cuts} & \multicolumn{9}{c}{\% gap closed} \\
  \cmidrule(l){4-5} \cmidrule(l){6-14}
{Instance} & \multicolumn{1}{r}{Rows} & \multicolumn{1}{r}{Cols} & \multicolumn{1}{r}{G} & \multicolumn{1}{r}{V} & {G} 
	& {DB} & {V} & {max(G,V)} & {V+G} & {GurF} & {V+GurF} & {GurL} & {V+GurL} \\
\midrule
\endfirsthead
\toprule
& & & \multicolumn{2}{c}{\# cuts} & \multicolumn{9}{c}{\% gap closed} \\
  \cmidrule(l){4-5} \cmidrule(l){6-14}
{Instance} & \multicolumn{1}{r}{Rows} & \multicolumn{1}{r}{Cols} & \multicolumn{1}{r}{G} & \multicolumn{1}{r}{V} & {G} 
	& {DB} & {V} & {max(G,V)} & {V+G} & {GurF} & {V+GurF} & {GurL} & {V+GurL} \\
\midrule
\endhead
10teams & 210 & 1600 & 153 & 16 & 100.000000 & 0.000000 & 0.000000 & 100.000000 & 100.000000 & 93.877551 & 93.877551 & 100.000000 & 100.000000 \\
23588 & 137 & 237 & 75 & 75 & 5.772831 & 72.182376 & 71.614817 & 71.614817 & 71.624582 & 15.591428 & 68.936554 & 24.381330 & 71.775864 \\
30n20b8 & 387 & 4191 & 184 & 190 & 11.099235 & 1.564622 & 0.033334 & 11.099235 & 11.099235 & 0.926468 & 1.331188 & 16.910036 & 16.300723 \\
50v-10 & 233 & 2013 & 29 & 29 & 45.753596 & 18.008191 & 11.184101 & 45.753596 & 45.823184 & 49.858925 & 50.152252 & 73.169318 & 74.444520 \\
a1c1s1 & 1876 & 2489 & 154 & 3 & 25.105518 & 4.895611 & 1.094331 & 25.105518 & 25.382401 & 45.584474 & 46.254097 & 88.806373 & 88.730339 \\
a2c1s1 & 1856 & 2459 & 157 & 18 & 24.372660 & 3.337449 & 0.461882 & 24.372660 & 24.372660 & 42.797493 & 44.248500 & 91.529658 & 92.019454 \\
aflow30a & 449 & 812 & 25 & 25 & 16.535257 & 17.165164 & 16.826909 & 16.826909 & 19.285755 & 29.601743 & 29.630046 & 64.901464 & 65.248266 \\
aflow40b & 1401 & 2687 & 34 & 36 & 11.998511 & 14.593600 & 13.335246 & 13.335246 & 15.418911 & 30.750832 & 32.147917 & 56.327792 & 58.695585 \\
\bottomrule
\end{tabularx}
} %

{\clearpage\footnotesize
\sisetup{
    table-alignment-mode = format,
    table-number-alignment = center,
    table-format = 4.2,
}
\setlength\LTleft{-24.5971pt}%
\setlength{\tabcolsep}{5pt} %
\begin{tabularx}{1.2\textwidth}{
		@{}l %
          *{1}{S[table-format=2.0,table-auto-round,table-number-alignment=center]} %
          *{1}{S[table-format=3.0,table-auto-round,table-number-alignment=center]} %
          *{1}{S[table-format=4.2,table-auto-round]} %
          *{1}{S[table-format=4.2,table-auto-round]} %
          *{1}{S[table-format=4.2,table-auto-round]} %
          *{2}{S[table-format=8.0,table-auto-round,table-number-alignment=center]} %
          @{}
}
\caption{%
            Time (in seconds) and number nodes taken to solve each instance,
            for the disjunction size with best solving time with VPCs per instance.
            The table is sorted by column 5 (``V'' under ``Time (s)'').
}
\label{app:tab:bb}\\
\toprule
\multicolumn{2}{r}{} & {\# cuts} & \multicolumn{3}{c}{Time (s)} & \multicolumn{2}{c}{{Nodes (\#)}} \\
\cmidrule{3-3} \cmidrule(l){4-6} \cmidrule(l){7-8}
{Instance} & {\# terms} & {V} & {Gur} & {V} & {Gen} & {Gur} & {V} \\
\midrule
\endfirsthead
\toprule
\multicolumn{2}{r}{} & {\# cuts} & \multicolumn{3}{c}{Time (s)} & \multicolumn{2}{c}{{Nodes (\#)}} \\
\cmidrule{3-3} \cmidrule(l){4-6} \cmidrule(l){7-8}
{Instance} & {\# terms} & {V} & {Gur} & {V} & {Gen} & {Gur} & {V} \\
\midrule
\endhead
\endfoot
\bottomrule
\endlastfoot
neos-796608 & 16.000000 & 2.000000 & 0.000778 & 0.000763 & 0.020000 & 1.000000 & 1.000000 \\
neos-501453 & 64.000000 & 1.000000 & 0.001200 & 0.001162 & 0.020000 & 1.000000 & 1.000000 \\
gt2 & 4.000000 & 3.000000 & 0.003071 & 0.003348 & 0.010000 & 1.000000 & 1.000000 \\
vpm1 & 8.000000 & 5.000000 & 0.004914 & 0.004515 & 0.020000 & 1.000000 & 1.000000 \\
set1cl & 64.000000 & 3.000000 & 0.004707 & 0.004676 & 0.350000 & 1.000000 & 1.000000 \\
pipex & 4.000000 & 6.000000 & 0.015418 & 0.009907 & 0.000000 & 7.142857 & 1.000000 \\
nexp-50-20-1-1 & 16.000000 & 2.000000 & 0.011774 & 0.010048 & 0.080000 & 1.285714 & 1.000000 \\
p0548 & 64.000000 & 3.000000 & 0.011057 & 0.010674 & 0.200000 & 1.000000 & 1.000000 \\
\end{tabularx}
} %

\section{Alternative cut-generating sets and point-ray collections}
\label{sec:splits-and-crosses}

In this section, we briefly mention experiments with other choices for cut-generating sets (instead of those derived from partial branch-and-bound trees) and with potential refinements of the simple point-ray collection used for all of the previous results.
\textbf{Importantly, we report results from \emph{preliminary} experiments, with an old version of the code. Thus, the numbers in these tables will not correspond to the other experiments in the paper, even though many of the instances are the same.}

\subsection{Gap closed using ``multiple'' split and cross disjunctions}
Instead of using one large, multiterm disjunction, as we do above, the typical approach in the literature is to generate cuts from the union of several shallower disjunctions.
We report results on preliminary computational experiments to assess the strength of VPCs obtained from multiple split disjunctions or multiple $2$-branch (cross) disjunctions.
An alternative that we do not test but merits exploration in the future is that of several partial branch-and-bound trees produced from different branching strategies.

Let $\sigma \defas \{j \in \intvars \suchthat \lpopt_j \notin I\}$ be the set of indices of integer variables that take fractional values in $\lpopt$.
For each $k \in \sigma$, there is a corresponding elementary split disjunction $(x_k \le \floor{\lpopt_k}) \vee (x_k \ge \ceil{\lpopt_k})$.
We generate VPCs from each of the elementary split disjunctions applied to \ref{P}.
We call these ``multiple'' split cuts.
We also report on the strength of VPCs from the nonconvex $P_I$-free set corresponding to a union of two split disjunctions from pairs of indices in $\sigma$.

We experiment on a smaller set of instances (37 in total) in order to conserve computational resources, and we report only the percent gap closed (without testing the cuts' effect on branch and bound time).
The reason for both of these restrictions is that these early experiments strongly support the use of partial trees as the cut-generating set over multiple split or cross disjunctions.
The size restriction for the experiments in this section is at most 500 rows and 500 columns, and instances from MIPLIB 2017 were not considered.
The other instance selection criteria remain unchanged.
The limit on the number of cuts per split or cross disjunction is set as $\card{\sigma}$ (which is the same as the limit from each partial branch-and-bound tree).

{
\begin{table}
\centering
\caption{Summary on small instance set of average percent gap closed, average percent of GMICs and VPCs active at the optimum after adding both sets of cuts to \ref{P}, average ``cut ratio'' between number of VPCs and number of GMICs, and the shifted geometric mean (by 60) for generating VPCs for each partial tree size as well as on runs with multiple split and cross disjunctions.}
\label{tab:summary-splits-and-crosses}
\begin{adjustbox}{max width=1\textwidth}
\begin{tabular}{@{}l*{8}{S[table-format=2.2,table-auto-round,table-number-alignment=center]}S[table-format=3.2,table-auto-round,table-number-alignment=center]@{}}
\toprule
 	& 
	 {\makecell[c]{{G}\\ \mbox{}}} & 
	 {\makecell[c]{{V+G}\\ {(2)}}} &
	 {\makecell[c]{{V+G}\\ {(4)}}} &
	 {\makecell[c]{{V+G}\\ {(8)}}} &
	 {\makecell[c]{{V+G}\\ {(16)}}} &
	 {\makecell[c]{{V+G}\\ {(32)}}} &
	 {\makecell[c]{{V+G}\\ {(46)}}} &
	 {\makecell[c]{{V+G}\\ {(splits)}}} &
	 {\makecell[c]{{V+G}\\ {(crosses)}}} \\
\midrule
\% gap closed & 16.33 & 18.48 & 19.55 & 21.73 & 25.52 & 30.74 & 35.26 & 21.79 & 26.92 \\
\% active GMIC & &  50.40 & 49.68 & 45.01 & 40.11 & 35.37 & 36.90 & 45.65 & 35.82 \\
\% active VPC & & 22.09 & 22.59 & 30.57 & 29.81 & 29.26 & 27.90 & 14.19 & 10.45 \\
Cut ratio &  & 0.63 & 0.64 & 0.66 & 0.69 & 0.71 & 0.72 & 14.27 & 108.23 \\
Time (gmean) &  & 0.16 & 0.54 & 2.34 & 7.24 & 13.16 & 14.43 & 7.38 & 53.40 \\
\bottomrule
\end{tabular}
\end{adjustbox}
\end{table} %
}

Table~\ref{tab:summary-splits-and-crosses} has columns for ``G'' (GMICs), ``V+G ($\ell$)'' for $\ell \in \{2,4,8,16,32,64\}$ (VPCs used together with GMICs from a partial branch-and-bound tree with $\ell$ leaf nodes), and ``V+G (splits)'' and ``V+G (crosses)'' (values corresponding to using splits and crosses, respectively).
The rows give the average percent gap closed, the average percent of active GMICs and VPCs at the post-cut optimum, the ratio between the number of VPCs and number of GMICs, and the geometric mean (with a shift of 60) of the time needed to generate cuts (including the time to generate the partial trees and set up the point-ray collections).

As Table~\ref{tab:summary-splits-and-crosses} shows, 
the gap closed by VPCs from multiple split disjunctions is comparable to that from using a partial tree with 8 leaves,
while multiple cross disjunctions yield a gap closed similar to that from partial trees with 16 leaves.
However, when using splits and crosses, the number of VPCs is considerably larger than the number of GMICs,
and cut generation time is also on average much greater.
This data supports our conclusion that using partial branch-and-bound trees to generate disjunctions for our procedure is preferable to using multiple split or cross disjunctions.

\subsection{Tightening the \texorpdfstring{$\mathcal{V}$}{V}-polyhedral relaxation}
\label{sec:tightening}

We have seen in Figures~\ref{fig:non-facet} and \ref{fig:non-facet2} that using the relaxations $C^t$ for each term $t \in \mathcal{T}$ can limit the set of cuts that can be generated.
A natural question to consider is whether a different relaxation would lead to stronger cuts.

One approach, which we have not tested computationally, involves refining the relaxations of each disjunctive term to some $\tilde{C}^t \subseteq C^t$ for each $t \in \mathcal{T}$ such that the following condition is satisfied for all $t, t' \in \mathcal{T}$, $t \ne t'$:
  \[
    \tilde{C}^t \cap \{x \in \R^n \suchthat D^{t'} x \ge D^{t'}_0\} = \emptyset. %
  \]
This would avoid the type of problem shown in Figure~\ref{fig:non-facet2}.
The essential idea would involve activating hyperplanes, but the process is generally made simpler by the fact that the disjunctive inequalities in practice all take on a simple form (each is a bound on a variable).

A different idea is to keep a simple cone as the relaxation for each term of the disjunction, but to use a different cobasis as the origin.
Specifically, we can apply the VPC procedure based on Proposition~\ref{prop:VPC1}, not on $p^t$, but on a neighbor of $p^t$ obtained by pivoting along any edge of $P^t$.
We tested this procedure on VPCs generated from the set of elementary split disjunctions.
The results were negative, in the sense that only a small additional percent gap was closed, whereas the extra computational expense involved was significant.
Our interpretation of this outcome is that the VPCs from simple cones contain the vast majority of the cuts 
that affect the objective function value and are obtainable from each elementary split.
This was in fact a primary motivation for pursuing more complicated disjunctions for cut generation.

\section{Complete tables for experiments with other cut-generating sets}

The tables in this section show that, overall, using partial branch-and-bound trees leads to comparable or better percent gap closed than generating cuts from multiple split or cross disjunctions (Table~\ref{app:tab:gap-closed-splits-and-crosses}), in less time (Table~\ref{app:tab:time-splits-and-crosses}) and with fewer cuts (Table~\ref{app:tab:num-cuts-splits-and-crosses}).
A summary of these tables appeared in Table~\ref{tab:summary-splits-and-crosses}.

The columns of Table~\ref{app:tab:gap-closed-splits-and-crosses} give the following information for each instance.
Columns~2 and 3 give the dimensions of the instance. 
Column~4 gives the number of GMICs generated 
(one for each elementary split on an integer variable fractional at $\lpopt$).
Column~5 contains the number of VPCs generated, 
while the next column (6) specifies the number of cuts that are active, 
i.e., tight, at the optimum of the LP after adding the cuts. 
Finally, columns~7 and 8 give the percentage of the integrality gap closed by the GMICs 
and the VPCs. 
The last column (9) gives the time used for generating the cuts.

\begin{table}
\centering
\caption{Comparison of percent gap closed by VPCs from partial branch-and-bound trees to using multiple split or cross disjunctions.}
\label{app:tab:gap-closed-splits-and-crosses}
\begin{adjustbox}{max width=1\textwidth}
\centering
\begin{tabular}{@{}l*{10}{S[table-format=2.2,table-auto-round,table-number-alignment=center]}@{}}
\toprule
	 {\makecell[l]{{Instance}\\ \mbox{}}} & 
	 {\makecell[c]{{G}\\ \mbox{}}} & 
	 {\makecell[c]{{V+G}\\ {(2)}}} &
	 {\makecell[c]{{V+G}\\ {(4)}}} &
	 {\makecell[c]{{V+G}\\ {(8)}}} &
	 {\makecell[c]{{V+G}\\ {(16)}}} &
	 {\makecell[c]{{V+G}\\ {(32)}}} &
	 {\makecell[c]{{V+G}\\ {(46)}}} &
	 {\makecell[c]{{V+G}\\ {(best)}}} &
	 {\makecell[c]{{V+G}\\ {(splits)}}} &
	 {\makecell[c]{{V+G}\\ {(crosses)}}} \\
\midrule
23588 & 5.8 & 17.3 & 28.7 & 35.7 & 47.8 & 60.9 & 71.8 & 71.8 & 28.7 & 39.8 \\
bell3a & 37.0 & 37.0 & 37.0 & 37.0 & 43.6 & 43.6 & 43.6 & 43.6 & 37.0 & 43.9 \\
bell3b & 31.2 & 84.1 & 84.1 & 84.1 & 84.5 & 84.5 & 84.7 & 84.7 & 51.4 & 84.3 \\
bell4 & 25.6 & 25.6 & 25.8 & 26.1 & 26.0 & 26.1 & 26.4 & 26.4 & 25.6 & 29.9 \\
bell5 & 13.8 & 13.8 & 13.8 & 25.7 & 25.6 & 77.8 & 62.9 & 77.8 & 14.8 & 22.7 \\
blend2 & 5.5 & 5.5 & 7.5 & 12.1 & 13.1 & 15.5 & 23.0 & 23.0 & 7.1 & 8.9 \\
bm23 & 16.8 & 18.0 & 19.8 & 19.8 & 39.5 & 56.2 & 70.9 & 70.9 & 18.0 & 20.2 \\
glass4 & 0.0 & 0.0 & 0.0 & 0.0 & 0.0 & 0.0 & 0.0 & 0.0 & 0.0 & 0.0 \\
go19 & 2.0 & 2.5 & 3.7 & 5.6 & 9.8 & 13.1 & 2.0 & 13.1 & 8.5 & 5.8 \\
gt2 & 87.1 & 87.1 & 87.1 & 87.1 & 87.1 & 87.1 & 87.1 & 87.1 & 87.1 & 87.1 \\
k16x240 & 11.4 & 11.4 & 12.9 & 12.8 & 16.9 & 17.7 & 22.5 & 22.5 & 13.5 & 16.4 \\
lseu & 5.8 & 6.1 & 6.1 & 6.1 & 10.1 & 11.4 & 19.2 & 19.2 & 6.2 & 9.1 \\
mas074 & 6.7 & 6.7 & 7.2 & 7.9 & 9.0 & 11.1 & 13.4 & 13.4 & 6.9 & 7.2 \\
mas076 & 6.4 & 6.4 & 6.5 & 6.4 & 6.8 & 8.9 & 13.0 & 13.0 & 6.4 & 6.5 \\
mas284 & 0.9 & 1.8 & 8.7 & 12.4 & 15.6 & 25.3 & 33.8 & 33.8 & 3.3 & 10.6 \\
mik-250-1-100-1 & 53.5 & 53.5 & 53.5 & 53.5 & 53.5 & 53.5 & 53.5 & 53.5 & 53.5 & 53.5 \\
misc03 & 0.0 & 0.0 & 0.0 & 5.4 & 10.5 & 17.4 & 44.3 & 44.3 & 0.0 & 4.1 \\
misc07 & 0.7 & 0.7 & 0.7 & 0.7 & 0.9 & 1.9 & 5.5 & 5.5 & 0.7 & 0.7 \\
mod008 & 24.4 & 24.4 & 24.4 & 25.8 & 25.5 & 26.7 & 27.8 & 27.8 & 24.4 & 24.4 \\
mod013 & 5.9 & 9.0 & 9.9 & 20.6 & 21.1 & 36.9 & 47.4 & 47.4 & 9.0 & 11.6 \\
modglob & 18.1 & 18.5 & 18.8 & 19.6 & 18.9 & 18.4 & 18.6 & 19.6 & 30.3 & 44.8 \\
neos-1420205 & 0.0 & 0.0 & 0.0 & 0.0 & 0.0 & 11.4 & 14.7 & 14.7 & 0.0 & 0.4 \\
neos5 & 11.1 & 11.1 & 11.1 & 18.8 & 21.9 & 29.2 & 37.5 & 37.5 & 11.2 & 15.7 \\
neos-880324 & 0.0 & 0.0 & 0.0 & 0.0 & 0.0 & 0.0 & 36.3 & 36.3 & 8.5 & 16.9 \\
p0282 & 3.2 & 3.2 & 3.2 & 3.4 & 6.4 & 7.3 & 10.6 & 10.6 & 65.2 & 77.9 \\
pipex & 35.6 & 35.6 & 35.7 & 35.7 & 35.6 & 35.6 & 36.7 & 36.7 & 35.6 & 35.6 \\
pp08aCUTS & 33.8 & 33.8 & 34.9 & 34.9 & 35.7 & 34.6 & 33.8 & 35.7 & 39.9 & 51.9 \\
pp08a & 54.5 & 54.5 & 54.5 & 54.5 & 54.5 & 54.5 & 54.5 & 54.5 & 56.1 & 70.3 \\
probportfolio & 4.6 & 4.8 & 5.4 & 5.7 & 7.0 & 8.4 & 11.3 & 11.3 & 14.7 & 13.5 \\
prod1 & 4.7 & 5.1 & 8.3 & 12.2 & 18.6 & 29.6 & 36.3 & 36.3 & 9.7 & 22.4 \\
rgn & 8.0 & 8.0 & 8.0 & 8.0 & 21.6 & 33.8 & 45.4 & 45.4 & 8.0 & 8.1 \\
roy & 4.5 & 7.1 & 10.4 & 11.3 & 15.8 & 24.3 & 30.3 & 30.3 & 7.4 & 9.3 \\
sentoy & 19.3 & 19.3 & 20.0 & 27.5 & 45.9 & 55.0 & 59.8 & 59.8 & 21.3 & 22.9 \\
stein27\_nocard & 8.3 & 13.1 & 17.9 & 29.5 & 44.4 & 50.0 & 59.3 & 59.3 & 28.5 & 36.4 \\
timtab1 & 24.1 & 24.1 & 24.1 & 24.1 & 24.1 & 24.1 & 24.1 & 24.1 & 27.2 & 34.1 \\
vpm1 & 15.7 & 15.7 & 15.7 & 15.7 & 27.6 & 26.2 & 22.8 & 27.6 & 16.5 & 17.3 \\
vpm2 & 18.2 & 18.6 & 18.2 & 18.6 & 19.0 & 19.3 & 20.1 & 20.1 & 23.8 & 31.8 \\
\midrule
Average & 16.33 & 18.48 & 19.55 & 21.73 & 25.52 & 30.74 & 35.26 & 36.17 & 21.79 & 26.92 \\ \bottomrule
\end{tabular}
\end{adjustbox}
\end{table} %

{
\small
\begin{table}
\small
\centering
\caption{Time to generate VPCs for each of the different partial branch-and-bound tree sizes and for multiple split and cross disjunctions.}
\label{app:tab:time-splits-and-crosses}
\begin{adjustbox}{max width=1\textwidth}
\centering
\begin{tabular}{@{}l
	*{2}{S[table-format=1.2,table-auto-round,table-number-alignment=center]}
	*{1}{S[table-format=2.2,table-auto-round,table-number-alignment=center]}
	*{5}{S[table-format=3.2,table-auto-round,table-number-alignment=center]}
	@{}
}
\toprule
	 {\makecell[l]{{Instance}\\ \mbox{}}} & 
	 {\makecell[c]{{V}\\ {(2)}}} &
	 {\makecell[c]{{V}\\ {(4)}}} &
	 {\makecell[c]{{V}\\ {(8)}}} &
	 {\makecell[c]{{V}\\ {(16)}}} &
	 {\makecell[c]{{V}\\ {(32)}}} &
	 {\makecell[c]{{V}\\ {(46)}}} &
	 {\makecell[c]{{V}\\ {(splits)}}} &
	 {\makecell[c]{{V}\\ {(crosses)}}} \\
\midrule
23588 & 0.5 & 2.3 & 8.8 & 23.9 & 77.5 & 224.8 & 36.2 & 912.9 \\
bell3a & 0.0 & 0.0 & 0.0 & 0.0 & 0.1 & 0.2 & 0.0 & 0.1 \\
bell3b & 0.0 & 0.0 & 0.0 & 0.1 & 0.1 & 0.2 & 0.1 & 2.4 \\
bell4 & 0.0 & 0.0 & 0.0 & 0.1 & 0.1 & 0.2 & 0.1 & 2.6 \\
bell5 & 0.0 & 0.0 & 0.0 & 0.0 & 0.1 & 0.1 & 0.1 & 0.3 \\
blend2 & 0.1 & 0.2 & 0.2 & 0.5 & 1.6 & 4.3 & 0.6 & 6.9 \\
bm23 & 0.0 & 0.0 & 0.0 & 0.1 & 0.1 & 0.3 & 0.1 & 0.2 \\
glass4 & 0.0 & 0.1 & 0.1 & 0.3 & 0.5 & 1.6 & 1.5 & 100.8 \\
go19 & 1.6 & 7.5 & 52.1 & 314.4 & 906.3 & 71.7 & 619.5 & 926.5 \\
gt2 & 0.0 & 0.0 & 0.0 & 0.2 & 0.4 & 0.9 & 0.2 & 2.9 \\
k16x240 & 0.1 & 0.1 & 0.3 & 0.7 & 1.6 & 3.5 & 0.7 & 8.8 \\
lseu & 0.0 & 0.0 & 0.0 & 0.1 & 0.2 & 0.4 & 0.1 & 1.0 \\
mas074 & 0.0 & 0.1 & 0.2 & 0.6 & 1.4 & 4.4 & 0.4 & 4.8 \\
mas076 & 0.0 & 0.1 & 0.3 & 0.8 & 1.9 & 4.2 & 0.4 & 4.8 \\
mas284 & 0.1 & 0.3 & 0.8 & 4.1 & 13.5 & 52.2 & 1.7 & 48.8 \\
mik-250-1-100-1 & 0.1 & 0.1 & 0.2 & 0.4 & 2.1 & 9.3 & 3.8 & 327.5 \\
misc03 & 0.1 & 0.1 & 0.5 & 1.1 & 3.2 & 8.7 & 0.9 & 20.1 \\
misc07 & 0.1 & 0.3 & 0.9 & 3.0 & 13.0 & 33.3 & 1.3 & 29.5 \\
mod008 & 0.1 & 0.1 & 0.2 & 0.4 & 1.2 & 2.9 & 0.2 & 1.2 \\
mod013 & 0.0 & 0.0 & 0.0 & 0.1 & 0.2 & 0.4 & 0.1 & 0.2 \\
modglob & 0.1 & 0.2 & 0.4 & 1.2 & 1.9 & 3.4 & 0.8 & 18.6 \\
neos-1420205 & 0.2 & 0.3 & 0.3 & 0.9 & 4.1 & 7.1 & 2.4 & 170.8 \\
neos5 & 0.1 & 0.2 & 0.4 & 0.7 & 1.8 & 4.9 & 1.4 & 114.6 \\
neos-880324 & 0.1 & 0.6 & 0.4 & 1.3 & 15.3 & 20.1 & 2.2 & 267.1 \\
p0282 & 0.0 & 0.1 & 0.1 & 0.4 & 0.7 & 1.6 & 0.5 & 10.9 \\
pipex & 0.0 & 0.0 & 0.0 & 0.1 & 0.1 & 0.2 & 0.0 & 0.2 \\
pp08aCUTS & 0.1 & 0.1 & 0.2 & 0.8 & 2.9 & 17.6 & 3.7 & 184.6 \\
pp08a & 0.0 & 0.0 & 0.1 & 0.2 & 0.4 & 1.1 & 0.9 & 39.1 \\
probportfolio & 2.3 & 7.4 & 41.2 & 264.7 & 682.7 & 903.3 & 88.5 & 907.2 \\
prod1 & 0.1 & 0.3 & 1.2 & 2.7 & 6.1 & 11.8 & 2.0 & 218.3 \\
rgn & 0.0 & 0.1 & 0.2 & 0.3 & 1.4 & 2.5 & 0.4 & 7.4 \\
roy & 0.0 & 0.0 & 0.1 & 0.2 & 0.3 & 0.6 & 0.2 & 1.4 \\
sentoy & 0.0 & 0.0 & 0.1 & 0.2 & 0.4 & 0.9 & 0.1 & 0.7 \\
stein27\_nocard & 0.0 & 0.1 & 0.1 & 0.2 & 0.4 & 0.5 & 0.4 & 19.2 \\
timtab1 & 0.0 & 0.0 & 0.1 & 0.2 & 0.3 & 0.7 & 2.3 & 339.6 \\
vpm1 & 0.0 & 0.0 & 0.1 & 0.2 & 0.3 & 0.7 & 0.1 & 1.1 \\
vpm2 & 0.0 & 0.0 & 0.1 & 0.2 & 0.4 & 0.8 & 0.5 & 8.9 \\
\midrule
Gmean & 0.16 & 0.54 & 2.34 & 7.24 & 13.16 & 14.43 & 7.38 & 53.40 \\ \bottomrule
\end{tabular}
\end{adjustbox}
\end{table} %
}

\begin{table}
\centering
\caption{Number of rows, columns, GMICs, and VPCs for small instances used to test multiple split and cross disjunctions. The last row gives the average ratio of number of VPCs as a fraction of the number of GMICs.}
\label{app:tab:num-cuts-splits-and-crosses}
\begin{adjustbox}{max width=1\textwidth}
\setlength\LTleft{-6.86802 pt}%
\centering
\begin{tabular}{@{}%
	l*{9}{S[table-format=3.0,table-auto-round,table-number-alignment=center]}
	S[table-format=5.0,table-auto-round,table-number-alignment=center]
	S[table-format=5.0,table-auto-round,table-number-alignment=center]
	@{}
}
\toprule
& & & \multicolumn{9}{c}{\# cuts} \\
\cmidrule{4-12}
Instance & {Rows} & {Cols} & {G} &
	 {\makecell[c]{{V}\\ {(2)}}} &
	 {\makecell[c]{{V}\\ {(4)}}} &
	 {\makecell[c]{{V}\\ {(8)}}} &
	 {\makecell[c]{{V}\\ {(16)}}} &
	 {\makecell[c]{{V}\\ {(32)}}} &
	 {\makecell[c]{{V}\\ {(46)}}} &
	 {\makecell[c]{{V}\\ {(splits)}}} &
	 {\makecell[c]{{V}\\ {(crosses)}}} \\
\midrule
23588 & 137 & 237 & 75 & 36 & 75 & 75 & 75 & 75 & 75 & 4890 & 32920 \\
bell3a & 63 & 82 & 7 & 1 & 5 & 7 & 4 & 5 & 7 & 4 & 19 \\
bell3b & 73 & 91 & 24 & 3 & 3 & 7 & 21 & 25 & 25 & 63 & 328 \\
bell4 & 73 & 88 & 27 & 3 & 4 & 8 & 5 & 7 & 9 & 22 & 178 \\
bell5 & 34 & 56 & 10 & 3 & 2 & 5 & 4 & 10 & 10 & 19 & 82 \\
blend2 & 154 & 302 & 13 & 13 & 13 & 13 & 13 & 13 & 13 & 117 & 840 \\
bm23 & 20 & 27 & 6 & 6 & 6 & 6 & 6 & 6 & 6 & 36 & 90 \\
glass4 & 392 & 317 & 72 & 2 & 4 & 10 & 14 & 0 & 72 & 142 & 906 \\
go19 & 361 & 361 & 357 & 97 & 228 & 357 & 357 & 175 & 0 & 35760 & 22557 \\
gt2 & 28 & 173 & 14 & 3 & 4 & 3 & 7 & 2 & 1 & 21 & 242 \\
k16x240 & 256 & 480 & 14 & 14 & 9 & 14 & 14 & 14 & 14 & 166 & 977 \\
lseu & 28 & 79 & 9 & 9 & 9 & 9 & 9 & 9 & 9 & 81 & 280 \\
mas074 & 13 & 148 & 12 & 12 & 12 & 12 & 12 & 12 & 12 & 144 & 779 \\
mas076 & 12 & 148 & 11 & 11 & 8 & 4 & 10 & 11 & 11 & 112 & 497 \\
mas284 & 68 & 148 & 20 & 20 & 20 & 20 & 6 & 9 & 4 & 387 & 3607 \\
mik-250-1-100-1 & 100 & 251 & 100 & 1 & 2 & 8 & 13 & 74 & 100 & 1 & 100 \\
misc03 & 95 & 138 & 18 & 18 & 18 & 18 & 18 & 18 & 18 & 310 & 2616 \\
misc07 & 211 & 232 & 16 & 16 & 16 & 16 & 16 & 16 & 16 & 256 & 1901 \\
mod008 & 6 & 319 & 6 & 6 & 6 & 3 & 6 & 6 & 6 & 30 & 88 \\
mod013 & 62 & 96 & 5 & 5 & 5 & 5 & 5 & 5 & 5 & 23 & 50 \\
modglob & 286 & 354 & 29 & 21 & 25 & 29 & 29 & 29 & 29 & 209 & 1049 \\
neos-1420205 & 341 & 231 & 44 & 40 & 44 & 0 & 0 & 2 & 1 & 613 & 9756 \\
neos5 & 63 & 63 & 35 & 25 & 35 & 2 & 1 & 1 & 1 & 719 & 8989 \\
neos-880324 & 182 & 135 & 45 & 10 & 7 & 0 & 0 & 0 & 11 & 543 & 17738 \\
p0282 & 160 & 200 & 24 & 13 & 10 & 24 & 24 & 24 & 24 & 181 & 1697 \\
pipex & 25 & 48 & 6 & 6 & 6 & 6 & 6 & 6 & 6 & 19 & 42 \\
pp08aCUTS & 239 & 235 & 46 & 6 & 7 & 3 & 5 & 11 & 46 & 924 & 12766 \\
pp08a & 133 & 234 & 53 & 4 & 5 & 4 & 4 & 2 & 2 & 246 & 1145 \\
probportfolio & 302 & 320 & 105 & 105 & 105 & 105 & 105 & 105 & 48 & 6157 & 9432 \\
prod1 & 75 & 117 & 40 & 18 & 40 & 40 & 40 & 40 & 40 & 1184 & 30508 \\
rgn & 24 & 180 & 18 & 16 & 18 & 18 & 18 & 18 & 3 & 222 & 1606 \\
roy & 147 & 139 & 11 & 7 & 6 & 11 & 11 & 11 & 11 & 104 & 368 \\
sentoy & 30 & 60 & 8 & 8 & 8 & 8 & 8 & 8 & 8 & 64 & 224 \\
stein27\_nocard & 117 & 27 & 27 & 7 & 11 & 27 & 1 & 2 & 1 & 163 & 6102 \\
timtab1 & 165 & 365 & 128 & 5 & 2 & 3 & 4 & 2 & 2 & 557 & 13566 \\
vpm1 & 128 & 188 & 11 & 11 & 0 & 5 & 11 & 9 & 11 & 56 & 104 \\
vpm2 & 127 & 187 & 24 & 24 & 7 & 13 & 24 & 24 & 24 & 247 & 1209 \\ 
\midrule
Avg (cut ratio) & & & & 0.63 & 0.64 & 0.66 & 0.69 & 0.71 & 0.72 & 14.27 & 108.23 \\
\bottomrule
\end{tabular}
\end{adjustbox}
\end{table} %

\end{document}